%% file: topK0.tex
\documentclass[12pt]{amsart}

\input{preamble/packages.tex}
\input{preamble/macros.tex}

\input{preamble/paperdata.tex}

\begin{document}

\input{preamble/abstract.tex}

\maketitle


\input{sections/intro.tex}

\input{sections/classK0.tex}

\input{sections/addTopK0.tex}

\input{sections/JHTopK0.tex}

\input{sections/triangTopK0.tex}

\input{sections/cblfDgAlg.tex}

\input{bibliography/bibliography.tex}

\end{document}

%% file: preamble/packages.tex
\usepackage{
	amsmath,
 	amsfonts,
  	amssymb,
 	amsthm,
	}

\usepackage{tikz-cd}
\usepackage{tikz}

\usepackage{rotating}

\usepackage{color}

\usepackage{stmaryrd}

\usepackage{hyperref}
\usepackage[capitalise]{cleveref}

\usepackage{a4wide}

\usepackage{mathabx}


%% file: preamble/macros.tex
\theoremstyle{plain}
\newtheorem{thm}{Theorem}[section]
\newtheorem{cor}[thm]{Corollary}
\newtheorem{lem}[thm]{Lemma}
\newtheorem{prop}[thm]{Proposition}

\theoremstyle{definition}
\newtheorem{rem}[thm]{Remark}

\newtheorem{exe}[thm]{Example}

\theoremstyle{definition}
\newtheorem{defn}[thm]{Definition}

\newenvironment{citethm}[1]{%
	\renewcommand\thethm{\ref{#1}}\thm}{\endthm\addtocounter{thm}{-1}}
	
\newenvironment{citecor}[1]{%
	\renewcommand\thethm{\ref{#1}}\cor}{\endthm\addtocounter{thm}{-1}}
	
\newenvironment{citeprop}[1]{%
	\renewcommand\thethm{\ref{#1}}\prop}{\endthm\addtocounter{thm}{-1}}

\def\makeautorefname#1#2{\expandafter\def\csname#1autorefname\endcsname{#2}}
\makeautorefname{lem}{Lemma}%
\makeautorefname{prop}{Proposition}%
\makeautorefname{rem}{Remark}%
\makeautorefname{section}{Section}%

\newcommand{\cA}{\mathcal{A}}

\newcommand{\cC}{\mathcal{C}}
\newcommand{\cD}{\mathcal{D}}

\newcommand{\cG}{\mathcal{G}}

\newcommand{\cT}{\mathcal{T}}

\newcommand{\bC}{\mathbb{C}}

\newcommand{\bN}{\mathbb{N}}

\newcommand{\bP}{\mathbb{P}}

\newcommand{\bR}{\mathbb{R}}

\newcommand{\bZ}{\mathbb{Z}}

\newcommand{\bc}{{\boldsymbol{c}}}
\newcommand{\be}{{\boldsymbol{e}}}
\newcommand{\bg}{{\boldsymbol{g}}}

\newcommand{\bk}{{\boldsymbol{k}}}
\newcommand{\bx}{{\boldsymbol{x}}}
\newcommand{\bz}{{\boldsymbol{z}}}

\newcommand{\vsimeq}{\rotatebox{-90}{\(\simeq\)}}

\newcommand{\brak}[1]{\langle #1\rangle}
\newcommand{\pp}[1]{(\!( #1 )\!)}

\newcommand{\id}{\text{Id}}

\DeclareMathOperator{\Hom}{Hom}
\DeclareMathOperator{\End}{End}
\DeclareMathOperator{\Ext}{Ext}

\DeclareMathOperator{\HOM}{HOM}
\DeclareMathOperator{\RHOM}{RHOM}

\DeclareMathOperator{\Lderiv}{L}
\newcommand{\Lotimes}{\otimes^{\Lderiv}}

\DeclareMathOperator{\Image}{im}
\DeclareMathOperator{\cok}{cok}

\DeclareMathOperator{\gdim}{gdim}

\DeclareMathOperator{\mcolim}{MColim}
\DeclareMathOperator{\mlim}{MLim}
\DeclareMathOperator{\colim}{colim}
\DeclareMathOperator{\cone}{Cone}

\newcommand{\bKO}{\boldsymbol{K}_0}
\newcommand{\bGO}{\boldsymbol{G}_0}

\DeclareMathOperator{\amod}{\mathrm{-}mod}
\DeclareMathOperator{\dgmod}{\mathrm{-}dgmod}
\DeclareMathOperator{\pmodfg}{\mathrm{-}pmod_{fg}}
\DeclareMathOperator{\pmodlfg}{\mathrm{-}pmod_{cblfg}}
\DeclareMathOperator{\modlf}{\mathrm{-}mod_{cblf}}

\DeclareMathOperator{\gmod}{\mathrm{-}gmod}

\DeclareMathOperator{\vect}{\mathrm{-}vect}

\DeclareMathOperator{\br}{\mathbf p} 


%% file: preamble/paperdata.tex
\title{Asymptotic Grothendieck groups and c.b.l.f. positive dg-algebras}
\author{Gr\'egoire Naisse}
\address{Max-Planck Institute for Mathematics\\
 Vivatsgasse 7 \\ 
53111 Bonn\\ 
Germany}
\email{gregoire.naisse@gmail.com}


%% file: preamble/abstract.tex

\begin{abstract} 
We introduce the notion of an asymptotic Grothendieck group for a subcategory of an abelian or a triangulated categories that is both AB4 and AB4*. 
Such an asymptotic Grothendieck group consists in quotienting the usual Grothendieck group by relations obtained from controlled infinite iterated extensions. 
We study when the asymptotic Grothendieck group of the heart of a triangulated category with a t-structure is isomorphic to the asymptotic Grothendieck group of the triangulated category itself. 
We also explain a connection with the notion of topological Grothendieck group from Achar--Stroppel.
Finally, we compute the asymptotic Grothendieck group of certain multigraded positive (dg-)algebras having a cone bounded, locally finite dimension. 
\end{abstract}


%% file: sections/intro.tex

\section{Introduction}\label{sec:intro}

Categorification is a trendy research area in mathematics nowadays. 
In short, the philosophy behind \emph{categorification} is to lift classical objects to the level of categories, turning set elements
into objects, functions into functors, equalities into isomorphisms, and revealing a whole new layer of structure achieving these isomorphisms. 
It is particularly relevant when one categorifies representations, turning a vector space into a category and the action (of a group or an algebra) into endofunctors. 
One usually refer to this procedure as \emph{higher representation theory} (see for example \cite{chuangrouquier,  FKS, lauda, soergel}). Categorification also has applications in low-dimensional topology (see for example \cite{khovanov, webster}) and deep connections with algebraic geometry (see \cite{catgeom} for a nice survey). 
While most results live in an integral setting,  it appears to be not enough to handle all situations, such as in the categorification of the Jones-Wenzl projectors~\cite{frenkelstroppelsussan}, or in the categorification of Verma modules~\cite{naissevaz1,naissevaz2,naissevaz3}, which is related by~\cite{naissevaztriply} to Khovanov--Rozansky triply--graded link homology~\cite{KRhomology}.

\smallskip

In the first case, it is necessary to categorify fractions of the form
\begin{equation}\label{eq:nfraction}
\frac{1}{1+q^2+ \cdots + q^{2n-2}},
\end{equation}
where $q$ is a formal variable. 
The polynomial $1+q^2 + \cdots + q^{2n-2}$ is the graded dimension of $R := \Bbbk[x]/(x^n)$, $\deg(x) = 2$, the cohomology ring of $\bC\bP^{n-1}$, as a graded $\Bbbk$-module. 
Thus, in the Grothendieck group $K_0(\Bbbk\vect)$ of $\bZ$-graded $\Bbbk$-module, we have
\[
[R] = 1+q^2 + \cdots + q^{2n-2}[\Bbbk].
\]
Moreover,  $\Bbbk$ admits a free resolution over $R$ with (graded) ranks matching \cref{eq:nfraction} expanded as a Laurent series in $\Bbbk\llbracket q \rrbracket [q^{-1}]$. 
Therefore, the authors of \cite{frenkelstroppelsussan} suggested to categorify the fraction \cref{eq:nfraction} by the Ext-algebra $\Ext^*_R(\Bbbk, \Bbbk)$. 
However, we usually cannot say that in  $K_0(\Bbbk\vect)$ we have
\[
[\Bbbk] = \frac{1}{1+q^2+ \cdots + q^{2n-2}}[R],
\]
since it would require to make sense of a convergent series, as the projective resolution is infinite.  
In~\cite{acharstroppel}, Achar and Stroppel developed a framework to handle such situations. 
They start with a nice graded, finite length abelian category with enough projectives, and explain how to construct a derived category where (minimal) projective resolutions of the objects exist.  Then, they develop a notion of \emph{topological Grothendieck group}, by modding out relations coming from these (potentially infinite) projective resolutions. 
We can sum up their framework as objects are finite, but can admit infinite resolutions. 
This is the notion used in~\cite{frenkelstroppelsussan}.

\smallskip

In the case of the categorification of Verma modules~\cite{naissevaz1,naissevaz2,naissevaz3}, one needs to categorify the rational fraction
\begin{equation}\label{eq:lambdafraction}
\frac{1-\lambda^2}{1-q^2},
\end{equation}
which appears as a structure constant. The solution used in~\cite{naissevaz1,naissevaz2} was to interpret $\lambda$ as a second grading, the minus sign on the numerator as a parity degree (thus working in a  ``super'' setting), and the denominator as the formal Laurent series $(1+q^2+q^4+ \cdots)$. 
From that, the author developed with Vaz in~\cite[Appendix]{naissevaz1} notions of Krull--Schmidt (resp. Jordan--H\"older) categories for infinite direct sums (resp. infinite composition series), controlled in a certain way. 
These categories also come with some `topological' versions of the Grothendieck groups. 
Here, the object are not finite anymore, but only locally finite with respects to the grading. However, the framework in~\cite[Appendix]{naissevaz1} does not handle infinite projective resolutions, or more generally derived categories. 

\smallskip

In this paper, we propose a framework bringing both the topological Grothendieck groups of~\cite{acharstroppel} and of~\cite{naissevaz1} under the same roof, and generalizing it even further. We call this generalization an \emph{asymptotic Grothendieck group}. Such a construction is necessary to handle properly the (de)categorification of (parabolic) Verma modules for all quantum Kac--Moody algebras, as done in~\cite{naissevaz3}. It also yields a more natural categorification of $\mathfrak{sl}_2$-Verma modules from~\cite{naissevaz1,naissevaz2} in the world of homological algebra rather than in the super-setting, with the minus sign on the numerator of \cref{eq:lambdafraction} induced by a difference in homological degree. 
Consequently, asymptotic Grothendieck groups are also used in other constructions, such as  \cite{naissevaztriply}, \cite{twoblob} and \cite{tensorvermas}, which all connect categorification of Verma modules to low-dimensional topology. The results in~\cite{frenkelstroppelsussan} can also be rewritten in our setting.

\subsection{Main example and goal of the paper}\label{sec:mainexintro}

As a main example for the paper, we consider the following algebra, whose definition is motivated by~\cite{KRhomology} and~\cite{naissevaz1}. Consider the $\Bbbk$-algebra $R := \Bbbk[x, \omega]/(\omega^2)$.  We endow it with a $\bZ \times \bZ^2$ grading, declaring that $\deg(x) := (0,2,0)$ and $\deg(\omega) := (1,0,2)$. We interpret the first $\bZ$-grading as homological, while we call each of the components of the $\bZ^2$-grading the $q$- and the $\lambda$-grading respectively. We equip $R$ with a trivial differential, turning it into a $\bZ^2$-graded dg-algebra $(R,0)$, with the differential grading given by the homological one. 
Note that the graded Euler characteristic of $R$ is
\[
\chi(R) = (1 - \lambda^2)(1+q^2+q^4+ \cdots) = \frac{1-\lambda^2}{1-q^2}.
\]
Furthermore, $R$ admits a unique (up to isomorphism and grading shift) projective module, given by $R$ itself, and a unique (up to isomorphism and grading shift) simple module $L := \Bbbk$. They both lift to $\bZ^2$-graded dg-modules when equipped with a trivial differential. Moreover, we obtain a filtration
\begin{equation}\label{eq:filtintroex}
\cdots
\hookrightarrow x^2 R 
\hookrightarrow (x^2 R + x \omega R)
\hookrightarrow xR
\hookrightarrow (xR + \omega R)
\hookrightarrow R,
\end{equation}
with quotients
\begin{align*}
\frac{R}{(xR + \omega R)} &\cong L, 
&
\frac{(xR + \omega R)}{x R} &\cong \lambda^2 L [1],
&
\frac{xR}{(x^2 R + x \omega R)} &\cong q^2 L,
&
\dots
\end{align*}
and a resolution
\begin{equation}\label{eq:resintroex}
\cdots
\rightarrow
  q^2 \lambda^4 R [2] \oplus \lambda^6 R [3]  \xrightarrow{ \begin{pmatrix} \omega & x \\ 0 & \omega \end{pmatrix} }
q^2 \lambda^2 R [1] \oplus \lambda^4 R [2]  \xrightarrow{ \begin{pmatrix} \omega & -x \\ 0 & \omega \end{pmatrix} } q^2 R \oplus \lambda^2 R [1] \xrightarrow{\begin{pmatrix} x  \\ \omega \end{pmatrix}} R \rightarrow L,
\end{equation}
where $q$ and $\lambda$ are grading shifts in the $\bZ^2$-grading and $[1]$ in the homological grading. 
Our goal is to construct a suitable derived category of $(R,0)$ such that its asymptotic Grothendieck group is a free $\bZ\pp{q,\lambda}$-module, and generated either by the class of $[R]$ or $[L]$, with formulas
\begin{align}\label{eq:changeofbasisintro}
[R] &= \frac{1-\lambda^2}{1-q^2} [L], & [L] &= \frac{1-q^2}{1-\lambda^2}[R],
\end{align}
 obtained from the categorical relations in \cref{eq:filtintroex} and in \cref{eq:resintroex}. 
 Throughout the paper, we will study different categories of modules over $R$ or its dg-version $(R,0)$, and their respective (asymptotic) Grothendieck group.

\subsection{Sketch of the construction}

Let us give an insight of what we mean by asymptotic Grothendieck group. 
Let $\cA$ be an abelian category. Given a sequence or arrows
\[
X = F_{m+1} \xrightarrow{f_m} \cdots \xrightarrow{f_2} F_2 \xrightarrow{f_1} F_1 \xrightarrow{f_0}  F_{0} = Y,
\]
one can deduce that in the (abelian) Grothendieck group $G_0(\cA)$ there is an equality
\[
[Y] - [X] = \sum_{i = 0}^m  [\cok{f_i}] - [\ker{f_i}].
\]
The idea of the asymptotic Grothendieck group is to take this phenomenon to the limit. One can extend such a composition of arrows to the left (resp. to the right)
\begin{align}
\label{eq:XislimYiscolim}
\begin{split}
X &= \lim\bigl( \cdots \xrightarrow{f_2} F_2 \xrightarrow{f_1} F_1 \xrightarrow{f_0}  F_{0}  =Y \bigr), \\
\text{(resp. }\quad  Y &= \colim \bigl( X = F_0 \xrightarrow{f_0} F_1 \xrightarrow{f_1} F_2 \xrightarrow{f_2} \cdots \bigr) \text{ )}, 
\end{split}
\end{align}
giving a filtered limit (resp. filtered colimit).
Naively, we could be tempted to ask that in this case we have
\begin{align} \label{eq:wrongdef}
\text{`` }
[X] &= [Y] +  \sum_{r \geq 0} \bigl( [\ker f_r] - [\cok f_r] \bigr).
\text{ ''}
\end{align}
 However, filtered limits (resp. filtered colimits) are in general not exact. 
 This is problematic since for example given a short exact sequence of filtered limits
 \begin{equation} \label{eq:seslimits}
 0 \rightarrow A_\bullet \rightarrow B_\bullet \rightarrow C_\bullet \rightarrow 0,
 \end{equation}
 with $A := \lim( \cdots \rightarrow A_1 \rightarrow A_0)$ and similarly for $B$ and $C$, then \cref{eq:wrongdef} would imply that
 \[
[B] = [A] + [C],
 \]
 while $B$ could be very far from being an extension of $C$ by $A$. 
Hopefully, when $\cA$ is a full subcategory of an abelian category having countable products (resp. coproducts), it is possible to compute by \emph{how much} a filtered limit $X$ (resp. colimit $Y$) is not exact, by considering its derived limit $\tilde X$ (resp. derived colimit $\tilde Y$). In particular, for a short exact sequence of limits as \cref{eq:seslimits}, there is an exact sequence
\begin{equation}\label{eq:derivexactseq}
0 \rightarrow A \rightarrow B \rightarrow C \rightarrow \tilde A \rightarrow \tilde B \rightarrow \tilde C \rightarrow 0. 
\end{equation}
Thus, we suggest to define the asymptotic Grothendieck group $\bGO(\cA)$ by taking $G_0(\cA)$ modded out by the relations
\begin{align*}
[X] - [\tilde X] &= [Y] + \sum_{r \geq 0} \bigl( [\ker f_r] - [\cok f_r] \bigr), \\
\text{(resp. }\quad  [Y] - [\tilde Y] &= [X] + \sum_{r \geq 0} \bigl( [\cok f_r] - [\ker f_r] \bigr) \text{ )}, 
\end{align*}
whenever we are in the case of \cref{eq:XislimYiscolim}, see \cref{def:topG0} for a precise definition of $\bGO(\cA)$. 
In particular, this definition implies that given \cref{eq:seslimits}, then 
\[
[B] - [\tilde B] = [A] -  [\tilde A] + [C] -[\tilde C],
\]
which agrees with \cref{eq:derivexactseq}. 
We use a similar definition in the case of a triangulated category
, using \emph{Milnor limits} and \emph{Milnor colimits} instead of filtered ones, and replacing $[\cok f_r] - [\ker f_r]$ by the class of the \emph{mapping cone} $[\cone (f_r)]$, see \cref{def:toptriangulatedK0}.

\subsection{Outline of the paper}

\cref{sec:classK0} is a reminder of the basics about Grothendieck groups. In \cref{sec:triangcat} we recall the construction of the derived category of a dg-algebra, which we will use as main example for later, and we fix the conventions (in particular all our differentials are of degree $-1$). 
In \cref{sec:finitemultigradings} we explain what we mean by a multigraded category and how it influences the Grothendieck group. 
In \cref{sec:formalLaurent} we recall the construction of the ring of formal Laurent series $\Bbbk\pp{x_1, \dots, x_n}$. 

Our contributions start in \cref{sec:cbladditive} where we extend the notion of locally Krull--Schmidt category of \cite[Appendix]{naissevaz1} to the non-graded case (\cref{def:locKS}). One of the main results is:
\begin{citecor}{cor:lockrullschmidt}
Let $\cC$ be a locally Krull--Schmidt category with a full collection of pairwise non-isomorphic indecomposable objects $\{X_i\}_{i \in I}$. 
If $\cC$ admits locally finite direct sums 
of indecomposable objects, then its split Grothendieck group is a free $\bZ$-module
\[
K_0^\oplus(\cC) \cong \prod_{i \in I} \bZ \cdot [X_i],
\]
with basis $\{[X_i]\}_{i \in I}$. 
\end{citecor}
We also recall the notions of cone bounded, locally finite (c.b.l.f.) direct sums and of c.b.l. additive category from \cite[Appendix]{naissevaz1} (appearing under the name `locally finite' in the reference). We also recall the notion of c.b.l.f. dimensional $\Bbbk$-vector space, which basically means the graded dimension of the vector space is an element of $\bZ\pp{x_1,\dots,x_n}$. 

In \cref{sec:topG0} we introduce the notion of asymptotic Grothendieck group for abelian categories. We refine the notion of `local Jordan--H\"older category' from \cite[Appendix]{naissevaz1}. The main difference being that we require in this paper filtered limits and filtered colimits to be exact (thus derived limits and derived colimits to be zero) in a locally Jordan--H\"older category (\cref{def:locJH}). Then, based on the results in \cite[Appendix]{naissevaz1}, we show the following:
\begin{citethm}{thm:locallyJHG0}
Let $\cC$ be a locally Jordan--H\"older category with a full collection of pairwise non-isomorphic simple objects $\{S_i\}_{i \in I}$. If $\cC$ admits locally finite direct sums of $\{S_i\}_{i \in I}$, then its asymptotic Grothendieck group is a free $\bZ$-module
\[
\bGO(\cC) \cong \prod_{i \in I} \bZ \cdot[S_i],
\]
with basis $\{[S_i]\}_{i \in I}$.
\end{citethm}
In particular, the asymptotic Grothendieck group of a locally Jordan--H\"older category is equivalent to the topological one of \cite[Appendix]{naissevaz1}. 
This allows us to compute the asymptotic Grothendieck group for certain categories of graded modules. 
Let $R$ be a c.b.l.f. dimensional, $\bZ^n$-graded $\Bbbk$-algebra. Consider the category $R\amod$ of $\bZ^n$-graded $R$-modules with degree zero maps, and let $R\modlf$ be the full subcategory of $R\amod$ consisting of c.b.l.f. dimensional modules. 
\begin{citeprop}{prop:RposJH}
If $R$ is positive, that is $R = \bigoplus_{\bg \succeq 0} R_\bg$, and $R_0$ is semi-simple, then $R\modlf$ is strongly c.b.l. Jordan--H\"older.
\end{citeprop}

\begin{citecor}{cor:GORmodlf}
Let $R \cong \bigoplus_{i \in I} Re_i$ be as in \cref{prop:RposJH} and take a full collection $\{e_j\}_{j \in J \subset I}$ of non-equivalent idempotents. Then we have
\[
\boldsymbol G_0(R\modlf) \cong \bigoplus_{j} \bZ\pp{x_1, \dots, x_n} \cdot [S_j],
\] with $S_j := Re_j/R_{\succ 0}e_j$.
\end{citecor}

Still in \cref{sec:topG0}, we also introduce the notion of a cobaric structure on a $\bZ^n$-graded abelian category $\cC$ (\cref{def:cobaricstruct}), inspired by Achar definition~\cite{baric}. A cobaric structure is loosely speaking a way to split $\cC$ in two parts: a positive one $\cC_{\succ 0}$, stable by shifting the degree up,  and a negative one $\cC_{\preceq 0}$, stable by shifting the degree down, such that any object $X \in \cC$ is an extension of a negative object $X_{\preceq 0}$ by a positive one $X_{\succ 0}$. 
Given such a category $\cC$ with a cobaric structure $\beta$, there is a natural notion of \emph{topological Grothendieck group $\bGO(\cC,\beta)$} similar to the one introduced by Achar--Stroppel in \cite{acharstroppel}. It is given by saying that $[X] = [Y] \in \bGO(\cC,\beta)$ whenever $[(x^{\bg}X)_{\preceq 0}] = [(x^{\bg}Y)_{\preceq 0}] \in G_0(\cC)$  for all $\bg \in \bZ^n$, where $x^{\bg}$ is the degree shift by $\bg$. 
As hinted by its name, the topological Grothendieck group comes as in~\cite{acharstroppel} with a natural topology. 
 If $(\cC,\beta)$ is nice enough, then we show the topological Grothendieck group coincides with the asymptotic one:
\begin{citeprop}{prop:topGOisasympGO}
Let $\cC$ be a Mittag--Leffler strictly $\bZ^n$-graded abelian category with  a  full, non-degenerate, stable, cone bounded cobaric structure $\beta$.
There is a surjection
\[
 \bGO(\cC,\beta) \twoheadrightarrow \bGO(\cC),
\]
induced by the identity on $G_0(\cC)$. 
If  $\beta$ is locally finite, then it is an isomorphism
\[
 \bGO(\cC,\beta) \cong \bGO(\cC). 
\]
\end{citeprop}

In \cref{sec:derivedtopK0}, we introduce the notion of asymptotic Grothendieck group for a triangulated category $\cC \subset \cT$ where $\cT$ admits products and coproducts, and these preserve distinguished triangles. We investigate the case when $\cT$ is equipped with a t-structure $\tau$. We show that, when $\tau$ is nice enough, one can compute the asymptotic Grothendieck group of $\cC$ by computing the asymptotic Grothendieck group of its heart $\cC^\heartsuit$: 
\begin{citethm}{thm:eqasympKOheart}
If $\tau$ is bounded from below, non-degenerate, full and stable, then
\[
\bKO(\cC) \cong \bGO(\cC^\heartsuit).
\]
\end{citethm}
We also extend the notion of baric structure from~\cite{baric} to $\bZ^n$-graded triangulated categories. Again, it comes with a notion of topological Grothendieck group, somehow generalizing the one from~\cite{acharstroppel}. 
As before, when the baric structure is nice enough, the topological Grothendieck group coincides with the asymptotic one: 
\begin{citeprop}{prop:baricisoKO}
Suppose $\beta$ is a  full, non-degenerate, stable, cone bounded baric structure on $\cC$.
There is a surjection
\[
 \bKO^\Delta(\cC,\beta) \twoheadrightarrow \bKO^\Delta(\cC),
\]
induced by the identity on $K_0^\Delta(\cC)$. 
If  $\beta$ is locally finite, then it is an isomorphism
\[
 \bKO^\Delta(\cC,\beta) \cong \bKO^\Delta(\cC). 
\]
\end{citeprop}

Finally, we investigate in \cref{sec:cblfdgalg} the case of the derived category $\cD(A,d_A)$ of a $\bZ^n$-graded positive dg-algebras $(A,d_A)$. 
We introduce the notion of \emph{c.b.l.f. derived category $\cD^{cblf}(A,d_A)$} given by the full subcategory of $\cD(A,d_A)$ consisting of objects having c.b.l.f. dimensional, bounded from below homology. 
 Since $\cD(A,d_A)$ comes with a canonical t-structure, as a consequence of  \cref{thm:eqasympKOheart}, we obtain the following:
 \begin{citecor}{cor:dgheart}
Let $(A,d_A)$ be a $\bZ^n$-graded dg-algebra. If $A$ is positive for the homological grading and $H^0(A,d_A)$ is locally finite dimensional, then
\[
\bKO^\Delta(\cD^{cblf}(A,d_A)) \cong \bGO(H^0(A,d_A)\modlf).
\]
\end{citecor}
Adding some more hypothesis on $(A,d_A)$ we introduce the notion of a \emph{c.b.l.f. positive dg-algebra} (see \cref{def:cblfposdgalg}). We show any c.b.l.f. dimensional $(A,d_A)$-module admits a cover by a c.b.l.f. iterated extension (see \cref{def:cblfitext}) of projective $A$-modules. This allows us to show there is a natural baric structure on $\cD^{cblf}(A,d_A)$, such that the corresponding topological Grothendieck group coincides with the asymptotic one. In addition, we obtain:
\begin{citethm}{thm:triangtopK0genbyPi}
Let $\{P_i\}_{i \in I}$ be a full collection of  distinct, indecomposable relatively projective $(A,d_A)$-modules.
The asymptotic Grothendieck group is a free $\bZ\pp{x_1, \dots, x_\ell}$-module
\begin{align*}
\bKO^\Delta(\cD^{cblf}(A,d_A))  & \cong  \bigoplus_{i \in I} \bZ\pp{x_1, \dots, x_\ell} \cdot [P_i],
\end{align*}
with basis given by $\{[P_i]\}_{i \in I}$. Moreover, $\bKO^\Delta(\cD^{cblf}(A,d_A))$ is also freely generated by the classes of distinct simple modules $\{[S_i]\}_{i \in I}$.
\end{citethm}

Finally, we show the following result, which is particularly useful for higher representation theory purposes:

\begin{citeprop}{prop:derivedtensorinduces}
Let $(A,d_A)$ and $(A',d_{A'})$ be two c.b.l.f. positive dg-algebras. Let $B$ be a c.b.l.f. dimensional $(A',d_{A'})$-$(A,d_A)$-bimodule. The derived tensor product functor
\[
F : \cD^{cblf}(A,d_A) \rightarrow \cD^{cblf}(A',d_{A'}), \quad F(X) := B \Lotimes_{(A,d_A)} X,
\]
induces a continuous map
\[
[F] : \bKO^\Delta(\cD^{cblf}(A,d_A))  \rightarrow \bKO^\Delta(\cD^{cblf}(A',d_{A'})).
\]
\end{citeprop}

\subsection*{Acknowledgments}
The author is grateful to the Max Planck Institute for Mathematics in Bonn for its hospitality and financial support. 
The author was a Research Fellow of the Fonds de la Recherche Scientifique - FNRS, under Grant no.~1.A310.16 while working on this project. 
%


%% file: sections/classK0.tex

\section{Usual Grothendieck groups}\label{sec:classK0}

We start by recalling classical facts about Grothendieck groups. 
There are three usual notions of Grothendieck groups:  one for additive categories, one for abelian categories (or Quillen exact categories), and one for triangulated categories. 
However, the recipe is always the same: we start with a pointed, small category~$\cC$ admitting some structure that comes with a collection of distinguished triples of objects $\{ \brak{X,Y,Z} \}$ for some $X,Y,Z \in \cC$, encoding the fact that $Y$ is built from $X$ and $Z$. 
Then, one constructs the free abelian group $F(\cC) := \bigoplus_{X \in [\cC]} \bZ \cdot [X] $ over the set $[\cC]$ of equivalences of objects of $\cC$ up to isomorphism. 
The \emph{Grothendieck group} is
\begin{align*}
K_0(\cC) &:= F(\cC)/T(\cC), 
\intertext{where}
 T(\cC) &:= \{[Y]-[X]-[Z] \ |\  \brak{X,Y,Z} \text{ is a distinguished triple}\}.
\end{align*}
The Grothendieck group inherits other properties from the category. Indeed, suppose we have two categories $\cC$ and $\cC'$, both having distinguished triples, and a functor $F : \cC \rightarrow \cC'$ that preserves these triples. Then it induces a map
\[
K_0(\cC) \xrightarrow{[F]} K_0(\cC').
\]
From this, we can start playing around. 
If the category $\cC$ has a monoidal structure and the bifunctor
\[
- \otimes - : \cC \times \cC \rightarrow \cC
\]
preserves distinguished triples, then the Grothendieck group comes with a ring structure
\[
[X] . [Y] := [X \otimes Y].
\]
Another example of interest for us is when $\cC$ is $\bZ$-graded. In this paper, it means that we endow $\cC$ with an auto-equivalence 
\[
\brak{1} : \cC \rightarrow \cC,
\]
called \emph{grading shift}, 
preserving distinguished triples.  We write $\brak{-1}  := \brak{1}^{-1}$, and in general $\brak{k} := \brak{1}^k$ for all $k \in \bZ$. In that case, the Grothendieck group is a $\bZ[q,q^{-1}]$-module with the action of $q^k$ given by
\[
q^{k} \cdot [X] := [X\brak{k}].
\]

\subsection{Additive categories}\label{sec:K0split}

The first examples of a categories that comes with distinguished triples are the additive categories, where there is a triple $\brak{X,Y,Z}$ whenever
\[
Y \cong X \oplus Z.
\]
For this particular case, one calls $K^\oplus_0(\cC)$ the \emph{split Grothendieck group} of $\cC$.

\subsubsection{Krull--Schmidt categories}
.
A \emph{Krull--Schmidt category} is an additive category in which every object decomposes as a finite direct sum of objects having local endomorphism rings.
The main property of such a category is that any object $X$ decomposes as a finite direct sum of indecomposable objects (which coincide with the objects having local endomorphism rings). Moreover, this decomposition is essentially unique. From this, one easily deduces the following classical result:

\begin{prop}
Let $\cC$ be a Krull--Schmidt category with a full collection of pairwise non-isomorphic indecomposable objects $\{X_i\}_{i \in I}$. The split Grothendieck group of $\cC$ is a free $\bZ$-module
\[
K^\oplus_0(\cC) \cong \bigoplus_{i \in I} \bZ \cdot [X_i],
\]
with basis $\{ [X_i] \}_{i \in I}$.
\end{prop}

\begin{defn}
A category $\cC$ is \emph{strictly $\bZ$-graded} if it is $\bZ$-graded and 
\[
X\brak{k} \ncong X
\]
for all non-zero object $X \in \cC$, and all $k \in \bZ\setminus\{0\}$. 
We say that a category is \emph{$\bZ$-graded Krull--Schmidt} if it is both strictly $\bZ$-graded and Krull--Schmidt.
\end{defn}

We say that two objects $X$ and $Y$ are \emph{distinct} if $X \ncong Y\brak{k}$ for all $k \in \bZ$. The following is classical:

\begin{prop}\label{prop:krullschmidt}
Let $\cC$ be a strictly $\bZ$-graded Krull--Schmidt category with a full collection of pairwise distinct indecomposable objects $\{X_i\}_{i \in I}$. 
The split Grothendieck group of $\cC$ is a free $\bZ[q,q^{-1}]$-module 
\[
K^\oplus_0(\cC) \cong \bigoplus_{i \in I} \bZ[q,q^{-1}] \cdot [X_i],
\]
with basis $\{[X_i]\}_{i \in I}$. 
\end{prop}

\begin{exe}\label{exe:splitbbbk}
Let $\Bbbk$ be a field. 
Consider the category $\Bbbk\gmod$ of finite dimensional $\bZ$-graded $\Bbbk$-vector spaces, with degree preserving maps. It is a graded Krull--Schmidt category with a unique indecomposable object, up to grading shift and isomorphism, given by $\Bbbk$. Thus, the split Grothendieck group is 
\[
K^\oplus_0(\Bbbk\gmod) \cong \bZ[q,q^{-1}].
\]
If we consider $\Bbbk\gmod$ with the monoidal structure given by tensor product over $\Bbbk$, we get a categorification of $\bZ[q,q^{-1}]$ as a ring.
\end{exe}

\begin{rem}\label{rem:eleinbergswindle}
It is important to restrict to finite dimensional modules, and in general we need to avoid categories with all (countably) infinite (co)products. Indeed, otherwise the Eilenberg swindle applies:
\begin{align*}
\coprod_{i \in \bN} X &\cong X \oplus \coprod_{i \in \bN} X & &\Rightarrow & [\coprod_{i \in \bN} X] = [X] + [\coprod_{i \in \bN} X ],
\end{align*}
implying that $[X] = 0$ for all $X$, and thus $K_0(\cC) \cong 0$.
\end{rem}

\subsection{Abelian categories}\label{sec:G0abelian}

Another class of categories that come with distinguished triples is given by the abelian categories. In this situation, a distinguished triple $\brak{X,Y,Z}$ is a short exact sequence
\[
0 \rightarrow X \hookrightarrow Y \twoheadrightarrow Z \rightarrow 0.
\]
This generalizes to Quillen exact categories~\cite{quillen}. Since abelian categories are also additive, we write $G_0(\cC)$
for the Grothendieck group obtained from short exact sequences, so that we do not mistake it with the split Grothendieck group $K^\oplus_0(\cC)$ (which is in general different).

\subsubsection{Jordan--H\"older}

If $\cC$ is a finite length abelian category, then by the classical Jordan--H\"older theorem every object admits an essentially unique finite composition series. Moreover, if 
\[
0 = X_{m+1} \subset X_m \subset\cdots \subset X_1 \subset X_0 = X,
\]
is a finite filtration of $X$ then
\[
[X] = \sum_{r=0}^{m} [X_{r}/X_{r+1}],
\]
in $G_0(\cC)$. 
Therefore, we get the following well-known result:

\begin{prop}
Let $\cC$ be a finite length category with a full collection of pairwise non-isomorphic simple objects $\{S_i\}_{i \in I}$. 
The Grothendieck group of $\cC$ is a free $\bZ$-module
\[
G_0(\cC) \cong \bigoplus_{i \in I}  \bZ \cdot [S_i],
\]
with basis $\{[S_i]\}_{i \in I}$.
\end{prop}

The result also holds in the graded case.

\begin{prop}
Let $\cC$ be a strictly $\bZ$-graded, finite length category with a full collection of pairwise  distinct simple objects $\{S_i\}_{i \in I}$. The Grothendieck group of $\cC$ is a free $\bZ[q,q^{-1}]$-module 
\[
G_0(\cC) \cong \bigoplus_{i \in I}  \bZ[q,q^{-1}] \cdot [S_i],
\]
with basis $\{[S_i]\}_{i \in I}$.
\end{prop}

\begin{exe}
In $\Bbbk\gmod$, there is a unique simple object up to grading shift and isomorphism, given by $\Bbbk$. Therefore
\[
G_0(\Bbbk\gmod) \cong \bZ[q,q^{-1}].
\]
In this particular case, it can be viewed also as a consequence of \cref{exe:splitbbbk} since every short exact sequence in $\Bbbk\gmod$ splits, and both Grothendieck groups coincide. 
\end{exe}

\subsection{Triangulated categories}\label{sec:triangulatedcat}

The third class of categories we are interested in is given by triangulated categories~\cite{verdier}. 
A triangulated category is the datum of an additive category $\cC$, an automorphism $[1] : \cC \rightarrow \cC$ called the \emph{translation functor}, and a class of \emph{distinguished triangles}:
\[
X \rightarrow Y \rightarrow Z \rightarrow X[1],
\]
respecting some axioms (see below). In this situation, one calls $Y$ an \emph{extension of $X$ by $Z$}.
In particular $\cC$ comes equipped with a $\bZ$-action such that $[1] \circ [-1] = \id = [-1] \circ [1]$.

\subsubsection{Axiomatic}

Here is the (redundant) list of axioms a triangulated category must respect:
\begin{enumerate}
\item every diagram $X \rightarrow Y \rightarrow Z \rightarrow X[1]$ isomorphic to a distinguished triangle is itself a distinguished triangle;
\item the diagram $X \xrightarrow{\id_X} X \rightarrow 0 \rightarrow X[1]$ is a distinguished triangle; 
\item for each map $X \xrightarrow{f} Y$ there is a \emph{mapping cone} $\cone(f)$ fitting in a distinguished triangle
\[
X \xrightarrow{f} Y \rightarrow \cone(f) \rightarrow X[1],
\]
\item a diagram $X \xrightarrow{f} Y \xrightarrow{g} Z \xrightarrow{h} X[1]$ is a distinguished triangle if and only if 
\[
Y \xrightarrow{g} Z \xrightarrow{h} X[1] \xrightarrow{-f[1]} Y[1]
\]
is also a distinguished triangle;
\item given a commutative diagram
\[
\begin{tikzcd}
X \ar{r}{f} \ar{d}{\alpha} & Y \ar{d}{\beta} \\
X' \ar{r}{f'} & Y'
\end{tikzcd}
\]
and two distinguished triangles $X \xrightarrow{f} Y \xrightarrow{g} Z \xrightarrow{h} X[1]$ and $X' \xrightarrow{f'} Y' \xrightarrow{g'} Z' \xrightarrow{h'} X'[1]$ then there is a (non-unique) map $Z \xrightarrow{\gamma} Z'$ such that the following diagram commutes:
\[
\begin{tikzcd}
X \ar{r}{f} \ar{d}{\alpha} & Y \ar{r}{g} \ar{d}{\beta} & Z \ar{r}{h} \ar[dashrightarrow]{d}{\exists\gamma} & X[1] \ar{d}{\alpha[1]} \\
X' \ar{r}{f'} & Y'  \ar{r}{g'} & Z'  \ar{r}{h'} & X'[1]
\end{tikzcd}
\]
\item 
given three distinguished triangles
\[
\begin{tikzcd}[row sep = tiny, column sep = scriptsize]
X \ar{r}{f} & Y \ar{r} & Y/X \ar{r} & X[1], \\
Y \ar{r}{g} & Z \ar{r} & Z/Y \ar{r} & Y[1], \\
X \ar{r}{g\circ f} & Z \ar{r} & Z/X \ar{r} & X[1],
\end{tikzcd}
\]
then there is a distinguished triangle
\[
Y/X \rightarrow Z/X \rightarrow Z/Y \rightarrow (Y/X)[1],
\]
such that the following diagram commutes:
\[
\begin{tikzcd}[column sep = small, row sep = small]
X \ar{rr} \ar{rd} && Z \ar{rr}  \ar{rd} && Z/Y \ar{rr} \ar{rd} && (Y/X)[1]. \\
&Y \ar{ru} \ar{rd} &&Z/X \ar{ru} \ar{rd} &&Y[1] \ar{ru}& \\
&& Y/X \ar{rr} \ar{ru} && X[1] \ar{ru} &&
\end{tikzcd}
\]
\end{enumerate}

By definition, triangulated categories come with distinguished triples given by the distinguished triangles. This defines the \emph{triangulated Grothendieck group} $K_0^\Delta(\cC)$.
One of the nice  properties of the triangulated Grothendieck group is that the second and fourth axioms of triangulated categories imply that
\[
X \rightarrow 0 \rightarrow X[1] \rightarrow X[1]
\] 
is a distinguished triangle. Therefore, in  $K_0^\Delta(\cC)$ we have
\[
0 = [X] + [X[1]],
\]
so that the translation functor descends as multiplication by $-1$ in $K_0^\Delta(\cC)$. 

\section{Derived categories} \label{sec:triangcat}

A common source of triangulated structures is given by looking at homotopy categories or derived categories. We will consider triangulated categories obtained from dg-algebras. For this, we will mainly follow \cite{keller} (see also~\cite{bernsteinlunts}) and borrow some ideas and notations from~\cite{qithesis}.  
A good survey for dg-categories is~\cite{keller2}.

\subsection{Dg-algebras}\label{sec:dgalgebras}

Let $\Bbbk$ be a commutative unital ring. 

\begin{defn}
A \emph{dg-($\Bbbk$-)algebra} $(A,d_A)$ is a unital, $\bZ$-graded $\Bbbk$-algebra $A = \bigoplus_{i \in \bZ} A^i$ with a \emph{differential} $d_A : A^i \rightarrow A^{i-1}$ of degree $-1$ such that $d_A^2 = 0$ and
\[
d_A(ab) = d_A(a)b + (-1)^{\deg(a)}ad_A(b),
\]
for all homogeneous elements $a,b \in A$.
 A morphism of dg-algebras is a map of algebras that preserves the grading and commutes with the differentials.
\end{defn}

\begin{rem}
In opposite to~\cite{keller}, we consider differentials of degree $-1$ instead of $+1$. 
 This will impact many definitions for the remaining of the paper. 
\end{rem}

\begin{defn}
A (left) \emph{dg-module} $(M, d_M)$ over a dg-algebra $(A,d_A)$ is the datum of a $\bZ$-graded (left)  $A$-module $M = \bigoplus_{i \in \bZ} M^i$ and a differential $d_M$ such that $d_M^2= 0$ and
\[
d_M(a\cdot m) = d_A(a) \cdot m + (-1)^{\deg(a)} a \cdot d_M(m),
\]
for all homogeneous elements $a \in A$ and $m \in M$. A morphism of dg-modules is a map of modules that preserves the grading and commutes with the differentials.
\end{defn}
There are similar obvious definitions for right dg-modules and dg-bimodules.

Given a dg-algebra $(A,d_A)$, one can construct its \emph{homology} as
\[
H(A,d_A) := \frac{\ker d_A}{\Image d_A},
\]
which is a $\bZ$-graded algebra. Similarly, one can construct the homology $H(M,d_M)$ of an $(A,d_A)$-module, and it is a $\bZ$-graded $H(A,d_A)$-module.  
As usual in homological algebra, a morphism of dg-modules induces a morphism of graded modules in homology, and a short exact sequence of dg-modules gives  rise to a long exact sequence in homology. Moreover, one says that a map $f : (X,d_X)  \rightarrow (Y,d_Y)$ is a \emph{quasi-isomorphism} if it induces an isomorphism in homology $f_* : H(X, d_X) \xrightarrow{\simeq} H(Y, d_Y)$. 
A dg-algebra is \emph{formal} if it is quasi-isomorphic to its homology, viewed as a dg-algebra with a trivial differential.

\smallskip

The category $(A,d_A)\amod$ of (left) dg-modules over a dg-algebra $(A,d_A)$ is a $\bZ$-graded abelian category with $[1]$ acting by:
\begin{itemize}
\item shifting the degree of all elements up by $1$;
\item switching the sign of the differential $d_{M[1]} := -d_M$;
\item introducing a sign in the action $r \cdot m[1]  := (-1)^{\deg r} (r \cdot m)[1]$.
\end{itemize}

\begin{rem}
If we consider right dg-modules instead, then the translation functor does not change the sign in the action:
\[
m[1] \cdot r := (m \cdot r)[1].
\]
 In the case of dg-bimodules, it only twists the left-action.
\end{rem}

Let $f : (X,d_X) \rightarrow (Y,d_Y)$ be a morphism of dg-modules. Then, one constructs the \emph{mapping cone} of $f$ as 
\begin{align} \label{eq:cone}
\cone(f) &:= (X[1] \oplus Y, d_C), & 
d_C &:= \begin{pmatrix} -d_X & 0 \\ f & d_Y \end{pmatrix}.
\end{align}
It is an $(A,d_A)$-module. It fits in a short exact sequence
\[
0 \rightarrow Y \hookrightarrow \cone(f) \twoheadrightarrow X[1] \rightarrow 0.
\]

\subsection{Derived category}\label{sec:derivedcat}

Given two dg-modules $X,Y \in  (A,d_A)\amod$, the \emph{graded enriched hom-space} is
\[
\HOM_{A\amod}(X,Y) := \bigoplus_{i\in\bZ} \Hom_{A\amod}(X,Y[i]).
\]
It is a graded space with $f : X \rightarrow Y[i]$ having degree $-i$. 
It can upgraded to a dg-hom space $(\HOM_{A\amod}(X,Y), d)$ with differential given by 
\[
d(f) = d_Y \circ f - (-1)^{\deg(f)} f \circ d_X.
\]
Then, we consider the dg-category of $(A,d_A)$-module, written $(A,d_A)\dgmod$, where the objects are the same as in $(A,d_A)\amod$ and the hom-spaces are the dg-hom spaces 
\begin{equation*}
\Hom_{(A,d_A)\dgmod}(X,Y) := (\HOM_{A\amod}(X,Y), d). 
\end{equation*}
Let $Z^0((A,d_A)\dgmod)$ and $H^0((A,d_A)\dgmod)$ be the categories with the same objects as $(A,d_A)\dgmod$ but with hom-spaces being
\begin{align*}
\Hom_{Z^0((A,d_A)\dgmod)}(X,Y) &:= \ker( \Hom_{A}(X,Y[0]) \xrightarrow{d}  \Hom_{A}(X,Y[1])), \\
\Hom_{H^0((A,d_A)\dgmod)}(X,Y) &:= \frac{\ker( \Hom_{A}(X,Y[0]) \xrightarrow{d}  \Hom_{A}(X,Y[1]))}
	{\Image( \Hom_{A}(X,Y[-1]) \xrightarrow{d}  \Hom_{A}(X,Y[0]))}. 
\end{align*}
Note that $Z^0((A,d_A)\dgmod)$ is equivalent to $(A,d_A)\amod$. 

\smallskip

One calls $H^0((A,d_A)\dgmod)$ the \emph{homotopy category} of $(A,d_A)$.
It is a triangulated category with translation functor given by $[1]$ defined as in $(A,d_A)\amod$. The distinguished triangles are isomorphic to triangles
$
X \rightarrow Y \rightarrow Z \rightarrow X[1]
$
obtained from short exact sequences $X \rightarrow Y \rightarrow Z$ in $Z^0((A,d_A)\dgmod)$ that splits in $A\amod$.  Mapping cones are constructed as in~\eqref{eq:cone}.

\smallskip

The \emph{derived category} $\cD(A,d_A)$ of $(A,d_A)\dgmod$ is given by localizing the homotopy category $H^0((A,d_A)\dgmod)$ along quasi-isomorphisms. It inherits a triangulated structure from the homotopy category, but now any short exact sequences in $Z^0((A,d_A)\dgmod)$ gives rise to a distinguished triangle in $\cD(A,d_A)$. 

\smallskip

It is in general a hard problem to compute the hom-spaces in the derived category, except for a specific class of objects.

\begin{defn}
An $(A,d_A)$-module is a \emph{relatively projective module} if it is a direct summand of a free module in $(A,d_A)\amod$. 
\end{defn}
\begin{defn}
An $(A,d_A)$-module $Y$ \emph{satisfies property (P)} if  there is an exhaustive filtration of $(A,d_A)$-modules
\[
0 = F_0 \subset F_1 \subset F_2 \subset  \cdots \subset F_r \subset F_{r+1} \subset \cdots \subset Y,
\]
such that each $F_{r+1}/F_r$ is isomorphic in  $(A,d_A)\amod$ to a relatively projective module. 
\end{defn}
\begin{defn}
An $(A,d_A)$-direct summand of a property (P) module is called \emph{cofibrant}.
\end{defn}

For $P$ cofibrant and any $Y \in (A,d_A)\amod$, one has
\[
\Hom_{\cD(A,d_A)}(P,Y) \cong \Hom_{H^0((A,d_A)\dgmod)}(P,Y). 
\]
Moreover, $\Hom_{H^0((A,d_A)\dgmod)}(P,Y)$ can be computed as $\Hom_{(A,d_A)\amod}(P,Y)/\sim$ 
where $f \sim g$ if there exists a morphism $h : P \rightarrow Y[1]$ such that 
\[
f-g = d(h) = d_Y \circ h + h \circ d_P.
\]
Luckily, there is enough cofibrant modules thanks to following result:
\begin{prop}[{\cite[\S3.1]{keller}}]
For any $(A,d_A)$-module $X$ there is an $(A,d_A)$-module satisfying property (P), denoted $\br(X)$, with a surjective quasi-isomorphism
\[
\br(X) \twoheadrightarrow X.
\]
One calls $\br(X)$ the \emph{bar resolution} of $X$, and the assignment $X \rightarrow \br(X)$ is functorial.
\end{prop}

This allows the construction of a \emph{derived hom functor} for each $X \in (A,d_A)\amod$ as
\begin{align*}
\RHOM(-,X) : \cD(A,d_A) &\rightarrow \cD(\Bbbk, 0), & \RHOM(Y,X) := (\HOM(\br(Y), X), d).
\end{align*}
There is a similar definition for $\RHOM(X,-)$ with 
\[
\RHOM(X,Y) = (\HOM(\br(X), Y), d). 
\]
In that case, if $X$ carries the structure of an $(A,d_A)$-$(B,d)$-bimodule, then it can be upgraded into a functor
\[
\RHOM(X,-) : \cD(A,d_A) \rightarrow \cD(B, d).
\]
The derived hom-functor admits a left  adjoint functor called \emph{derived tensor product functor}
\begin{align*}
X \Lotimes_{(B,d)} - :\cD(B,d) &\rightarrow \cD(A,d_A), & X \Lotimes_{(B,d)} Y := X \otimes_{(B,d)} \br{(Y)},
\end{align*}
where $\br{(Y)}$ is the bar resolution of $Y$ in $(B,d)\amod$. 
 
\subsection{Back to Grothendieck groups}\label{sec:compactderivedcat}

In order to avoid the Eilenberg swindle (cf. \cref{rem:eleinbergswindle}), one usually restricts to a subcategory of the derived category. 
In this setting, as explained in ~\cite{qithesis} for example, there is a special class of objects that play the role of finitely generated objects. 

\begin{defn}
An object $X$ in a category with (infinite) coproducts is called \emph{compact} if, for any family of objects $\{N_i\}_{i \in I}$, the natural map
\[
\coprod_{i \in I} \Hom(M,N_i) \xrightarrow{\simeq} \Hom(M, \coprod_{i\in I} N_i)
\]
is an isomorphism.
\end{defn}

The \emph{compact derived category} is the full subcategory $\cD^c(A,d_A) \subset \cD(A,d_A)$ given by objects quasi-isomorphic to compact objects. It is a triangulated, idempotent complete subcategory of $\cD(A,d_A)$. Luckily, in the setting of dg-categories, $\cD(A,d_A)$ is compactly generated (see~\cite{neeman92, neeman96,ravenel}) by the free module $(A,d_A)$ and all its shifts $A[r]$, $r \in \bZ$. Equivalently, it is compactly generated by finitely generated, relatively projective modules. Therefore, we can obtain all compact objects as direct summands of finite iterated extensions of the compact generators. For this reason, as in~\cite{qithesis}, one defines the following:

\begin{defn}\label{def:finitecellmodule}
A \emph{finite cell module} is an $(A,d_A)$-module satisfying property (P) and being finitely generated as $A$-module.
\end{defn}

Thus, finite cell modules are modules that admit finite filtrations in $(A,d_A)\amod$ where each composition factor is a finitely generated, relatively projective module. 
In particular, any compact object is isomorphic to a direct summand of a finite cell module. 
%
%
We say that a collection $\{K_i\}_{i \in I}$ is a \emph{minimal compact generating set} of $\cC$ if
\begin{itemize}
\item the collection $\{K_i[n]\}_{i\in I, n \in \bZ}$ compactly generates $\cC$;
\item $K_i$ cannot be obtain as a direct summand of an iterated extension of objects in $\{K_j[n]\}_{j\in I\setminus\{i\}, n \in \bZ}$.
 \end{itemize}
 Note that the second condition implies in particular that $K_i$ is distinct from $K_j$ with respect to $[1]$ for all $i \neq j$. 
 Then we obviously have the following result:

\begin{prop}
Let $\cC^c$ be the compact subcategory of a compactly generated triangulated category $\cC$ with minimal compact generating set $\{K_i\}_{i \in I}$. 
 The triangulated Grothendieck group of $\cC^c$ is a free $\bZ$-module 
\[
K_0^\Delta(\cC^c) \cong \bigoplus_{i \in I} \bZ \cdot [K_i],
\]
with basis $\{[K_i]\}_{i \in I}$.
\end{prop}

Note that it is in general a non-trivial problem to determine whether a set of compact generators is a minimal compact generating one or not.

\begin{exe}
Here $\Bbbk$ is a field. 
The triangulated Grothendieck group of $(\Bbbk,0)$ respects
\[
K_0^\Delta(\cD^c(\Bbbk,0)) \cong \bZ,
\]
with minimal compact generating collection given by $\{ (\Bbbk,0) \}$.
\end{exe}

\section{Multigradings}\label{sec:finitemultigradings}

We say that a category $\cC$ is \emph{$\bZ^n$-(multi)graded} if it is endowed with a collection of auto-equivalences $\brak{1}_i $ for $1 \leq i \leq n$, strictly commuting with each others. If $\cC$ is triangulated, then we say it is a $\bZ^n$- graded triangulated categories if $\brak{1}_i$ strictly commutes with the translation functor $[1]$ for each $i$ (making it a $\bZ^{1+n}$-graded category), and $\brak{1}_i$ preserves distinguished triangles. 
The Grothendieck group of a $\bZ^n$-graded category is a $\bZ[x_1^{\pm 1}, \dots, x_n^{\pm 1}]$-module. 
In order to simplify notations, we write a shift in the $\bZ^n$-multigrading as $x^\bg$ for $\bg  =  (g_1, \dots, g_n) \in \bZ^n$, or as monomial $x_1^{g_1}\cdots x_n^{g_n} \in \bZ[x_1^{\pm 1}, \dots, x_n^{\pm 1}]$, so that we have
\[
x^\bg  X := x_1^{g_1}\cdots x_n^{g_n} X := X\brak{g_1}\cdots\brak{g_n},
\]
which is well defined since the shifts strictly commute with each others. We keep the notation $[1]$ for the shift in the homological grading.
We say that $\cC$ is \emph{strictly} $\bZ^n$-graded if it is $\bZ^n$-graded and if 
\[
x^\bg X \ncong X
\]
for all $x \in \bZ^n\setminus\{0\}$, and non-zero object $X \in \cC$.

\begin{rem}
Note that a (non-zero) category admitting arbitrary coproducts cannot be strictly $\bZ^n$-graded, since one would have $\coprod_{i \in \bZ} x_1^i X \cong x_1 (\coprod_{i \in \bZ} x_1^i X)$ for any object $X$.
\end{rem}

A $\bZ^n$-graded dg-structure is a structure that is both $\bZ^n$-graded and dg, and with the differential preserving the $\bZ^n$-grading. In particular, a $\bZ^n$-graded dg-algebra $(A,d_A)$ decomposes as a $\bZ \times \bZ^n $-graded $\Bbbk$-vector space as a direct sum
\[
A \cong \bigoplus_{(h,\bg) \in \bZ \times \bZ^n } A_\bg^h,
\]
and the multiplication and differential respect
\begin{align*}
A_\bg^h \cdot A_{\bg'}^{h'} &\subset A_{\bg+\bg'}^{h+h'}, & d_A(A_\bg^h) &\subset A_\bg^{h-1}.
\end{align*}

Everything said in~\cref{sec:triangulatedcat} generalizes to multigraded dg-structures. In this case, we let $(A,d_A)\amod$ to be the $\bZ \times \bZ^n$-graded abelian category of $\bZ^n$-graded dg-modules with maps that preserve all gradings. Note that in this context, the enriched graded hom-spaces are given by
\[
\HOM_{A\amod}(X,Y) := \bigoplus_{(h,\bg) \in \bZ \times \bZ^n} \Hom_{A\amod}(X,x^\bg Y[h]),
\]
where a map $f : X \rightarrow x^\bg Y[h]$ has degree $(-h,-\bg)$.

\begin{exe}
Here $\Bbbk$ is a field.
We consider $(\Bbbk,0)$ as a $\bZ^n$-graded dg-algebra concentrated in degree $0$, and its derived category is $\bZ^n$-graded triangulated. 
Its triangulated Grothendieck group respects
\[
K_0^\Delta(\cD^c(\Bbbk, 0)) \cong \bZ[x_1^{\pm 1}, \dots, x_n^{\pm 1}].
\]
\end{exe}

\begin{exe}
Recall the $\bZ \times \bZ^2$-graded algebra $R$ from \cref{sec:mainexintro} and its simple module $L$, as well as its dg-version $(R,0)$. 
Consider the additive category of $\bZ^2$-graded, finitely generated projective $R$-modules $R\pmodfg$, or alternatively the $\bZ \times \bZ^2$-graded compact derived category $\cD^c(R,0)$). We have
\begin{align*}
K_0(R\pmodfg) &\cong \bZ[h^{\pm 1},q^{\pm 1},\lambda^{\pm 1}] \cdot [R],
&
K^\Delta_0(\cD^c(R,0)) &\cong \bZ[q^{\pm 1},\lambda^{\pm 1}] \cdot [(R,0)].
\end{align*}
However, note that $L \notin \pmodfg$ and $(L, 0) \notin \cD^c(R,0)$. 

We can also consider the abelian category of $\bZ^2$-graded, finite-dimensional $R$-modules $R\amod_{fd}$, or alternatively the full subcategory $\cD^{fd}(R,0)$ of the derived  category $\cD(R,0)$ given by $\bZ \times \bZ^2$-graded dg-modules with finite-dimensional homology. Then we obtain
\begin{align*}
G_0(R\amod_{fd}) &\cong \bZ[h^{\pm 1},q^{\pm 1},\lambda^{\pm 1}] \cdot [L],
&
K^\Delta_0(\cD^{fd}(R,0)) &\cong \bZ[q^{\pm 1},\lambda^{\pm 1}] \cdot [(L,0)].
\end{align*}
However, note that since $R$ is infinite-dimensional, we have $R \notin R\amod_{fd}$ and $(R,0) \notin \cD^{fd}(R,0)$. In particular, \cref{eq:changeofbasisintro} does not make sense in either of these contexts.
\end{exe}


%% file: sections/addTopK0.tex

\section{Ring of formal Laurent series}\label{sec:formalLaurent}

We follow the description given in \cite{laurent}. We choose an arbitrary additive total order $\prec$ on $\bZ^n$. That is, $a \prec b$ implies $a + c \prec b + c$, for every $a,b,c \in \bZ^n$.

\begin{defn}
A \emph{cone} is a subset $C \subset \bR^n$ such that
\[
C = \{ \alpha_1 v_1 + \dots + \alpha_n v_n | \alpha_i \in \bR_{\geq 0}\},
\]
for some generating elements $v_1, \dots, v_n \in \bZ^n$. 
Moreover, $C$ is \emph{compatible with the order $\prec$} if $0 \prec v_i$ for all $i \in \{1, \dots, n\}$.
\end{defn}

\begin{rem}
An important fact to note here is that given a cone $C$ compatible with~$\prec$, then we can write $C \cap\bZ^n = \{0 = \bc_0, \bc_1, \bc_2, \dots \}$ with $\bc_i \prec \bc_{i+1}$ for all $i \in \bN$. In particular given an element $\bz \in C \cap\bZ^n$, then there is only finitely many elements $\bc \in  C \cap\bZ^n$ sich that $\bc \prec \bz$.
\end{rem}

Let $\Bbbk$ be a unital, commutative ring. Recall that for $\bk = (k_1, \dots, k_n) \in \bZ^n$ we write 
$
x^{\bk} := x_1^{k_1} \dots x_n^{k_n}.
$ 
Let $C \subset \bR^n$ be a cone compatible with $\prec$ and define
\[
\Bbbk_C\llbracket x_1, \dots, x_n \rrbracket := \left\{ \sum_{\bk  \in \bZ^n} a_{\bk} x^{\bk} | a_{\bk} \in \Bbbk, a_{\bk} = 0 \text{ if } \bk \notin C \right\}.
\]

\begin{prop}[{\cite[Theorem 10]{laurent}}]
The set $\Bbbk_C\llbracket x_1, \dots, x_n \rrbracket$ together with the point-wise addition and usual multiplication forms a unital, commutative ring. 
Moreover, for $f(\bx) \in  \Bbbk_C\llbracket x_1, \dots, x_n \rrbracket$ with $f(0) = a_0 \in \Bbbk^\times$ invertible, then there exists $f^{-1}(\bx) \in \Bbbk_C\llbracket x_1, \dots, x_n \rrbracket$ such that $f(\bx) f^{-1}(\bx) = 1$.
\end{prop}

\begin{defn}[{\cite[Definition~14]{laurent}}]
We put 
\[
\Bbbk_\prec \llbracket x_1, \dots, x_n\rrbracket := \bigcup_{C} \Bbbk_C\llbracket x_1, \dots, x_n \rrbracket,
\]
where the union is over all cones compatible with $\prec$. The \emph{ring of formal Laurent series} is 
\[
\Bbbk_\prec \pp{x_1, \dots, x_n} := \bigcup_{\be \in \bZ^n} x^{\be} \Bbbk_\prec \llbracket x_1, \dots, x_n\rrbracket.
\]
\end{defn}

\begin{thm}[{\cite[Theorem 15]{laurent}}]
The set $\Bbbk_\prec \pp{x_1, \dots, x_n}$ together with the usual addition and multiplication forms a ring. 
Whenever $\Bbbk$ is a field, then $\Bbbk_\prec \pp{x_1, \dots, x_n}$ is a field.
\end{thm}

\subsection{Topological structure on the ring of Laurent series}

The ring of formal Laurent series $\Bbbk_\prec \pp{x_1, \dots, x_n}$ comes with a natural topology, which we refer as \emph{$(x_1, \dots, x_n)$-adic topology}. It is defined by using
\begin{equation}\label{eq:neighborhoodLaurent}
\{ V_{\be} := x^{\be} \Bbbk_\prec \llbracket x_1, \dots, x_n\rrbracket \}_{\be \in \bZ^n},
\end{equation}
as neighborhood basis of zero, extending so that $\Bbbk_\prec \pp{x_1, \dots, x_n}$ becomes a topological group. Note that $V_{\be} \subset V_{\boldsymbol f}$ whenever $\be \succ \boldsymbol f$.

\begin{rem}
We would like to point out that the subgroup $V_\be$ is not an ideal of the ring $\Bbbk_\prec \pp{x_1, \dots, x_n}$, and thus the $(x_1, \dots, x_n)$-adic topology is not an `adic topology' in the proper sense. 
\end{rem}

\begin{prop}\label{prop:laurentistop}
The ring of formal Laurent series $\Bbbk_\prec \pp{x_1, \dots, x_n}$  equipped with the $(x_1, \dots, x_n)$-adic topology is a Hausdorff topological ring. 
If $\Bbbk$  is a field, then $\Bbbk_\prec \pp{x_1, \dots, x_n}$ is a topological field.
\end{prop}

\begin{proof}
We recall that a subset $U \subset \Bbbk_\prec \pp{x_1, \dots, x_n}$ is open if and only if for each $f(\boldsymbol x) \in U$ there exists a neighborhood $ V_{\be}$ such that $f(\boldsymbol x) +  V_{\be} \subset U$. We only need to show that the multiplication map is continuous. Therefore we want to show that
\[
U := \{ (f(\boldsymbol x), g(\boldsymbol x)) | f(\boldsymbol x)g(\boldsymbol x) \in V_{\be}\},
\]
for any $\be \in \bZ^n$  
is open. We take any $(f(\boldsymbol x), g(\boldsymbol x)) \in U$. It means we can find $\boldsymbol f$ and $\boldsymbol g$ such that $f(\boldsymbol x) \in V_{\boldsymbol f}$ and $g(\boldsymbol x) \in V_{\boldsymbol g}$ with $\boldsymbol f + \boldsymbol g \succeq \boldsymbol e$. Then we put $\boldsymbol r = \max\{\boldsymbol e, \boldsymbol e - \boldsymbol f\}$ and $\boldsymbol s = \max\{\boldsymbol e, \boldsymbol e - \boldsymbol g\}$ so that
\[
(f(\boldsymbol x), g(\boldsymbol x)) + (V_{\boldsymbol s},V_{\boldsymbol r})  \subset U.
\]
This shows $\Bbbk_\prec \pp{x_1, \dots, x_n}$ is a topological ring.

Moreover, we clearly have 
\[
\bigcap_{\boldsymbol e \in \bZ^n} x^{\boldsymbol e} \Bbbk_\prec \llbracket x_1, \dots, x_n\rrbracket  = \{0\},
\]
and thus, the $(x_1, \dots, x_n)$-adic topology is Hausdorff. 

If $\Bbbk$ is a field, then we write $U := \{ f(\boldsymbol x) | f^{-1}(\boldsymbol x) \in V_{\boldsymbol e}\}$. We have $f(\boldsymbol x) \in U$ if and only if $f(\boldsymbol x) = x^{\boldsymbol f} h(\boldsymbol x)$ for some $h(\boldsymbol x)$ such that $h(0) \neq 0$, and $\boldsymbol f \preceq -\boldsymbol e$. Then we obtain
\[
f(\boldsymbol x) + V_{\boldsymbol f} \subset U,
\]
which concludes the proof.
\end{proof}

\section{C.b.l. additive categories}\label{sec:cbladditive}

We review and extend the notions presented in~\cite[A.1]{naissevaz1}.

If $\cC$ is a category admitting the zero morphisms (e.g. a preadditive category) and $Y_j \xrightarrow{\imath_j} \coprod_i Y_i$ is a coproduct in $\cC$, then there is a unique map $p_j : \coprod_i Y_i \rightarrow Y_j$ given by universal property of the coproduct making the diagram 
\begin{equation} \label{eq:coprodproj}
\begin{tikzcd}
Y_k \ar{r}{\imath_k} \ar[swap]{dr}{\delta{jk}} & \coprod_i Y_i \ar[dashrightarrow]{d}{\exists ! p_j} \\
& Y_j
\end{tikzcd}
\end{equation}
commutes for all $j,k$. Then, by universal property of the product $\prod_i \Hom(-, Y_i) \xrightarrow{\pi_j} \Hom(-,Y_j)$, we get in turns a map $r : \Hom(-, \coprod_i Y_i)  \rightarrow \prod_i \Hom(-, Y_i)$ making the diagram
\begin{equation} \label{eq:prodhomproj}
\begin{tikzcd}
\Hom(-, \coprod_i Y_i) \ar[swap, dashrightarrow]{d}{\exists ! r} \ar{rd}{p_j \circ - } &  \\
\prod_i \Hom(-, Y_i) \ar[swap]{r}{\pi_j} & \Hom(-, Y_j)
\end{tikzcd}
\end{equation}
commutes for all $j$.

\begin{defn}
One says that a coproduct $\coprod_i Y_i$ in a category with the zero morphisms is a \emph{biproduct} if the natural transformation
\[
\Hom(-, \coprod_i Y_i) \xrightarrow{r} \prod_i \Hom(-, Y_i)
\]
is an isomorphism. In that case we write it with a $\bigoplus$ symbol. 
\end{defn}

\begin{prop}\label{prop:localcoprod}
Let $\cC$ be a category admitting the zero morphisms. If $\coprod_i Y_i$ is a biproduct, then it is a product as well. Moreover, if $p_j : \coprod_i Y_i \rightarrow Y_j$ and $\imath_j : Y_i \rightarrow \coprod_i Y_i$ are the projection and injection maps respectively, then $\prod_j (\imath_j \circ p_j) = \id$.
\end{prop}

\begin{proof}
Let $Z \xrightarrow{f_i} Y_i$ be a collection of arrows in $\cC$. Consider the maps $p_j : \coprod_i Y_i \rightarrow Y_j$ as in~\eqref{eq:coprodproj}, and the maps $r :\Hom(-, \coprod_i Y_i) \xrightarrow{r} \prod_i \Hom(-, Y_i)$ and $\pi_j : \prod_i \Hom(-, Y_i) \rightarrow  \Hom(-, Y_j)$ as in~\eqref{eq:prodhomproj}. Then we obtain $p_j \circ r^{-1}(\prod_i f_i) = \pi_j(\prod_i f_i) = f_j$, and by consequence we get a commutative diagram
\[
\begin{tikzcd}
Z \ar[swap]{d}{r^{-1}(\prod f_i)} \ar{dr}{f_j} & \\
\coprod_i Y_i \ar[swap]{r}{p_j} & Y_j.
\end{tikzcd}
\]
This means $\coprod_i Y_i$ is a product with projection maps given by $p_j$.
The second claim follows by the universal property of the product.
\end{proof}

\begin{rem}
The reciprocal of \cref{prop:localcoprod} is also true: if the maps $p_j$ turn a coproduct $\coprod Y_i$ into a product, then it is a biproduct. This follows directly from the fact that in general $\Hom(-, \prod_i X_i) \cong \prod_i \Hom(-, X_i)$.
\end{rem}

\begin{rem}
In order to avoid confusion, we only use the symbol $\bigoplus$ for (potentially infinite) biproducts, and we usually refer to them as direct sums (not to be confused with coproducts).
\end{rem}

\subsection{Locally Krull--Schmidt categories}

\begin{defn}
Let $\{X_i\}_{i \in I}$ be a collection of objects in a category $\cC$ with the zero morphisms. We say that a direct sum is a \emph{locally finite direct sum of $\{X_i\}_{i \in I}$} if it takes the form 
\[
M := \bigoplus_{i\in I} X_i^{\oplus n_i},
\]
for some $\{n_i \in \bN\}_{i \in I}$. We call $n_i$ the \emph{multiplicity of $X_i$ in $M$}.
\end{defn}

\begin{defn}\label{def:locKS}
We say that an additive category $\cC$ is \emph{locally Krull--Schmidt} if every object decomposes into a locally finite direct sum of  non-isomorphic compact objects having local endomorphism ring. 
\end{defn}

Clearly, a locally Krull--Schmidt category is idempotent complete. 
Moreover, the compactness condition allows us to mimic the classical proof (see for example~\cite[Theorem~1]{krullschmidt})   of the Krull--Schmidt property in a Krull--Schmidt category: that is every object decomposes into an essentially unique direct sum of indecomposable objects.

\begin{thm}
\label{thm:cbluniquedecomp}
Consider an isomorphism between two locally finite direct sums in a locally Krull--Schmidt category
\[
\bigoplus_{i\in I} X_i^{\oplus n_i}
 \cong 
\bigoplus_{j \in J} Y_j^{\oplus m_j}
\]
where each $X_i, X_{i'}$ (resp. $Y_{j},Y_{j'}$) are pairwise non-isomorphic indecomposable objects. 
There is a bijection $\sigma : I \rightarrow J$ such that $X_i \cong Y_{\sigma(i)}$, and $n_i = m_{\sigma(i)}$. 
\end{thm}

\begin{proof}
The proof is essentially the same as in \cite[Theorem A.10]{naissevaz1}, and omitted. 
\end{proof}

\begin{cor}
A locally Krull--Schmidt category possesses the cancelation property: for each triple $A,B,C$ such that $A \oplus B \cong A \oplus  C$ then $B \cong C$.
\end{cor}

Let $X$ be an indecomposable object in a locally Krull--Schmidt category $\cC$. We write
\[
X^* : K_0(\cC) \rightarrow \bN, 
\]
where $X^*(M)$ counts the multiplicity of $X$ in $M$. 
We obtain the following result:

\begin{cor}\label{cor:lockrullschmidt}
Let $\cC$ be a locally Krull--Schmidt category with a full collection of pairwise non-isomorphic indecomposable objects $\{X_i\}_{i \in I}$. 
If $\cC$ admits locally finite direct sums 
of indecomposable objects, then its split Grothendieck group is a free $\bZ$-module
\[
K_0^\oplus(\cC) \cong \prod_{i \in I} \bZ \cdot [X_i],
\]
with basis $\{[X_i]\}_{i \in I}$. 
Moreover, we have $[M] = \sum_{i \in I} X_i^*(M) [X_i]$ for all $M \in \cC$. 
\end{cor}

\begin{defn}
Let $\cC, \cC'$ be two locally Krull--Schmudt categories with collections of distinct indecomposable objects $\{X_i\}_{i \in I}$ and $\{X_j'\}_{j \in J}$ respectively. 
We say that a functor $F : \cC \rightarrow \cC'$  is \emph{locally finite} if for all $M \in \cC$ and $j \in J$ there exists a finite subset $I_{M,j} \subset I$ such that
\[
X_i^*(M){X'_j}^*(FX_i) = 0,
\]
for all $i \notin I_{M,j}$.
\end{defn}

\begin{prop}\label{prop:locfinFKS}
Let $F : \cC \rightarrow \cC'$ be an additive locally finite functor between two locally Krull--Schmidt categories. The induced map 
$[F] : K_0(\cC) \rightarrow K_0(\cC')$ 
respects
\[
[M]
\mapsto
 \sum_{j \in J}  \sum_{i \in I} X_i^*(M) {X_j'}^*(FX_i) [X_j']. 
\]
\end{prop}

\begin{proof}
We want to show that for any 
\[
M \cong  \bigoplus_{i \in I} X_i^{\oplus n_i},
\]
then 
\[
\sum_{j \in J} X_j'(FM) [X_j] = \sum_{i \in I} \sum_{j \in J} X_i^*(M){X_j'}^*(X_i)[X_j'] \in  K_0(\cC').
\]
We take $M'$ and $M'' := \bigoplus_{i \in I_{M,j}} X_i^{\oplus n_i}$ such that
\[
M \cong M' \oplus M''.
\]
Thus $X_i^*(M') = 0$ for all $i \in I_{M,j}$. Then, because $I_{M,j}$ is finite and $F$ additive, we obtain
\[
{X_j'}^*(FM) = {X_j'}^*(FM') + \sum_{i \in I_{M,j}} n_i {X_j'}^*(FX_i).
\]
Because $F$ is additive, we have $I_{M',j} \subset I_{M,j} \setminus I_{M'',j}$, and thus $I_{M',j} = \emptyset$, and ${X_j'}^*(FM')  = 0$. Moreover, we have
\[
 \sum_{i \in I_{M,j}} n_i {X_j'}^*(FX_i) = \sum_{i \in I} X_i^*(M) {X_j'}^*(FX_i),
\]
which concludes the proof.
\end{proof}

\subsection{C.b.l. Krull--Schmidt categories}

\begin{defn}
Let $\{X_1, \dots, X_m\}$ be a finite collection of objects in a stricly $\bZ^n$-graded category. Then we consider a coproduct of the form
\[
\coprod_{\bg \in \bZ^n} x^{\bg} (X_1^{\oplus k_{1,\bg}} \oplus \cdots \oplus X_m^{\oplus k_{m,\bg}}),
\]
with each $k_{j, \bg} \in \bN$. We call such coproducts \emph{graded locally finite}. We call it  \emph{c.b.l.f. (cone bounded, locally finite)} if in addition,  there exists a cone $C$ compatible with $\prec$, and $\be \in \bZ^n$ such that we have $k_{j,\bg} = 0$ whenever $\bg -  \be \notin C$.
\end{defn}

One of the nice properties of c.b.l.f. coproducts is that taking a c.b.l.f. coproduct of c.b.l.f. coproducts yields again a c.b.l.f. coproduct. 

\begin{defn}
An $Ab$-enriched, strictly $\bZ^n$-graded category admitting all c.b.l.f. coproducts is called \emph{c.b.l. additive} if all its c.b.l.f. coproducts are biproducts. In that case, we write c.b.l.f. coproducts with a $\bigoplus$ sign, and refer to them as \emph{c.b.l.f. direct sums}.
\end{defn}

Because a c.b.l. additive category $\cC$ admits c.b.l.f. direct sums, its Grothendieck group $K_0(\cC)$ is naturally a $\bZ\pp{x_1, \dots, x_n}$-module with the action of $\sum_{\bg \in C} a_\bg x^{\be + \bg} \in \bZ\pp{x_1, \dots, x_n}$ defined as 
\begin{equation}\label{eq:ZppactiononK0}
\sum_{\bg \in C} a_\bg x^{\be + \bg} [X] := [\bigoplus_{\bg \in C} x^{\bg + \be} X^{\oplus a_\bg}].
\end{equation}
Note that when we write $\bg \in C$ in this context, we implicitly suppose we are considering $\bg \in C \cap \bZ^n$ and $C$ compatible with $\prec$. We will abuse of this notation whenever it is clear from the context.

\begin{defn}
We say a that a c.b.l. additive category $\cC$ is \emph{c.b.l. Krull--Schmidt} if every object decomposes  into a c.b.l.f. direct sum of compact objects having local endormophism rings. 
\end{defn}

Note that in particular a c.b.l. Krull--Schmidt category is locally Krull--Schmidt. 
From a restriction of \cref{thm:cbluniquedecomp}, we easily obtain the following, which is a ``c.b.l.f. version'' of \cref{prop:krullschmidt}.

\begin{prop}\label{prop:cblkrullschmidt}
Let $\cC$ be a c.b.l. Krull--Schmidt category with a full collection of pairwise distinct indecomposable objects $\{X_i\}_{i \in I}$. The split Grothendieck group of $\cC$ is a free $ \bZ\pp{x_1, \dots, x_n}$-module 
\[
K_0^\oplus(\cC) \cong \bigoplus_{i \in I} \bZ\pp{x_1, \dots, x_n} \cdot [X_i],
\]
with basis $\{[X_i]\}_{i \in I}$.
\end{prop}

Consider a c.b.l. Krull--Schmidt category $\cC$  with collection of pairwise distinct indecomposable objects $\{X_i\}_{i \in I}$. We define
\[
K_{\be} := \bigoplus_{i \in I} V_{\be} \cdot [X_i],
\]
 for each $\be \in \bZ^n$ and where $V_{\be} \subset \bZ\pp{x_1, \dots, x_n}$ is as in~\eqref{eq:neighborhoodLaurent}. Then we endow~$K_0^\oplus(\cC)$ with a topology by using $\{K_{\be}\}_{\be \in \bZ^n}$ as basis of neighborhood of zero. We also refer to this topology as the $(x_1, \dots, x_n)$-adic topology. Note that the choice of $X_i$ does not matter for defining the topology.

\begin{prop}
The split Grothendieck group $K_0^\oplus(\cC)$ of a c.b.l. Krull--Schmidt category $\cC$ endowed with the $(x_1, \dots, x_n)$-adic topology is a Hausdorff topological module over $\bZ\pp{x_1, \dots, x_n}$.
\end{prop}

\begin{proof}
It follows from similar arguments as in the proof of \cref{prop:laurentistop}. We leave the details to the reader.
\end{proof}

We also obtain the following nice property:

\begin{prop}\label{prop:addfunctorinducemap}
Let $F : \cC \rightarrow \cC'$ be a $\bZ^n$-graded, additive functor between c.b.l. Krull--Schmidt categories. There is an induced continuous map 
\[
K_0^\oplus(\cC) \xrightarrow{[F]} K_0^\oplus(\cC') .
\]
\end{prop}

\begin{proof}
Let $\{X_i\}_{i \in I}$ and $\{X'_j \}_{j \in J}$ be two collections of indecomposables for $\cC$ and $\cC'$ respectively. 
We want to show that $F$ is locally finite so that we can apply \cref{prop:locfinFKS}. 
Since $M$ decomposes as a c.b.l.f. direct sum of $\{X_i\}_{i \in I}$, we can extract a finite subset $I_M \subset I$ such that $x^{\bg}X_i^*(M) = 0$ for all $\bg \in \bZ^n$ and $i \notin I_M$. Similarly, we take $I'_{FM} \subset I'$ such that ${x^{\bg}X_i'}^*(FM) = 0$ for all $\bg \in \bZ^n$ and $i \notin I'_{FM}$. Fix $i' \in I'_{FM}$, $\bg' \in \bZ^n$ and $i \in I_{M}$. Consider the subset $Z \in \bZ^{n}$ such that
\[
x^{\bg}X^*_i(M) x^{\bg'}{X'_j}^*(F(x^{\bg}X_i)) \neq 0,
\]
for all $\bg \in Z$. Since $F$ preserves the $\bZ^n$-grading, we obtain 
\[
 x^{\bg'}{X'_j}^*(F(x^{\bg}X_i)) = x^{\bg' - \bg}{X_j'}^*(FX_i).
\]
For $\bg \ll 0$, we have $x^{\bg}X^*_i(M) = 0$ since $M$ is cone bounded. Similarly for $\bg' - \bg \ll 0$, we have $x^{\bg' - \bg}{X_j'}^*(FX_i) = 0$. Thus, $Z$ is finite. 
\end{proof}

\subsubsection{An example with modules}

Here $\Bbbk$ is a field. 
We say that a $\bZ^n$-graded $\Bbbk$-vector space $M = \bigoplus_{\bg \in \bZ^{n}} M_\bg$  is \emph{c.b.l.f. dimensional} if $\dim M_{\bg} < \infty$ for all $\bg$, and if there exists a cone $C_M$ compatible with $\prec$, and $\be \in \bZ^n$, such that $M_\bg = 0$ whenever $\bg - \be \notin C_M$. We call such $C_M$ a \emph{bounding cone of~$M$} and $\be$ is the \emph{minimal degree} of $M$. Note that $M$ is c.b.l.f. dimensional if and only if
\[
\gdim(M) := \sum_{\bg \in \bZ^n} \dim(M_\bg) x^{\bg} \in x^{\be} \bZ_{C_M}\llbracket x_1, \dots, x_n \rrbracket \subset \bZ\pp{x_1,\dots,x_n},
\]
for some $\be \in \bZ^n$ and $C_M$ compatible with $\prec$. 
Similarly, a $\bZ^n$-graded $R$-module $M$ is \emph{c.b.l.f. generated} (cone bounded, locally finitely generated) if there is a collection of homogeneous elements $\{x_i\}_{i \in I} \subset M$ generating $M = \sum_{i \in I} Rx_i$ such that $d_{\bg} := \# \{x_i | \deg(x_i) = \bg\}  < \infty$ and $d_\bg = 0$ whenever $\bg - \be \notin C_M$.

\smallskip

Let $R$ be a c.b.l.f. dimensional, $\bZ^n$-graded $\Bbbk$-algebra. Consider the category $R\amod$ of $\bZ^n$-graded $R$-modules with degree zero maps. Let $R\modlf$ be the full subcategory of $R\amod$ consisting of c.b.l.f. dimensional modules. Similarly, $R\pmodlfg$ is the full subcategory of c.b.l.f. generated, projective $R$-modules. There are embeddings of categories:
\[
R\pmodlfg \subset R\modlf \subset R\amod.
\]

\begin{prop}
Both $R\pmodlfg$ and $R\modlf$ are c.b.l. additive categories. In addition, $R\pmodlfg$ is c.b.l. Krull--Schmidt.
\end{prop}

\begin{proof}
They clearly both have all c.b.l.f. coproducts.
For the first claim, we can suppose without loosing generality that we are in $R\modlf$.
 Let
\[
M = \coprod_{\bg \in \bZ^n} x^\bg \left( X_1^{\oplus k_{1,\bg}} \oplus \cdots \oplus X_m^{\oplus k_{m, \bg}} \right)
\]
be a c.b.l.f. coproduct. We will show it is a biproduct in $R\modlf$ (and actually even in $R\amod$).
Let $Z$ be an object in $R\modlf$ with arrows $f_{i,\bg} : Z \rightarrow x^\bg X_{i}^{\oplus k_{i,\bg}}$. If the number of non-zero arrows is finite, then we are in the classical case of additive categories and we are done. Therefore, we suppose that for each $\bg$ there is $\bg' \succ \bg$ and $i'$ such that $f_{i', \bg'} \neq 0$. 
Choose any homogeneous element $z \in Z$. We claim that for almost all $(i, \bg)$ then $f_{i,\bg}(z) = 0$. Indeed,  suppose that it does not hold. Then for each $\bg$ we can find $0 \neq f_{i', \bg'}(z) \in X_{i'}$ with $\bg' \succ \bg$. But $\deg(f_{i',\bg'}(z)) = \deg(z) - \bg'$ which is a contradiction since $X_{i'}$ is cone bounded. Therefore, $\prod_{(i,\bg)} f_{i, \bg}$ defines a map $Z \rightarrow \coprod_i Y_i$, and we get a bijection
\[
\Hom(Z, \coprod_i Y_i) \xrightarrow{\cong} \prod_i \Hom(Z, Y_i).
\]
Since $Z$ is arbitrary, this concludes the proof of the first claim.

\smallskip

For the second claim, we only need to show that endomorphism rings of indecomposable modules are local. 
Let $P := Re$ be an indecomposable projective module given by projection by a primitive idempotent $e \in R$. Then $\End(P) = R_0e$ is an indecomposable, finite-dimensional $R_0$-module. Take $f \in  R_0e$. By Fitting's lemma, $f$ is either $e$ or $0$, and therefore either $f$  or $f-e$ is invertible, implying that $\End(P)$ is local.
Moreover, since $R_0$ is finite-dimensional, there is only finitely many distinct indecomposable objects.
\end{proof}

\begin{cor}
Let $R$ be as above. We extract from the collection of primitive idempotents $\{e_i\}_{i \in I}$ a full sub-collection $\{e_j\}_{j \in J \subset I} \subset \{e_i\}_{i \in I}$ of non-equivalent idempotents. Then we obtain
\[
K_0^\oplus(R\pmodlfg) \cong \bigoplus_{j \in J} \bZ\pp{x_1, \dots, x_n} \cdot [Re_j].
\]
\end{cor}

\begin{rem}
In general, one can construct products in a category of graded modules by taking the direct sum of the products degree-wise. More explicitly, let $\{M_i\}_{i \in I}$ be a collection of $\bZ^n$-graded $R$-modules, then we put
\begin{align*}
\prod_{i \in I} M_i  &:= \bigoplus_{\bg \in \bZ^n} \left(\prod_{i \in I} M_i\right)_\bg,
&
\left(\prod_{i \in I} M_i\right)_\bg &:= \prod_{i\in I} (M_i)_\bg.
\end{align*}
If we take a c.b.l.f. product of c.b.l.f. modules, then each homogeneous part of the product is a finite product. In particular, the product coincides with the coproduct, as expected. 
\end{rem}

\begin{exe}
Recall the $\bZ \times \bZ^2$-graded algebra $R$ from \cref{sec:mainexintro} and its simple module $L$. We choose the additive order on $\bZ \times \bZ^2$ given by $(0,0,0) \prec (0,1,0) \prec (0,0,1) \prec (1,0,0)$. Therefore, it means that if we identify $h$ with $(1,0,0)$, $q$ with $(0,1,0)$ and $\lambda$ with $(0,0,1)$, then we have
\begin{align*}
\frac{1 + h \lambda^2}{1-q^2} &= (1+h\lambda^2)(1+q^2+q^4+\cdots), \\
\frac{1-q^2}{1+h\lambda^2} &= (1-q^2)(1-h\lambda^2 +h^2\lambda^{4} - \cdots),
\end{align*}
in $\bZ\pp{h,q,\lambda}$. 
We observe that $R$ is a c.b.l.f. dimensional $\bZ \times \bZ^2$-graded $\Bbbk$-algebra. Thus, we have
\[
K_0^\oplus(R\pmodlfg) \cong \bZ\pp{h,q,\lambda} \cdot [R].
\]
However, note that $L \notin R\pmodlfg$, and thus \cref{eq:changeofbasisintro} does not make sense in this context.
\end{exe}


%% file: sections/JHTopK0.tex

\section{Asymptotic Grothendieck group for abelian categories}\label{sec:topG0}

We refine the notions presented in~\cite[A.2]{naissevaz1}. 

Let $\cC$ be an abelian category $\cC$ and suppose there is a sequence of arrows
\[
X = F_0 \xrightarrow{f_0} F_1 \xrightarrow{f_1} F_2 \xrightarrow{f_2} \cdots \xrightarrow{f_m}  F_{m+1} = Y.  
\]
For each $f_i$ there is an exact sequence
\[
0 \rightarrow \ker{f_i} \rightarrow F_i \xrightarrow{f_i} F_{i+1} \rightarrow \cok{f_i} \rightarrow 0,
\]
so that in $G_0(\cC)$ we obtain
\[
[F_{i+1}] - [F_i] = [\cok{f_i}] - [\ker{f_i}],
\]
and thus
\[
[Y] - [X] = \sum_{i = 0}^m  [\cok{f_i}] - [\ker{f_i}].
\]
We would like to take this phenomenon to the limit. For example, if there is an object $Y \in \cC$ admitting a filtration
\[
X \subset \cdots \subset F_r \subset F_{r-1} \subset \cdots \subset F_1 \subset F_0 = Y,
\]
with $X \cong \varprojlim_{i} X_i$, then we would like to have $[Y] = [X] + \sum_{r \geq 0} [F_r/F_{r+1}]$. However, the usual notion of Grothendieck group does not allow to have such equalities, thus we need to take a quotient of it. 
Before doing so, let us introduce the notation
\[
\sum_{i \in I} [X_i] := [\bigoplus_{i \in I}  X_i] \in G_0(\cC),
\]
whenever $\bigoplus_{i \in I} X_i \in \cC$. 
 We use this notation distributively with respect to the addition and subtraction in $G_0(\cC)$. 
 
 \subsection{Derived limits and colimits}
 
The \emph{filtered limit }
 \[
X := \lim\bigl( \cdots \xrightarrow{f_1} X_1 \xrightarrow{f_0} X_0\bigr),
 \]
 is the universal object in $\cC$ with canonical maps $\imath_r : X \rightarrow X_r$ such that
 \[
 \begin{tikzcd}
 X_{r+1} \ar{rr}{f_r}&& X_{r} \\
 &X \ar[swap]{ur}{\imath_{r+1}} \ar{ul}{\imath r}&
 \end{tikzcd}
 \]
 commutes for all $r \geq 0$. The filtered colimit is the dual. In general, filtereted limits (resp. filtered colimits) only preserve kernels (resp. cokernels). Recall Grothendieck axiom's \cite{grothendieck} for an abelian category $\cA$:
 \begin{itemize}
 \item $\cA$ is AB3 (resp. AB3*) if it admits arbitrary coproducts (resp. products);
 \item $\cA$ is AB4 (resp. AB4*) if it is AB3 (resp. AB3*) and coproducts (resp. products) preserve short exact sequences;
 \item $\cA$ is AB5 (resp. AB5*) if it is AB3 (resp. AB3*) and filtered colimits (resp. limits) preserve short exact sequences.
 \end{itemize}
 
 \smallskip
 
 In an AB3* category, one can construct the filtered limit
\[
X := \lim \bigl( \cdots X_2 \xrightarrow{f_1} X_1 \xrightarrow{f_0} X_0\bigr),
\]
as the kernel fitting in the following exact sequence:
\begin{equation} \label{eq:limaskernel}
0 \rightarrow  X \rightarrow \prod_{r\geq 0} X_r \xrightarrow{1-f_\bullet} \prod_{r \geq 0} X_r \rightarrow \tilde X \rightarrow 0,
\end{equation}
where
\[
1 - f_\bullet :=
\begin{pmatrix}
1      & -f_0       &  0 & 0 & \cdots \\
0 & 1       & -f_1 & 0 & \cdots  \\
0      & 0  & 1 & -f_2  & \cdots \\
\vdots & \ddots & \ddots & \ddots & \ddots
\end{pmatrix}
\]
The cokernel $\tilde X =: \lim^1$ is called the \emph{derived limit}, and is in general non-zero. Yet, whenever $\cA$ is AB4* and $\tilde X \cong 0$, then the filtered limit preserve cokernels. More precisely, given a short exact sequence of filtered limits
\[
0 \rightarrow X_\bullet \rightarrow Y_\bullet \rightarrow Z_\bullet \rightarrow 0,
\]
then one obtains a long exact sequence by the snake lemma
\[
\begin{tikzcd}
0 \ar{r} & X \ar{r} & Y \ar{r} & Z \ar{lld} & \\
 &\tilde X \ar{r} & \tilde Y \ar{r} & \tilde Z \ar{r} & 0.
\end{tikzcd}
\]
Therefore, there is a short exact sequence $0 \rightarrow X \rightarrow Y \rightarrow Z \rightarrow 0$ whenever $\tilde X \cong 0$. 

\smallskip

Dually, if the category is AB3, one obtains that the filtered colimit
\[
Y := \colim \bigl( Y_0 \xrightarrow{g_0} Y_1 \xrightarrow{g_1} \cdots \bigr),
\]
is the cokernel in the following short exact sequence:
\[
0 \rightarrow \tilde Y \rightarrow \coprod_{r\geq 0} Y_r \xrightarrow{1-g_\bullet}   \coprod_{r\geq 0} Y_r \rightarrow Y \rightarrow 0.
\]
The colimit preserves kernels whenever the derived colimit ${\colim}^1 := \tilde Y$ is zero and $\cA$ is AB4. 

\smallskip

In a category of abelian groups
, then ${\colim}^1$ is always zero, and filtered colimits commute with kernels (the category is actually AB5). 
There exists a condition that ensures a limit preserves cokernels is such a category:

\begin{defn}
A filtered limit 
\[
X := \lim \bigl( \cdots  \xrightarrow{f_1} X_1 \xrightarrow{f_0} X_0\bigr),
\]
is said to respect the \emph{Mittag--Leffler condition} if for all $i \geq 0$ there exists $k_i > i$ such that
\[
\Image(X_r \rightarrow X_i) = \Image(X_{k_i } \rightarrow X_i),
\]
for all $r \geq k_i$. Dually, a filtered colimit respects the Mittag--Leffler condition if the coimages stabilize. 
\end{defn}

\begin{defn}
We say that an abelian category is \emph{Mittag--Leffler} if all filtered limits (resp. colimits) that respect the Mittag--Leffler condition preserve cokernels (resp. kernels).
\end{defn}

For example, categories of abelian groups (and consequently of modules) 
 are Mittag--Leffler~\cite[Proposition 3.5.7]{weibel}.

\begin{rem}
As pointed out by Neeman~\cite{neemancounterex}, not all abelian categories are Mittag--Leffler.
\end{rem}

\begin{lem}\label{lem:limepiisML}
Consider a filtered limit 
\[
\lim \bigl( \cdots  \overset{f_1}\twoheadrightarrow X_1 \overset{f_0}\twoheadrightarrow X_0\bigr).
\]
If all $f_r$ are epimorphisms, then the filtered limit respects the Mittag--Leffler condition.
\end{lem}

\begin{proof}
We have $\Image(X_r \twoheadrightarrow X_\ell) \cong X_\ell$ for all $r > \ell \geq 0$.
\end{proof}

\subsection{Asymptotic Grothendieck group}

In this subsection, we assume that $\cC \subset \cA$ is a full subcategory of some abelian category $\cA$ being both AB4 and AB4*. Therefore a filtered limit (resp. colimit) of cokernels (resp. of kernels) is a cokernel (resp. a kernel) if and only if the derived limit (resp. derived colimit) of the corresponding kernels (resp. cokernels) vanishes. 

\begin{defn}
We say that a filtered limit, derived limit, filtered colimit or derived colimit of $\{f_r\}_{r \geq 0}$ is \emph{acceptable} if $\bigoplus_{r \geq 0} \cok f_r \in \cC$ and $\bigoplus_{r \geq 0} \ker f_r \in \cC$. 
\end{defn}

\begin{defn}\label{def:topG0}
The \emph{asymptotic Grothendieck group} of $\cC$ is 
\[
\bGO(\cC) := G_0(\cC)/\bigl(J(\cC) + J^*(\cC)\bigr),
\]
where $J(\cC)$ is generated by
\[
[Y] - [\tilde Y]  = [X] + \sum_{r \geq 0} \bigl( [C_r] - [K_r] \bigr), 
\]
for $Y, \tilde Y, X, \bigoplus_{r \geq 0} C_r, \bigoplus_{r \geq 0} K_r \in \cC$,
whenever 
\begin{align*}
Y \cong \colim \bigl(X \xrightarrow{f_0} Y_1 \xrightarrow{f_1} Y_2 \xrightarrow{f_2} \cdots \bigr), \\
\tilde Y \cong {\colim}^1 \bigl(X \xrightarrow{f_0} Y_1 \xrightarrow{f_1} Y_2 \xrightarrow{f_2} \cdots \bigr), 
\end{align*}
are acceptable,
and  
\begin{align}\label{eq:locGOasympt}
[K_r] &= [\ker f_r] \in G_0(\cC), & [C_r] &= [\cok f_r] \in G_0(\cC), 
\end{align}
for all $r \geq 0$. Dually, $J^*(\cC)$ is generated by 
\[
[X] - [\tilde X] = [Y] + \sum_{r \geq 0} \bigl( [K_r] - [C_r] \bigr),
\]
whenever again all objects and direct sums are in $\cC$ and 
\begin{align*}
X \cong \lim \bigl( \cdots \xrightarrow{f_2}  X_2 \xrightarrow{f_1} X_1  \xrightarrow{f_0} Y \bigr), \\
\tilde X \cong {\lim}^1 \bigl(   \cdots \xrightarrow{f_2}  X_2 \xrightarrow{f_1} X_1  \xrightarrow{f_0} Y\bigr),
\end{align*}
are acceptable, 
and we have \cref{eq:locGOasympt}.
\end{defn}


\subsection{Locally Jordan--H\"older categories}

%

Let $X$ be an object in an abelian category $\cC$. We call $\bN$-filtration a sequence of sub-objects $X_i \hookrightarrow X$ for all $i \in \bN$ such that $X_{i+1} \hookrightarrow X_{i}$ and $X_0 = X$. We write it as
\[
\cdots \subset X_r \subset X_{r-1} \subset \cdots \subset X_1 \subset X_0 = X.
\]
Such a filtration is called \emph{Hausdorff} if the inverse limit $\varprojlim_{i} X_i \cong 0$ (i.e. the filtered limit on the inclusion maps is zero).  We say it has simple quotients if all quotient objects $X_i/X_{i+1}$ are either trivial or simple. A Hausdorff, $\bN$-filtration with simple quotients is called an \emph{$\bN$-composition series}, and we refer to the non-trivial quotient objects as \emph{composition factors}. 

\begin{defn}
Let $\{S_i\}_{i \in I}$ be a collection of 
 objects in an abelian category $\cC$. Consider an $\bN$-filtration
\[
\cdots \subset X_{r+1} \subset X_{r} \subset \cdots \subset X_1 \subset X_0 = X,
\]
such that each $X_{r}/X_{r+1} \cong  S_i$ for some $i \in I$. If $\kappa_{i}:= \#\{ r \in \bN | X_r/X_{r+1}  \cong S_i\} < \infty$, then we say that the filtration is \emph{locally finite}, and we refer to $\kappa_{i}$ as the \emph{multiplicity} of $S_i$. 
\end{defn}

Clearly, in general in an abelian category, we cannot hope to always have $\bN$-composition series. Moreover, even if we have such infinite composition series, they do not have to be essentially unique. It was given in~\cite[Definition A.20]{naissevaz1} one condition that ensures the uniqueness of such a filtration:

\begin{defn}
We say that an object $X$ in an abelian category $\cC$ is \emph{stable for the filtrations} if for all pairs of $\bN$-composition series
\begin{align*}
 0 \subset \cdots \subset X_1 \subset X_0 &= X, \\
 0 \subset \cdots \subset X'_1 \subset X'_0 &= X,
\end{align*}
then for each $i \in \bN$ there exists $k \in \bN$ such that $X'_k \subset X_i$.
\end{defn}

\begin{thm}
\label{thm:uniquefilt}
Suppose $X \in \cC$ admits two locally finite, $\bN$-composition series
\begin{align*}
0 \subset \cdots \subset X_1 \subset X_0 &= X, \\
0 \subset \cdots \subset X'_1 \subset X'_0 &= X, 
\end{align*}
with respective multiplicities $\kappa_{i}$ and $\kappa'_{j}$, and collections of  pairwise non-isomorphic compositions factors $\{S_i\}_{i\in I}$ and $\{S_j\}_{j \in J}$. 
If $X$ is stable for the filtrations, then there exists a bijection $\sigma : I \rightarrow J$ such that $S_i \cong S_{\sigma(j)}$, and $\kappa_i = \kappa'_{\sigma(i)}$.
\end{thm}

\begin{proof}
The proof is essentially the same as \cite[Proposition A.25]{naissevaz1}, and we refer to the reference for the details. 
\end{proof}



%


\begin{defn}
We say that an abelian category $\cC$ is \emph{locally AB5} if acceptable filtered colimits preserve kernels, that is if whenever
\begin{align*}
X &\cong \colim (X_0 \xrightarrow{f_0} X_1 \xrightarrow{f_1} \cdots), &
Y &\cong \colim (Y_0 \xrightarrow{g_0} Y_1 \xrightarrow{f_1} \cdots),
\end{align*}
are acceptable filtered colimits, and we have a collection of commuting maps $u_\bullet : X_\bullet \rightarrow Y_\bullet$ then
\[
\ker (X \xrightarrow{u} Y) \cong \colim \bigl( \ker(X_0 \xrightarrow{u_0} Y_0) \rightarrow \ker(X_1 \xrightarrow{u_1} Y_1) \rightarrow  \cdots \bigr).
\]
Dually, we define the notion of \emph{locally AB5*} where acceptable filtered limits preserve cokernels. 
\end{defn}

From now on, we assume again that $\cC \subset \cA$ is a full subcategory of some abelian category $\cA$ being both AB4 and AB4*. 
In particular, $\cC$ is locally AB5 (resp. AB5*) if and only if all of its acceptable derived colimits (resp.acceptable derived limits) vanish.

\begin{defn}\label{def:locJH}
Let $\cC$ be both a locally AB5 and a locally AB5* category. 
We say that $\cC$ is \emph{locally Jordan--H\"older} if each object in $\cC$ is stable for the filtration and admits a locally finite composition series. 
\end{defn}

\begin{rem}
Note that a locally Jordan--H\"older category admitting arbitrary biproducts (and thus limits and colimits) would be trivial since it would satisfies both AB5 and AB5*.
\end{rem}

\begin{exe}
Consider the category $\Bbbk\gmod$ of locally finite dimensional $\bZ$-graded $\Bbbk$-vector spaces. Since it is a subcategory of the whole category of graded $\Bbbk$-vector spaces, which is AB5, $\Bbbk\gmod$ is locally AB5. 
For the dual,
because each $X_i$ is locally finite dimensional, the limit respects a Mittag--Leffler-like condition in each degree. It allows to easily prove that $1-f_\bullet$ in \cref{eq:limaskernel} is surjective (see \cite[Proposition 4.2.3]{kochman} for the classical case), recalling that the product of graded spaces is given by direct sum of the product in each degree. 
Thus, the derived limit is zero, and the filtered limit preserves cokernels. 
\end{exe}


Let $\cC$ be a locally Jordan--H\"older category and $S \in \cC$ be a simple object. By \cref{thm:uniquefilt} we have a map 
\[
S^* : \cC \rightarrow \bN, \quad X \mapsto S^*(X)
\]
that counts the multiplicity of $S$ in the essentially unique $\bN$-composition series of $X$. 
By \cite[Proposition A.31]{naissevaz1}, we know that whenever there is a short exact sequence
\[
0 \rightarrow X \rightarrow Y \rightarrow Z \rightarrow 0,
\]
in $\cC$, then $S^*(Y) = S^*(X) + S^*(Z)$. Therefore, we obtain an induced map
\[
S^* : G_0(\cC) \rightarrow \bN. 
\]

\begin{lem}\label{lem:Sstarcolim0}
Let $\cC$ be a locally Jordan--H\"older category. 
Consider an acceptable filtered colimit
\[
X \cong \colim( X_0 \xrightarrow{f_0} X_1 \xrightarrow{f_1} \cdots ).
\]
If $S^*(X_r) = 0$ for all $r \geq 0$, then $S^*(X) = 0$. 
\end{lem}

\begin{proof}
Suppose by contradiction that $S^*(X) > 0$, and thus that there exists $Y \subset Z \subset X \in \cC$ such that
\[
0 \rightarrow Y \rightarrow Z \rightarrow S \rightarrow 0,
\]
is a short exact sequence. Define $\tilde X_r := \Image(\psi_r) \subset X$ where $\psi_r : X_r \rightarrow X$ is the canonical inclusion map. 
We obtain a commutative diagram 
\[
\begin{tikzcd}
0 \ar{r} & Y \ar{r} & Z \ar{r} & S \ar{r} & 0 \\
0 \ar{r} & Y \cap \tilde X_r \ar{r} \ar[hookrightarrow]{u} & Z \cap \tilde X_r \ar{r}  \ar[hookrightarrow]{u}  & \tilde S_r \ar{r}  \ar[swap,hookrightarrow]{u}{s}  & 0,
\end{tikzcd}
\]
where both rows are exact, 
and the three vertical arrows are monomorphisms. 
Because $S$ is simple, $s$ is either $0$ or an epimorphism. Thus  $\tilde S_r$ is either $0$ or isomorphic to $S$.

Suppose there exists $r \geq 0$ such that $\tilde S_r \cong S$. Then we have $S^*(\tilde X_r) > 0$. Moreover, there is a short exact sequence
\[
0 \rightarrow \ker \psi_r \rightarrow X_r \rightarrow \tilde X_r \rightarrow 0,
\]
so that $S^*(X_r) = S^*(\ker \psi_i) + S^*(\tilde X_r) > 0$. This contradicts the fact that $S^*(X_r) = 0$. 

Because acceptable filtered colimits are exact, we obtain a short exact sequence
\[
0 \rightarrow \colim_{r \geq 0} (\tilde X_r \cap Y) \cong Y \rightarrow \colim_{r \geq 0} (\tilde X_r \cap Z) \cong Z \rightarrow 0 \rightarrow 0,
\]
but this is a contradiction. In conclusion, $S^*(X) = 0$.  
\end{proof}

\begin{lem}\label{lem:Sstarcolim}
Let $\cC$ be a locally Jordan--H\"older category. 
Consider an acceptable filtered colimit
\[
X \cong \colim( X_0 \xrightarrow{f_0} X_1 \xrightarrow{f_1} \cdots ).
\]
If $S^*(\cok f_r) = S^*(\ker f_r) = 0$ for all $r \geq 0$, then 
$
S^*(X) = S^*(X_0).
$
\end{lem}

\begin{proof}
Let $\psi_0 : X_0 \rightarrow X$ be the canonical map. 
The exact sequence 
\[
0 \rightarrow \ker(\psi_0) \rightarrow X_0 \rightarrow X \rightarrow \cok(\psi_0) \rightarrow 0,
\]
shows that $S^*(X) - S^*(X_0) = S^*(\cok(\psi_0)) - S^*(\ker(\psi_0))$. We will show that $S^*(\cok \psi_0) = S^*(\ker \psi_0) = 0$. 
We construct a commutative diagram
\[
\begin{tikzcd}[column sep = 0ex]
&\ker f_0 \ar[hookrightarrow]{dr}  \ar[hookrightarrow]{dl} &&&&\cok f_1& 
\\
\ker f_{01} \ar[hookrightarrow]{rr}   \ar[twoheadrightarrow]{dr} &&X_0 \ar{rr}{f_{01}} \ar[swap]{dr}{f_0} &&X_2 \ar[twoheadrightarrow]{ur}  \ar[twoheadrightarrow]{rr} && \cok f_{01}  \ar[twoheadrightarrow]{ul}
\\
&K  \ar[hookrightarrow]{dr}  \ar[hookrightarrow]{rr} &&X_1  \ar[swap]{ur}{f_1}  \ar[twoheadrightarrow]{rr}  \ar[twoheadrightarrow]{dr} &&C  \ar[hookrightarrow]{ur}&
\\
&&\ker f_1  \ar[hookrightarrow]{ur} &&\cok f_0  \ar[twoheadrightarrow]{ur} &&
\end{tikzcd}
\]
where $K := \Image(\ker f_{01} \rightarrow X_1)$ and $C := \Image(X_1 \rightarrow \cok f_{01})$.
Since $K \subset \ker f_1$ we have $S^*(K) \leq S^*(\ker f_1) = 0$. Moreover, the sequence
\[
0 \rightarrow 
\ker f_0  \hookrightarrow \ker f_{01} \twoheadrightarrow K \rightarrow 0,
\]
is exact, so that $S^*(\ker f_{01}) = S^*(\ker f_0) + S^*(K) = 0$. Similarly, we obtain $S^*(\cok f_{01}) = 0$. 
Applying the same reasoning recursively with $X_0 \rightarrow X_{r} \rightarrow X_{r+1}$ shows that $S^*(\cok(X_0 \rightarrow X_r)) = S^*(\ker(X_0 \rightarrow X_r)) = 0$. 
Moreover, since $\cC$ is locally AB5, we have
\[
\ker \psi_0 \cong \colim_{r\geq 0}(\ker(X_0 \rightarrow X_r)).
\]
Therefore, by \cref{lem:Sstarcolim0} we obtain $S^*(\ker \psi_0) = 0$. Similarly, we have $S^*(\cok \psi_0) = 0$, concluding the proof. 
 \end{proof}

\begin{lem}\label{lem:Sstarlim0}
Let $\cC$ be a locally Jordan--H\"older category. 
Consider an acceptable filtered limit
\[
X \cong \lim( \cdots \xrightarrow{f_1} X_1 \xrightarrow{f_0} X_0 ).
\]
If $S^*(X_r) = 0$ for all $r \geq 0$, then
$S^*(X) = 0$.
\end{lem}

\begin{proof}
Suppose by contradiction that $S^*(X) > 0$, and thus there exists $Y \subset Z \subset X$ such that
\[
0 \rightarrow Y \rightarrow Z \rightarrow S \rightarrow 0, 
\]
is a short exact sequence. Let $\phi_r : X \rightarrow X_i$ be the canonical map. 
We write $\tilde Y_r := \Image(Y \rightarrow X \xrightarrow{\phi_r} X_r)$, and similarly for $\tilde Z_r$. We obtain a commutative diagram
\[
\begin{tikzcd}
0 \ar{r} & Y \ar{r} \ar[twoheadrightarrow]{d} & Z \ar{r}  \ar[twoheadrightarrow]{d} & S \ar{r}  \ar[twoheadrightarrow]{d}{s} & 0 \\
0 \ar{r} & \tilde Y_r \ar{r} & \tilde Z_r \ar{r} & \tilde S_r \ar{r} & 0
\end{tikzcd}
\]
where the rows are exact and the vertical maps are epimorphisms. Because $S$ is simple, $\tilde S_r$ is either $0$ or isomorphic to $S$.
If $\tilde S_r \cong S$ then the filtration
\[
\tilde Y_r \subset \tilde Z_r \subset X_r,
\]
gives $S^*(X_r) > 0$. Thus $\tilde S_r \cong 0$ for all $r \geq 0$.  
Then, because $\cC$ is locally AB5* we obtain an exact sequence
\[
\begin{tikzcd}
0 \ar{r} & \lim_{r\geq 0} \tilde Y_r \cong Y  \ar{r} & \lim_{r\geq 0} \tilde Z_r \cong Z \ar{r} & 0 \ar{r} & 0,
\end{tikzcd}
\]
which is a contradiction. 
\end{proof}

\begin{lem}\label{lem:Sstarlim}
Let $\cC$ be a locally Jordan--H\"older category. 
Consider an acceptable filtered limit
\[
X \cong \lim( \cdots \xrightarrow{f_1} X_1 \xrightarrow{f_0} X_0 ).
\]
If $S^*(\cok f_r) = S^*(\ker f_r) = 0$ for all $r \geq 0$, then 
$
S^*(X) = S^*(X_0).
$
\end{lem}

\begin{proof}
The proof is dual to the one of \cref{lem:Sstarcolim}, replacing \cref{lem:Sstarcolim0} with \cref{lem:Sstarlim0}, and omitted.
\end{proof}

%

\begin{lem}\label{lem:JHS0trivial}
Let $\cC$ be a locally Jordan--H\"older category and $X \in \cC$. If $S^*(X) = 0$ for all simple object $S \in \cC$, then $X \cong 0$. 
\end{lem}

\begin{proof}
If $S^*(X) = 0$ for all $X \in \cC$, then it means $X$ has no non-trivial subobject, thus $X \cong 0$. 
\end{proof}

\begin{lem}\label{lem:JHsameSiso}
Let $\cC$ be a locally Jordan--H\"older category. Consider a morphism
\[
f : X \rightarrow Y.
\]
If $f$ is a monomorphism or an epimorphism and if $S^*(X) = S^*(Y)$ for all simple object $S \in \cC$, then $f$ is an isomorphism. 
\end{lem}

\begin{proof}
Suppose $f$ is a monomorphism, then we have a short exact sequence
\[
0 \rightarrow X \rightarrow Y \rightarrow \cok f \rightarrow 0.
\]
Therefore, we have $S^*(Y) = S^*(X) + S^*(\cok f)$ for all simple object $S \in \cC$. Thus, $S^*(\cok f) = 0$, and $\cok f = 0$ by \cref{lem:JHS0trivial}. The case where $f$ is an epimorphism is similar.
\end{proof}

\begin{thm}\label{thm:locallyJHG0}
Let $\cC$ be a locally Jordan--H\"older category with a full collection of pairwise non-isomorphic simple objects $\{S_i\}_{i \in I}$. If $\cC$ admits locally finite direct sums of $\{S_i\}_{i \in I}$, then its asymptotic Grothendieck group is a free $\bZ$-module
\[
\bGO(\cC) \cong \prod_{i \in I} \bZ \cdot[S_i],
\]
with basis $\{[S_i]\}_{i \in I}$.
Moreover, we have
\[
[X] = \sum_{i\in I} S_i^*(X) [S_i] \in \bGO(\cC),
\]
for all $X \in \cC$.
\end{thm}

\begin{proof}
We claim that $J(\cC) \subset \ker S^*$ and $J^*(\cC) \subset \ker S^*$,  so that there is an induced map $S^* : \bGO(\cC) \rightarrow \bN$. In particular any relation obtained from a filtered limit or colimit in $ \bGO(\cC)$ can be obtained from a relation obtained by decomposing everything in sums of $[S_i]$. Thus $J(\cC) + J^*(\cC)$ is generated by the relation $[X] =  \sum_{i \in I} S_i^*(X)[S_i]$. 

We now prove our claim for $J(\cC) \subset \ker S^*$, the dual case being similar and left to the reader. Consider an acceptable filtered colimit
\[
Y \cong \colim \bigl( X = F_0 \xrightarrow{f_0} F_1 \xrightarrow{f_1} \cdots \bigr),
\]
in $\cC$ and $\{C_r, K_r \in \cC\}_{r\geq 0}$, such that $\bigoplus_r C_r, \bigoplus_r K_r \in \cC$, and $[C_r] = [\cok f_r]$, $[K_r] = [\ker f_r]$ for all $r\geq 0$. We want to show that
\begin{equation}\label{eq:SstarGO}
S^*(Y) - S^*(X) = S^*(\bigoplus_{r\geq 0} C_r) - S^*(\bigoplus_{r \geq 0} K_r).
\end{equation}
We have $S^*(\bigoplus_r C_r) = \sum_r S^*(C_r) \in \bN$, and thus $S^*(C_r) = 0$ for $r \gg 0$. The same applies for $\ker f_r, \cok f_r$ and $K_r$.
 Therefore, we obtain that
\begin{equation}\label{eq:SstartcoprodCr}
S^*(\bigoplus_r C_r) = S^*(\bigoplus_r \cok f_r), \quad S^*(\bigoplus_r K_r) = S^*(\bigoplus_r \ker f_r).
\end{equation}
Moreover, still for $r \gg 0$, we have that 
\begin{equation}\label{eq:SstarYFr}
S^*(Y) = S^*(F_r),
\end{equation}
by \cref{lem:Sstarcolim}, since
 \[
 Y \cong \colim\bigl(F_r \xrightarrow{f_{r}} F_{r+1} \xrightarrow{f_{r+1}} \cdots \bigr). 
 \]
We also have
\begin{equation}\label{eq:SstarXFr}
S^*(F_r) - S^*(X) = \sum_{i = 0}^{r-1} S^*(\cok f_i) - S^*(\ker f_i).
\end{equation}
Putting together \cref{eq:SstartcoprodCr}, \cref{eq:SstarYFr} and \cref{eq:SstarXFr} gives \cref{eq:SstarGO}.
\end{proof}

\begin{cor}
Let $\cC$ be a finite length abelian category. If $\cC$ is Mittag--Leffler, then the projection induces an isomorphism
$
G_0(\cC) \cong \bGO(\cC).
$
\end{cor}

\begin{proof}
All filtered limits and colimits respect the Mittag-Leffler condition in a finite length abelian category (see \cref{lem:finitelengthlim}). Thus, $\cC$ is locally Jordan--H\"older. 
\end{proof}

\begin{defn}
Let $\cC$ and $\cC'$ be two locally Jordan--H\"older categories with full collections of simple objects $\{S_i\}_{i \in I} \subset \cC$ and $\{S_j'\}_{j \in J}$.
We say that a functor $F : \cC \rightarrow \cC'$ is \emph{locally finite} if for all $X \in \cC$ and $j \in J$, there exists a finite subset $I_{X,j} \subset I$ such that
\[
S_i^*(X){S'_j}^* (FS_i) = 0,
\]
for all $i \notin I_{X,j}$.
\end{defn}

\begin{prop}\label{prop:lfexactinducemap}
Let $F : \cC \rightarrow \cC'$ be an exact functor  between two locally Jordan--H\"older categories with collections of simple objects $\{S_i\}_{i \in I} \subset \cC$ and $\{S_j'\}_{j \in J}$. If $F$ is locally finite, then there is an induced map
\[
[F] : \bGO(\cC) \rightarrow \bGO(\cC),
\]
given by $[F][X] := [F(X)]$. 
\end{prop}

\begin{proof}
We only need to show that for all $X \in \cC$ and $j \in J$ then
\[
{S'_j}^*(FX) = \sum_{i \in I} S_i^*(X){S'_j}^* (FS_i).
\]
Because $I_{X,j}$ is finite, we can find $X'$ such that there is a short exact sequence
\[
0 \rightarrow X' \rightarrow X \rightarrow X/X' \rightarrow 0,
\]
and $S_i^*(X') = 0$ for all $i \in I_{X,j}$. 
Moreover, if we write $\{L_r\}_{r \in R}$ for the compositions factors of $X/X'$, then we have
\[
{S_j'}^*(FX) = {S_j'}^*(FX') + \sum_{r \in R} {S_j'}^*(FL_r),
\]
since $F$ is exact and $R$ is finite. 
Moreover, for any $r \in R$ such that $L_r \cong S_i$ for $i \notin I_{X,j}$, then ${S_j'}^*(FL_R) = {S_j'}^*(FS_i) = 0$. Thus 
\[
\sum_{r \in R} {S_j'}^*(FL_R) =  \sum_{i \in I} S_i^*(X){S'_j}^* (FS_i).
\]
Finally, because $X' \subset X$ we have $I_{X',j} \subset I_{X,j} \setminus I_{X/X',j}$, and thus$I_{X',j} = \emptyset$. In particular, we obtain ${S_j'}^*(FX') = 0$. 
\end{proof}

\subsection{C.b.l. Jordan--H\"older categories}


Let $\cC$ be both a strictly $\bZ^n$-graded abelian category and a c.b.l. additive category. Its Grothendieck group $G_0(\cC)$ is a $\bZ\pp{x_1,\dots,x_n}$-module through~\eqref{eq:ZppactiononK0}. 
Mimicking the definition of a c.b.l.f. coproduct in a $\bZ^n$-graded category, there is an obvious notion of c.b.l.f. filtration.

\begin{defn}
Let $\{S_1, \dots, S_m\}$ be a finite collection of 
 objects in a $\bZ^n$-graded abelian category $\cC$. Consider an $\bN$-filtration
\[
\cdots \subset X_{r+1} \subset X_{r} \subset \cdots \subset X_1 \subset X_0 = X,
\]
such that each $X_{r}/X_{r+1} \cong x^\bg S_i$ for some $i \in \{1, \dots, m\}$ and $\bg \in \bZ^n$. If $\kappa_{i,\bg}:= \#\{ r \in \bN | X_r/X_{r+1}  \cong x^\bg S_i\} < \infty$, then we say that the filtration is \emph{graded locally finite}, and we refer to $\kappa_{i,\bg}$ as the \emph{degree $\bg$ multiplicity} of $S_i$.
We call such filtration a \emph{c.b.l.f. filtration} if in addition, there exists a cone $C$ compatible with $\prec$, and $\be \in \bZ^n$ such that $\kappa_{i,\bg} = 0$ whenever $\bg - \be \notin C$.
\end{defn}

\begin{defn}
An object $X$ admits a \emph{c.b.l.f. composition series} if it admits an $\bN$-composition series which is a c.b.l.f. filtration.
A locally Jordan--H\"older category $\cC$ being c.b.l. additive is called \emph{c.b.l.  Jordan--H\"older} if each object 
admits a c.b.l.f. composition series. 
\end{defn}

By the same proof as \cref{thm:uniquefilt} we obtain the following:

\begin{cor}\label{cor:topG0}
Let $\cC$ be a c.b.l. Jordan--H\"older category with a full collection of pairwise distinct simple objects $\{S_j\}_{j \in J}$. 
The asymptotic Grothendieck group of $\cC$ is a free $\bZ\pp{x_1, \dots, x_n}$-module 
\[
\bGO(\cC) \cong \bigoplus_{j \in J} \bZ\pp{x_1, \dots, x_n}  \cdot [S_j],
\]
with basis $\{[S_j]\}_{j \in J}$.
\end{cor}

Moreover, under the hypothesis of \cref{cor:topG0}, we endow $\bGO(\cC) $ with an $(x_1, \dots, x_n)$-adic topology, turning it into a Hausdorff topological $\bZ\pp{x_1, \dots, x_n}$-module, using 
\[
G_{\be} := \bigoplus_{j \in J} V_{\be} \cdot [S_j],
\]
as neighborhood basis of zero. 

\begin{prop}\label{prop:exactfunctoronG0}
Let $F : \cC \rightarrow \cC'$  be an exact functor between 
c.b.l. Jordan--H\"older categories. If $F$ commutes with the grading shift, then $F$ induces a continuous map
\[
\bGO(\cC) \xrightarrow{[F]} \bGO(\cC').
\]
\end{prop}

\begin{proof}
The proof is almost identical to the one of \cref{prop:addfunctorinducemap}, proving that $F$ is locally finite and using \cref{prop:lfexactinducemap}.
\end{proof}


\subsection{Strongly c.b.l. Jordan--H\"older categories}

We introduce a notion inspired by the definition of mixed abelian category with a tate twist (see for example~\cite{mixedcat}).

\begin{defn}
We say that a c.b.l. Jordan--H\"older category is \emph{strongly c.b.l. Jordan--H\"older} if there is only a finite number of distinct simple objects $\{S_1, \dots, S_m\}$, and they respect $\Ext^1_\cC(S_i, x^\bg S_j) = 0$ whenever $\bg \preceq 0$. We refer to these objects as \emph{primitive simples}.
\end{defn}

We recall that $\Ext^1_\cC(A, B)$ counts the number of non-trivial extensions of $A$ by $B$, that is short exact sequences $ B \rightarrow E \rightarrow A$, up to equivalence and where $E \ncong A \oplus B$.  

\smallskip

Let $\cC$ be a strongly c.b.l. Jordan--H\"older category with fixed set of primitive simples $\{S_1, \dots, S_m\}$. 
Therefore any $X \in \cC$ admits a c.b.l.f. composition series with composition factors given by the primitive simples. We call such a filtration a \emph{primitive composition series}.
Moreover, each object~$X$ admits a unique maximal element $\deg(X) \in \bZ^n$ such that $X$ admits a primitive composition series with multiplicities $\kappa_{i,\bg+\deg(X)} = 0$ whenever $\bg \prec 0$. 

\begin{prop}\label{prop:ordercomp} 
Let $\cC$ be a strongly c.b.l. Jordan--H\"older category. Let $X$ be an object in $\cC$, with primitive composition series contained in a cone $C_X = \{\deg(X) = \bc_0, \bc_1, \bc_2, \dots\}$ with $\bc_i \prec \bc_{i+1}$. Then $X$ admits an ordered primitive composition series
\[
0 \subset \cdots \subset X_2 \subset X_1 \subset X_0 = X
\]
where $X_r/X_{r+1} \cong x^{\bc_r} \left( S_1^{\oplus \kappa_{1,c_r}} \oplus \cdots \oplus S_m^{\oplus \kappa_{m,c_r}} \right)$ is a finite direct sum of the primitive simples shifted by $\bc_r$. Moreover, this composition series is unique whenever $C_X$ is fixed.
\end{prop}

In order to prove this proposition, we will need the following (most probably classical) lemma:

\begin{lem}\label{lem:trivialdoubleext}
A filtration $X_3 \subset X_2 \subset X_1$ with $X_1/X_2 \cong A$ and $X_2/X_3 \cong B$ gives rise to a commutative diagram
\[
\begin{tikzcd}
X_3 \ar[hookrightarrow]{r} \ar[equal]{d} & X_2  \ar[hookrightarrow]{d}\ar[twoheadrightarrow]{r} & B \ar[hookrightarrow]{d} \\
X_3  \ar[hookrightarrow]{r}  & X_1 \ar[twoheadrightarrow]{d} \ar[twoheadrightarrow]{r} & E \ar[twoheadrightarrow]{d} \\
& A \ar[equal]{r} & A
\end{tikzcd}
\]
where each row and column are short exact sequences. If the sequence $B \rightarrow E \rightarrow A$ splits, then there exists $X_2'$ such that $X_3 \subset X_2' \subset X_1$, and $X_1/X_2' \cong B$ and $X_2'/X_3 \cong A$.
\end{lem}

\begin{proof}
Composing the map $X_1 \twoheadrightarrow E$ with the splitting map $E \rightarrow B$ yields an epimorphism $p  : X_1 \twoheadrightarrow B$. We put $X_2' = \ker p$. Then the universal property of the kernel induces  a monomorphism $X_3 \hookrightarrow \ker p$ such that $\ker p/X_3 \cong A$.
\end{proof}

\begin{proof}[Proof of \cref{prop:ordercomp}]
We take a primitive composition series $Y_\bullet$ of $X$ and consider $\boldsymbol c_r$ for some $r \geq 0$. Since the filtration is c.b.l.f., we can extract a sub-filtration of $Y_\bullet$ containing all occurrences of the primitive simples in degree $\preceq \bc_r$. Then, thanks to \cref{lem:trivialdoubleext} we can apply a bubble sorting algorithm to bring all of these primitive simples to the beginning of the filtration, exchanging $x^{\bg }S_i \cong Y_s/Y_{s+1}$ with $x^{\bg'} S_{i'} \cong Y_{s+1}/Y_{s+2}$ whenever $\bg' \prec \bg$. Therefore we obtain a filtration
\[
 X_{\preceq \bc_r} \subset X,
\]
where $X/X_{\preceq \bc_r}$ is finite length and contains all composition factors of degree $\preceq \bc_r$. We conclude by taking the filtration given by $X_r := X_{\preceq \bc_r}$ for each $r \in \bN$.
\end{proof}

\subsection{C.b.l.f. cobaric abelian categories}\label{sec:cobaric}

Recall that a subcategory $\cA' \subset \cA$ of an abelian category is said to be \emph{thick} if for any short exact sequence $0 \rightarrow X \rightarrow Y \rightarrow Z \rightarrow 0$ in $\cA$, then $Y \in \cA'$ if and only if $X, Z \in \cA'$. 
Inspired by the definition of baric structure in~\cite{baric}, we introduce the following notion:

\begin{defn}\label{def:cobaricstruct}
A \emph{cobaric structure} on a strictly $\bZ^n$-graded abelian category $\cC$ is the datum of a pair of thick subcategory $\cC_{\preceq 0}, \cC_{\succeq 0}$ such that
\begin{itemize}
\item $x^\bg \cC_{\preceq 0} \subset \cC_{\preceq 0}$ for all $\bg \prec 0$;
\item $x ^\bg \cC_{\succeq 0} \subset \cC_{\succeq 0}$ for all $\bg \succ 0$;
\item $\Ext^1_\cC(Y,X) = 0$ whenever $Y \in \cC_{\succeq 0}$  and $X \in \cC_{\preceq 0}$;
\item for each $X \in \cC$ there is a short exact sequence
\begin{equation}\label{eq:SEScobaric}
0 \rightarrow X_{\succ 0} \rightarrow X \rightarrow X_{\preceq 0} \rightarrow 0,
\end{equation}
with $X_{\preceq 0} \in \cC_{\preceq 0}$ and $X_{\succ 0} \in \cC_{\succ 0}$,
\end{itemize}
where $\cC_{\succ 0} := \{ X \in \cC | x^{-\bz} X \in \cC_{\succeq 0} \}$ and $\bz \succ 0 \in \bZ^n$ the minimal non-zero element of $\bZ^n$.
\end{defn}

We define $\cC_{\preceq \bg} := \{ X \in \cC | x^{-\bg} \in \cC_{\preceq 0}\}$ and similarly for $\cC_{\prec \bg}, \cC_{\succeq \bg}$ and $\cC_{\succ \bg}$. 
We also put $\cC_\bg := \cC_{\preceq \bg} \cap \cC_{\succeq \bg}$. Note that $\cC_\bg$ is a semi-simple category. For this reason, and because $\cC$ is stricly $\bZ^n$-graded, we have $\Hom_\cC(Y,X) = 0$ whenever  $Y \in \cC_{\succ 0}$ and $X \in \cC_{\preceq 0}$.

\begin{lem}\label{lem:cobarictruncSES}
The assignments $\beta_{\succ \bg}X := X_{\succ \bg}$ and $\beta_{\preceq \bg}X := X_{\preceq \bg}$  are functorial.  
Moreover, there is a short exact sequence
\[
0 \rightarrow \beta_{\succ \bg} X \rightarrow X \rightarrow \beta_{\preceq \beta} X \rightarrow 0,
\]
and any similar short exact sequence is canonically isomorphic to this one. 
\end{lem}

\begin{proof}
Consider $f : X \rightarrow Y$. 
Because $\Hom_\cC(\beta_{\succ \bg}X, \beta_{\preceq \bg}Y) = 0$, we obtain a commutative diagram
\[
\begin{tikzcd}
0 \ar{r}& \beta_{\succ \bg}X  \ar{r} \ar[dashed]{d}{f_0} & X \ar{r} \ar{d}{f} & \beta_{\preceq \bg} X  \ar[dashed]{d}{f_1}  \ar{r}& 0 \\
0 \ar{r}& \beta_{\succ \bg}Y \ar{r} & Y \ar{r} & \beta_{\preceq \bg} Y  \ar{r}& 0 
\end{tikzcd}
\]
where $f_0$ (resp. $f_1$) is induced by the universal property of $\beta_{\succ \bg}Y$ as kernel (resp. $\beta_{\preceq \bg}Y$ as cokernel). 
\end{proof}

\begin{lem}\label{lem:betacompab}
For $\bg \prec \bg'$ we have natural isomorphisms
\begin{align*}
\beta_{\preceq \bg} \circ \beta_{\preceq \bg'} &\cong \beta_{\preceq \bg},  &
\beta_{\preceq \bg} \circ \beta_{\succeq \bg '} &\cong   \beta_{\succeq \bg '}  \circ \beta_{\preceq \bg}  \cong 0,
\\
\beta_{\succeq \bg} \circ \beta_{\succeq \bg'} &\cong \beta_{\succeq \bg'},  &
\beta_{\succeq \bg} \circ \beta_{\preceq \bg '} &\cong  \beta_{\preceq \bg '} \circ \beta_{\succeq \bg}.
\end{align*}
\end{lem}

\begin{proof}
The proof is an abelian version of~\cite[Proposition 1.3.5]{perversesheaves}. 
The only non-trivial relation is $\beta_{\succeq \bg} \circ \beta_{\preceq \bg '} \cong  \beta_{\preceq \bg '} \circ \beta_{\succeq \bg}$. 
Consider the filtration $\beta_{\succ g'} X \subset \beta_{\succ g} X \subset X$. It gives a commutative diagram
\[
\begin{tikzcd}
\beta_{\succ \bg'} X \ar[hookrightarrow]{r} \ar[equal]{d} & \beta_{\succeq \bg} X  \ar[hookrightarrow]{d}\ar[twoheadrightarrow]{r} & Y \ar[hookrightarrow]{d} \\
\beta_{\succ \bg'} X  \ar[hookrightarrow]{r}  & X \ar[twoheadrightarrow]{d} \ar[twoheadrightarrow]{r} & \beta_{\preceq \bg'} X \ar[twoheadrightarrow]{d} \\
& \beta_{\prec \bg} X \ar[equal]{r} & \beta_{\prec \bg} X
\end{tikzcd}
\]
where all rows and columns are short exact sequences. In particular 
\begin{align*}
&0 \rightarrow \beta_{\succ \bg'} X  \rightarrow  \beta_{\succeq \bg} X \rightarrow Y \rightarrow 0, \\
\text{(resp. }\quad  &0 \rightarrow Y \rightarrow \beta_{\preceq \bg'} X \rightarrow \beta_{\prec \bg} X \rightarrow 0 \text{ )}, 
\end{align*}
is  short exact sequence equivalent to
\begin{align*}
&0 \rightarrow \beta_{\succ \bg'} X  \rightarrow  \beta_{\succeq \bg} X \rightarrow \beta_{\preceq \bg'} \beta_{\succeq \bg} X \rightarrow 0, \\
\text{(resp. }\quad  &0 \rightarrow \beta_{\succeq \bg}\beta_{\preceq \bg'} X \rightarrow \beta_{\preceq \bg'} X \rightarrow \beta_{\prec \bg} X \rightarrow 0 \text{ )}.
\end{align*}
Thus, by \cref{lem:cobarictruncSES} we have $ \beta_{\preceq \bg'} \beta_{\succeq \bg} X \cong Y \cong  \beta_{\succeq \bg}\beta_{\preceq \bg'} X$. 
\end{proof}

\begin{lem}\label{lem:sesbetag}
For all $X \in \cC$ and $\bg \in \bZ^n$ there are short exact sequences
\begin{align*}
&0 \rightarrow \beta_{\succ \bg} X \rightarrow \beta_{\succeq \bg} X \rightarrow \beta_\bg X \rightarrow 0, \\
&0 \rightarrow \beta_{\bg} X \rightarrow \beta_{\prec \bg} X \rightarrow \beta_{\preceq \bg} X \rightarrow 0. 
\end{align*}
\end{lem}

\begin{proof}
By applying a version of~\cref{eq:SEScobaric} on $\beta_{\succeq \bg} X$ we obtain a short exact sequence
\[
0 \rightarrow\beta_{\succ \bg} \circ \beta_{\succeq \bg} X \rightarrow \beta_{\succeq \bg} X \rightarrow \beta_{\preceq \bg} \circ \beta_{\succeq \bg} X  \rightarrow 0,
\]
which gives the first desired sequence by~\cref{lem:betacompab}. Similarly, we obtain
\[
0 \rightarrow \beta_{\succeq \bg} \circ \beta_{\preceq \bg} X \rightarrow \beta_{\preceq \bg} X \rightarrow \beta_{\succ \bg} \circ \beta_{\preceq \bg} X \rightarrow 0,
\]
which gives the second desired sequence.
\end{proof}

\begin{lem}
The functors $\beta_{\succ \bg} : \cC \rightarrow \cC_{\succ \bg}$ and $\beta_{\preceq \bg} : \cC \rightarrow \cC_{\preceq \bg}$ are both exact.
\end{lem}

\begin{proof}
The proof is basically the same as in~\cite[Proposition 2.3]{baric}. Let $0 \rightarrow X \rightarrow Y \rightarrow Z \rightarrow 0$ be a short exact sequence. 
 By the 3x3 lemma, we obtain a commutative diagram 
\[
\begin{tikzcd}
& 0 \ar{d} & 0\ar{d} & 0\ar{d} & 
\\
0 \ar{r} & \beta_{\succ \bg} X \ar{d}\ar{r} & \beta_{\succ \bg}Y \ar{d} \ar{r}& K_1 \ar{d} \ar{r}& 0
\\
0  \ar{r}& X \ar{d}  \ar{r}& Y \ar{d}\ar{r} & Z \ar{r} \ar{d}& 0 
\\
0  \ar{r}& \beta_{\preceq \bg} X \ar{d}\ar{r} & \beta_{\preceq \bg} Y  \ar{d}\ar{r}& K_2\ar{d} \ar{r} & 0
\\
& 0 & 0 & 0 & 
\end{tikzcd}
\]
where all rows and columns are exact. Because $\cC_{\succ \bg}$ is thick, we have $K_1 \in \cC_{\succ \bg}$. Thus we must have $\beta_{\succ \bg} Z \cong K_1$. Similarly, $\beta_{\preceq \bg} Z \cong K_2$. 
\end{proof}

\begin{defn}
A cobaric structure $(\cC_{\preceq 0}, \cC_{\succeq 0})$ is:
\begin{itemize}
\item \emph{non-degenerate} if $\bigcap_{\bg \in \bZ^n} \cC_{\preceq \bg} = 0 = \bigcap_{\bg \in \bZ^n} \cC_{\succeq \bg}$;
\item \emph{locally finite} if $\cC_0$ is finite length generated by a finite collection of simple objects;
\item \emph{cone bounded} if for any $X \in \cC$ there exists a \emph{bounding cone} $C_X \subset \bR^{n}$ and a \emph{minimal degree} $\be \in \bZ^n$ such that $\beta_\bg(X) = 0$ for all $\bg - \be \notin C_X$;
\item \emph{stable} if $\beta_{\preceq \bg}$ commutes with acceptable filtered limits and acceptable filtered colimits;
\item \emph{full} if $\bigoplus_{\bg \in \bZ^n} \beta_\bg X \in \cC$ for all $X \in \cC$; 
\item \emph{strongly c.b.l.f.} if it is non-degenerate, locally finite, cone bounded, stable and full. 
\end{itemize}
\end{defn}

Note that when a cobaric structure is cone bounded, it is enough to verify $ \bigcap_{\bg \in \bZ^n} \cC_{\succeq \bg} = 0$ for it to be non-degenerate. Also $\beta_{\succeq \be} X = X$ whenever $\be$ is the minimal degree of $X$. Moreover, when $\beta$ is cone bounded and locally finite, then $\cC_{\preceq \bg}$ is a finite length abelian category. 

\smallskip

Our goal for the remaining of the section will be to prove the following theorem. 

\begin{thm}\label{thm:eqcobaricJH}
Let $\cA$ be an AB4 and AB4* abelian category. 
Let $\cC$ be a c.b.l. additive abelian subcategory of $\cA$. 
If $\cC$ is strongly c.b.l. Jordan--H\"older, then it admits a strongly c.b.l.f. cobaric structure. 
Conversly, if $\cC$ admits such a cobaric structure and admits colimits of kernel of acceptable colimits, and if $\cA$ is Mittag--Leffler, then $\cC$ is strongly c.b.l. Jordan--H\"older. 
\end{thm}

From now on we suppose $\cC$ is a strictly $\bZ^n$-graded abelian subcategory of $\cA$.   

\begin{lem}\label{lem:cblisfull}
Let $\beta$ be a cobaric structure on $\cC$. If $\beta$ is cone bounded, locally finite and if $\cC$ is c.b.l. additive, then $\beta$ is full.
\end{lem}

\begin{proof}
Let $\{S_1, \dots, S_m\}$ be the set of distinct simple objects of $\cC_0$. 
Take $X \in \cC$. Then $\beta_\bg X$ is a finite direct sum of $\{x^\bg S_1, \dots, x^\bg S_m\}$. Thus $\bigoplus_{\bg \in \bZ^n} \beta_\bg X$ is a c.b.l.f. direct sum of $\{S_1,\dots,S_m\}$, and is in $\cC$. 
\end{proof}

\begin{prop}\label{prop:cobariccanlim}
If $\cC$ admits a cobaric structure $\beta$ which is non-degenerate, stable and cone bounded, 
then we have
\begin{align*}
0 &\cong \lim\bigl( \cdots \hookrightarrow \beta_{\succ \be + \bc_1} X \hookrightarrow \beta_{\succ \be + \bc_0} X   \bigr), \\
X &\cong  \lim\bigl( \cdots \twoheadrightarrow \beta_{\preceq \be + \bc_1} X \twoheadrightarrow \beta_{\preceq \be + \bc_0} X   \bigr), 
\end{align*}
for all $X \in \cC$, where $C_X = \{0 = \bc_0 \prec \bc_1 \prec \cdots \}$ is a bounding cone of $X$ with minimal degree $\be$. 
If $\beta$ is full, then these filtered limits are acceptable. 
\end{prop}

\begin{proof}
Let 
\[
K := \lim\bigl( \cdots \hookrightarrow \beta_{\succ \be + \bc_1} X \hookrightarrow \beta_{\succ \be + \bc_0} X   \bigr).
\]
Since $\beta$ is cone bounded and locally finite, $\bigoplus_{\bg} \beta_{\bg} X$ is a c.b.l.f. direct sum of the simple objects in $\cC_0$. 
Then we have 
\[
\beta_{\preceq \bg} K \cong  \lim\bigl( \cdots \hookrightarrow \beta_{\preceq \bg}\beta_{\succ \be + \bc_1} X \hookrightarrow \beta_{\preceq \bg}\beta_{\succ \be + \bc_0} X   \bigr) \cong 0,
\]
since $\be + \bc_r \succ \bg$ for $r \gg 0$, so that $\beta_{\preceq \bg}\beta_{\succ \be + \bc_r} X = 0$. Thus, $K \in \bigcap_{\bg} \cC_{\succ \bg} = 0$. 

Let
\[
C := \lim\bigl( \cdots \twoheadrightarrow \beta_{\preceq \be + \bc_1} X \twoheadrightarrow \beta_{\preceq \be + \bc_0} X   \bigr).
\]
Because limits commute with kernels we have
\begin{align*}
0 = K &= 
 \lim\bigl( \cdots \hookrightarrow \ker(X \rightarrow \beta_{\preceq \be + \bc_1} X) \hookrightarrow \ker(X \rightarrow \beta_{\preceq \be + \bc_0} X)   \bigr) \\
 &\cong
 \ker(X \rightarrow C).
\end{align*}
Moreover, because $\beta_{\preceq \bg}$ is exact, we obtain
\[
\beta_{\preceq \bg }\cok( X \hookrightarrow C) \cong \cok(\beta_{\preceq \bg} X \hookrightarrow \beta_{\preceq \bg}C).
\]
For $r \gg 0$ we have $\beta_{\preceq \bg}\beta_{\preceq \be+ \bc_r} = \beta_{\preceq\bg}$. Therefore $\beta_{\preceq \bg} X \cong \beta_{\preceq \bg} C$, and $\cok( X \hookrightarrow C)  \in \bigcap_{\bg \in \bZ^n} \cC_{\succ \bg} = 0$. Thus, $X \cong C$. 
\end{proof}

\begin{lem}\label{lem:finitelengthlim}
Let $\cA'$ be a finite length abelian category. If $\cA'$ is Mittag--Leffler, then 
filtered limits and filtered colimits are exact.
\end{lem}

\begin{proof}
Consider a limit
\[
X := \lim (\cdots \xrightarrow{f_1} X_1 \xrightarrow{f_0} X_0),
\]
and fix $i \geq 0$. Then we obtain a filtration
\[
\cdots \subset \Image(X_{i+2} \rightarrow X_i) \subset \Image(X_{i+1} \rightarrow X_i) \subset X_i.
\]
Because $\cA'$ is finite length, the filtration stabilizes and the limits respects the Mittag--Leffler condition.
The colimit case is dual. 
\end{proof}

\begin{prop}\label{prop:MLcobaricisAB5loc}
If $\cA$ is  Mittag--Leffler and $\cC$ admits a cobaric structure $\beta$ which is non-degenerate, stable and cone bounded, then $\cC$ is locally AB5*. 
\end{prop}

\begin{proof}
Consider two acceptable filtered limits
\begin{align*}
X &:= \lim (\cdots \xrightarrow{f_1} X_1 \xrightarrow{f_0} X_0), &
Y &:= \lim (\cdots \xrightarrow{g_1} Y_1 \xrightarrow{g_0} Y_0),
\end{align*}
with a collection of commuting maps $u_\bullet : X_\bullet \rightarrow Y_\bullet$ inducing a map $u : X \rightarrow Y$. We put
\begin{align*}
C &:= \cok (X \xrightarrow{u} Y), \\
C_r &:= \cok(X_r \xrightarrow{u_r} Y_r).
\end{align*}
Then we obtain
\begin{align*}
\beta_{\preceq \bg}C &\cong \cok(\beta_{\preceq \bg} X \rightarrow \beta_{\preceq \bg} Y) \\
&\cong \cok\biggl( \lim\bigl( \cdots\rightarrow \beta_{\preceq\bg} X_1 \rightarrow \beta_{\preceq \bg} X_0 \bigr)\rightarrow \lim\bigl( \cdots\rightarrow \beta_{\preceq\bg} Y_1 \rightarrow \beta_{\preceq \bg} Y_0 \bigr)\biggr).
\end{align*}
Since $\beta_{\preceq \bg} X_r$ and $\beta_{\preceq \bg} Y_r$ are in the finite length category $\cC_{\preceq\bg}$ for all $r\geq 0$, we can apply \cref{lem:finitelengthlim} to get 
\begin{align*}
\beta_{\preceq \bg}C &\cong \lim \bigl(\cdots \rightarrow \cok(\beta_{\preceq \bg} X_1 \rightarrow \beta_{\preceq \bg} Y_1) \rightarrow \cok(\beta_{\preceq \bg} X_0 \rightarrow \beta_{\preceq \bg} Y_0)  \bigr) \\
&\cong  \lim \bigl(\cdots \rightarrow \beta_{\preceq \bg} C_1 \rightarrow \beta_{\preceq \bg} C_0  \bigr) .
\end{align*}
By \cref{prop:cobariccanlim}, we obtain
\begin{align*}
\lim \bigl(\cdots \rightarrow  C_1 \rightarrow C_0  \bigr) &\cong
\lim \bigl(\cdots \rightarrow \lim_{\bg}( \beta_{\preceq \bg} C_1 ) \rightarrow \lim_\bg( \beta_{\preceq \bg} C_0 ) \bigr) \\
&\cong \lim_\bg \lim\bigl( \cdots \rightarrow \beta_{\preceq \bg} C_1 \rightarrow \beta_{\preceq \bg} C_0  \bigr) \\
&\cong \lim_\bg \beta_{\preceq \bg}C \cong C.
\end{align*}
This concludes the proof. 
\end{proof}

\begin{prop}\label{prop:MLcobaricisAB5locbis}
Suppose $\cA$ is  Mittag--Leffler and $\cC$ admits a cobaric structure $\beta$ which is non-degenerate, stable and cone bounded. 
If $\cC$ admits colimits of kernels of acceptable colimits, then $\cC$ is locally AB5.
\end{prop}

\begin{proof}
Consider two acceptable filtered colimits
\begin{align*}
X &:= \colim ( X_0 \xrightarrow{f_0} X_1 \xrightarrow{f_1} \cdots ), &
Y &:= \colim (Y_0 \xrightarrow{g_0} Y_1 \xrightarrow{g_1} \cdots ),
\end{align*}
with a collection of commuting maps $u_\bullet : X_\bullet \rightarrow Y_\bullet$ inducing a map $u : X \rightarrow Y$. 
We put
\begin{align*}
K &:= \ker (X \xrightarrow{u} Y), \\
K_r &:= \ker(X_r \xrightarrow{u_r} Y_r), \\
K' &:= \colim \bigl(K_0 \rightarrow K_1 \rightarrow \cdots \bigr).
\end{align*}
Then we have
\begin{align*}
\beta_{\preceq \bg}K &\cong \ker(\beta_{\preceq \bg} X \rightarrow \beta_{\preceq \bg} Y) \\
&\cong \ker\biggl( \colim\bigl(  \beta_{\preceq \bg} X_0 \cdots\rightarrow \beta_{\preceq\bg} X_1 \rightarrow\cdots \bigr)\rightarrow \colim\bigl( \beta_{\preceq \bg} Y_0 \rightarrow \beta_{\preceq\bg} Y_1 \rightarrow \cdots \bigr)\biggr),
\end{align*}
and 
\[
\beta_{\preceq \bg} K' \cong \colim \bigl( \ker(\beta_{\preceq \bg} X_0 \rightarrow \beta_{\preceq \bg} Y_0) \rightarrow \ker(\beta_{\preceq \bg} X_1 \rightarrow \beta_{\preceq \bg} Y_1) \rightarrow\cdots  \bigr).
\]
Because $\beta_{\preceq \bg} X_r$ and $\beta_{\preceq \bg} Y_r$ are both in the finite length category $\cC_{\preceq\bg}$ for all $r\geq0$, we can apply \cref{lem:finitelengthlim} to get 
\begin{align*}
\beta_{\preceq \bg} \cok(K' \rightarrow K) &\cong \cok(\beta_{\preceq \bg} K' \rightarrow \beta_{\preceq \bg} K) = 0, \\
\beta_{\preceq \bg} \ker(K' \rightarrow K) &\cong \ker(\beta_{\preceq \bg} K' \rightarrow \beta_{\preceq \bg}  K) = 0.
\end{align*}
 Thus $\cok(K' \rightarrow K), \ker(K' \rightarrow K)  \in \bigcap_{\bg \in \bZ^n} \cC_{\succ \bg} = 0$, and $K' \cong K$. 
\end{proof}

%

\begin{lem}\label{lem:cobaricstablefilt}
Let $\beta$ be a strongly c.b.l.f. cobaric structure on $\cC$. All objects of $\cC$ are stable for the filtrations.
\end{lem}

\begin{proof}
Let $X \in \cC$. Suppose $X$ admits two $\bN$-composition series
\begin{align*}
0 \subset \cdots \subset X_2 \subset X_1 \subset X_0 = X, \\
0 \subset \cdots \subset X'_2 \subset X'_1 \subset X'_0 = X.
\end{align*}
Fix $i \in \bN$. 
For all $\be \in \bZ^n$, we have $\beta_{\succ \be} X \cap X_i = \beta_{\succ \be} X_i$ so that 
\[
\frac{\beta_{\succ \be} X}{ \beta_{\succ \be} X \cap X_i} = \beta_{\succ \be}(X/X_i).
\]
Since $X/X_i$ admits a finite composition series and $\beta$ is cone bounded, locally finite, we have $\beta_{\succ \be}(X/X_i) = 0$ for $\be \ggcurly 0$. Thus $\beta_{\succ \be} X \subset X_i$ for $\be \ggcurly 0$. Fix such an $\be \in \bZ^n$.  For any $r \geq 0$ we obtain a commutative diagram
\[
\begin{tikzcd}
0 \ar{r} & \beta_{\succ \be} X_r' \ar{r} \ar[hookrightarrow]{d} & X_r' \ar{r} \ar[hookrightarrow]{d} & \beta_{\preceq \be} X_r' \ar{d} \ar{r} & 0 \\
0 \ar{r} & \beta_{\succ \be} X \ar{r} & X \ar{r} & \beta_{\preceq \be} X \ar{r} & 0,
\end{tikzcd}
\]
where the rows are exact and the two vertical arrows on the left are monomorphisms. 
Since $\beta$ is cone bounded, locally finite, we have  $\beta_{\preceq \be} X'_r = 0$ for $r \gg 0$, so that $\beta_{\succ \be} X_r' \cong X_r'$. Therefore, for $r$ big enough we obtain $X'_r \subset \beta_{\succ \be} X \subset X_i$.
%
%
\end{proof}

\begin{proof}[Proof of \cref{thm:eqcobaricJH}]
Suppose $\cC$ is strongly c.b.l. Jordan--H\"older. We construct a cobaric structure $\cC_{\preceq 0}, \cC_{\succ 0}$ by putting
\begin{align*}
\cC_{\preceq 0} &:= \{ X \in \cC | x^{\bg} S_i^*(X) = 0 \text{ for all $\bg \succ 0$ and $i \in I$ } \}, \\
\cC_{\succ 0} &:= \{ X \in \cC | x^{\bg} S_i^*(X) = 0 \text{ for all $\bg  \preceq 0$ and $i \in I$ } \}.
\end{align*}
By \cref{prop:ordercomp}  we obtain the short exact sequence \cref{eq:SEScobaric}. Thus, it forms a cobaric structure, which is non-degenerate by \cref{lem:JHS0trivial}, and obviously cone bounded and locally finite. 
We now prove the cobaric structure is stable for the limits, the dual case being similar. Consider an acceptable filtered limit
\[
X := \lim\bigl(\cdots \rightarrow X_1 \rightarrow X_0\bigr).
\]
Because $\cC$ is locally AB5* we obtain a commutative diagram 
\[
\begin{tikzcd}
&& 0 \ar{d} & 0 \ar{d} && \\
0 \ar{r}& K  \ar{r} &\beta_{\succ \bg} X  \ar{d}  \ar{r}& \lim\bigl(\cdots \rightarrow \beta_{\succ \bg}X_1 \rightarrow \beta_{\succ \bg} X_0\bigr)  \ar{d}  \ar{r} & C \ar{r}  & 0\\
0 \ar{r}&0 \ar{r} &X  \ar{d}  \ar[equals]{r} & \lim\bigl(\cdots \rightarrow X_1 \rightarrow X_0\bigr)  \ar{d} \ar{r} & 0  \ar{r} & 0 \\
0 \ar{r}&K' \ar{r} &\beta_{\preceq \bg} X  \ar{d}  \ar{r} & \lim\bigl(\cdots \rightarrow \beta_{\preceq \bg} X_1 \rightarrow \beta_{\preceq \bg} X_0\bigr)  \ar{d}  \ar{r} & C'  \ar{r} & 0\\
& &0 & 0 & &
\end{tikzcd}
\]
where the rows and columns are exact. By the snake lemma, we obtain that $K = C' = 0$ and thus 
\[
0 \rightarrow K' \rightarrow \beta_{\preceq \bg} X \rightarrow \lim\bigl(\cdots \rightarrow \beta_{\preceq \bg} X_1 \rightarrow \beta_{\preceq \bg} X_0\bigr)  \rightarrow 0,
\]
is a short exact sequence. Take a simple object $S \in \cC$.  Because $S^*(\beta_{\preceq \bg} X_r) = S^*(\beta_{\preceq \bg} X_{r+1})$ for $r \gg 0$, we have
\[
S^*(\beta_{\preceq \bg} X) = S^* \bigl( \lim\bigl(\cdots \rightarrow \beta_{\preceq \bg} X_1 \rightarrow \beta_{\preceq \bg} X_0\bigr) \bigr).
\]
Thus, $K' \cong 0$ by \cref{lem:JHsameSiso}, concluding the first part of the proof. 

For the second part, suppose $\cC$ is Mittag--Leffler and admits a strongly c.b.l.f. cobaric structure $\beta = (\cC_{\preceq 0}, \cC_{\succ 0})$. Because $\beta$ is cone bounded, locally finite there is only a finite set of distinct simple objects.  By \cref{prop:cobariccanlim} any object admits a $\bN$-filtration where the quotients are finite length objects. Thus we can refine the filtration into an $\bN$-composition series. Because $\beta$ is cone bounded, locally finite, it gives a c.b.l.f. composition series. By \cref{lem:cobaricstablefilt} all objects are stable for the filtrations. Finally, $\cC$ is both locally AB5 and locally AB5* thanks to \cref{prop:MLcobaricisAB5loc} and \cref{prop:MLcobaricisAB5locbis}.
\end{proof}

\subsection{Topological Grothendieck group}

As in Achar--Stroppel~\cite{acharstroppel}, we define the following notion:

\begin{defn}
Let $\cC$ be a $\bZ^n$-graded category with a c.b.l.f. cobaric structure $\beta$. The \emph{topological Grothendieck group} of $(\cC,\beta)$ is
\[
\bGO(\cC,\beta) := G_0(\cC) / J(\cC,\beta),
\]
where
\[
J(\cC,\beta) :=  \{ f \in G_0(\cC) \ |\  [\beta_{\preceq \bg}] f = 0 \in G_0(\cC) \text{ for all $\bg \in \bZ^n$}\}.
\]
The canonical topology on $\bGO(\cC,\beta)$ is given by using $\{\bGO(\cC_{\succ \bg})  \subset \bGO(\cC) \}_{\bg \in \bZ}$  as basis of neighborhood of zero. 
\end{defn}

Note that if $\beta$ is non-degenerate, then the canonical topology on  $\bGO(\cC,\beta)$ is Hausdorff. 
Inspired by~\cite[Definition~2.3]{acharstroppel}, we define the following:

\begin{defn}
Let $\cC$ and $\cC'$ be two $\bZ^n$-graded categories with c.b.l.f. cobaric structures $\beta = \{\cC_{\preceq 0}, \cC_{\succeq 0}\}$ and $\beta' = \{\cC'_{\preceq 0}, \cC'_{\succeq 0}\}$ respectively. We say that a functor $F : \cC \rightarrow \cC'$ has \emph{finite amplitude} if there exists $|F| \in \bZ^n$ such that $F\cC_{\succeq \bg} \subset \cC'_{\succeq \bg + |F|}$. 
\end{defn}

\begin{prop}\label{prop:abfinampinduces}
Let $F : (\cC,\beta) \rightarrow (\cC',\beta')$  be a $\bZ^n$-homogeneous exact functor. If $F$ has finite amplitude, then it induces a continuous map
\[
[F] : \bGO(\cC,\beta) \rightarrow \bGO(\cC',\beta')
\]
by $[F][X] := [F(X)]$. 
\end{prop}

\begin{proof}
By exactness we have a map
\[
[F]  :  G_0(\cC)  \rightarrow \bGO(\cC',\beta'),
\]
and thus we only need to prove $F(J(\cC,\beta)) \subset J(\cC',\beta')$. 

First, we observe that there is an isomorphism
\begin{equation}\label{eq:GOdecomp}
G_0(\cC') \xrightarrow{\simeq} G_0(\cC'_{\preceq \bg}) \oplus G_0(\cC'_{\succ \bg}), \quad [X] \mapsto [\beta'_{\preceq  \bg}X] + [\beta'_{\succ \bg} X],
\end{equation}
for all $\bg \in \bZ^n$. 

Take $f \in J(\cC,\beta)$ and $\bg \in \bZ^n$. By  \cref{eq:GOdecomp} we obtain
\begin{align*}
f &= [\beta_{\preceq \bg}] f + [\beta_{\succ  \bg}]f, & [F]f &= [\beta_{\preceq  \bg + |F|}][F]f +  [\beta_{\succ  \bg + |F|}][F]f.
\end{align*}
We have $ [\beta_{\preceq  \bg}] f = 0$ and $[\beta_{\succ  \bg + |F|}][F]f = [F][\beta_{\succ  \bg}]f$. Therefore, we deduce that $[F]f =  [\beta_{\succ  \bg + |F|}][F]f$ and $[\beta_{\preceq  \bg + |F|}][F]f = 0$. Since $\bg$ is arbitrary, we conclude that $[F]f \in J(\cC',\beta')$.
\end{proof}

\begin{prop}\label{prop:topGOisasympGO}
Let $\cC$ be a Mittag--Leffler strictly $\bZ^n$-graded abelian category with  a  full, non-degenerate, stable, cone bounded cobaric structure $\beta$.
There is a surjection
\[
 \bGO(\cC,\beta) \twoheadrightarrow \bGO(\cC),
\]
induced by the identity on $G_0(\cC)$. 
If  $\beta$ is locally finite, then this surjection is an isomorphism
\[
 \bGO(\cC,\beta) \cong \bGO(\cC). 
\]
Moreover, a functor between two such categories is of finite amplitude if and only if it is locally finite. 
\end{prop}

\begin{proof}
Take $[X] - [Y] \in J(\cC,\beta)$. Then $[\beta_\bg X] = [\beta_\bg Y]$ for all $\bg \in \bZ^n$. By \cref{prop:cobariccanlim} and \cref{lem:limepiisML} we obtain that $[X] - [Y] \in J^*(\cC)$. Thus, there is an induced surjective map
\[
 \bGO(\cC,\beta) \twoheadrightarrow \bGO(\cC).
\]
Suppose that $\beta$ is locally finite. 
Consider an acceptable filtered colimit
\[
Y \cong \colim \bigl( X = F_0 \xrightarrow{f_0} F_1 \xrightarrow{f_1} \cdots \bigr),
\]
in $\cC$ and $\{C_r, K_r \in \cC\}_{r\geq 0}$, such that $\bigoplus_r C_r, \bigoplus_r K_r \in \cC$, and $[C_r] = [\cok f_r]$, $[K_r] = [\ker f_r]$ in $G_0(\cC)$ for all $r\geq 0$. We claim that
\[
[\beta_{\preceq \bg} Y] - [\beta_{\preceq \bg} X] = \sum_{r \geq 0} \bigl( [\beta_{\preceq \bg} C_r] - [\beta_{\preceq \bg} K_r] \bigr) \in G_0(\cC),
\]
for all $\bg \in \bZ^n$. First, we obtain by exactness of $\beta_{\preceq \bg}$ that $[\beta_{\preceq \bg} C_r] = [\beta_{\preceq \bg}  \cok f_r]$ and $[ \beta_{\preceq \bg}  K_r] = [ \beta_{\preceq \bg}  \ker f_r]$ in $G_0(\cC)$ for all $r \geq 0$. Moreover, since $\beta$ is cone bounded locally finite, we have $\beta_{\preceq \bg} C_r = \beta_{\preceq \bg} K_r = \beta_{\preceq \bg} \ker f_r = \beta_{\preceq \bg} \cok f_r = 0$ for $r \gg 0$. Let $r'$ be big enough such that $f_r$ is an isomorphism for all $r \geq r'$. Then we have
\begin{align*}
[\beta_{\preceq \bg} Y] - [\beta_{\preceq \bg} X]  &=  \sum_{r = 0}^{r'} \bigl( [\beta_{\preceq \bg} \cok f_r] - [\beta_{\preceq \bg} \ker f_r] \bigr) \\
 &=  \sum_{r = 0}^{r'} \bigl( [\beta_{\preceq \bg} C_r] - [\beta_{\preceq \bg} K_r] \bigr),
\end{align*}
 in $G_0(\cC)$. The dual case is similar, proving there is an isomorphism $\bGO(\cC,\beta) \cong \bGO(\cC)$. 
\end{proof}

\subsection{An example with modules}

Consider as before a c.b.l.f. dimensional, $\bZ^n$-graded $\Bbbk$-algebra $R$, where $\Bbbk$ is a field.

\begin{prop}\label{prop:RposJH}
If $R$ is positive, that is $R = \bigoplus_{\bg \succeq 0} R_\bg$, and $R_0$ is semi-simple, then $R\modlf$ is strongly c.b.l. Jordan--H\"older.
\end{prop}

\begin{proof}
Since $\End(R) = R_0$ is finite dimensional, there are only finitely many projective indecomposable modules $\{P_i = Re_i\}_{i \in I}$, given by primitive idempotents $\{e_i \in R_0\}_{i \in I}$.
Because $R$ is positive, then $R_{\succ 0} := \bigoplus_{\bg \succ 0} R_\bg$ is a c.b.l.f. dimensional submodule of $R$, and $R_0 \cong R/R_{\succ 0}$ is finite-dimensional.
Since $R_0$ is semi-simple and $S_i := P_i/R_{\succ 0} P_i$ is $R_0$-indecomposable, $S_i$ is a finite-dimensional simple module.
Furthermore, any simple $R$-module is isomorphic (up to grading shift) to $S_i$ for some~$i$. Since $R$ is positive, we have $\Ext^1_R(S_i, x^\bg S_j) = 0$ whenever $\bg \prec 0$. 

\smallskip

In addition, for any c.b.l.f. module $M = \bigoplus_{\bg \in \bZ^n} M_{\bg}$ and $\be \in \bZ^n$ then $M_{\succeq \be} := \bigoplus_{\bg \succeq \be} M_{\bg}$ is a submodule. In particular we get a short exact sequence
\[
M_{\succ \be} \hookrightarrow M_{\succeq \be} \twoheadrightarrow M_{\succeq \be}/M_{\succ \be},
\]
where $M_{\succeq \be}/M_{\succ \be}$ is isomorphic to a finite direct sum of elements in~$\{S_i\}_{i \in I}$. Therefore, we obtain by the classical Jordan--H\"older theorem a finite filtration with simple quotients
\[
M_{\succ \be} = M_{\be}^0  \subset M_{\be}^1 \subset \cdots \subset M_{\be}^{r-1} \subset M_{\be}^r = M_{\succeq \be}.
\]
Let $C_M = \{\deg(M) = \bc_0 \prec \bc_1 \prec \bc_2 \prec \cdots \}$ be a bounding cone of $M$. Applying the same reasoning on each $\bc_i$ with $\be = \bc_{i+1}$, we obtain by concatening all these filtrations a c.b.l.f. composition series of~$M$.

\smallskip

We now prove $M$ is stable for the filtrations. Suppose there are two composition series $X_\bullet$ and $X'_\bullet$ of $M$. Suppose by contradiction that $X'_k \not\subset X_i$ for all $k \in \bN$. Then $X_i \subsetneq X_i + X'_k$. Moreover, we know that $M/X_i$ is a finite dimensional $\Bbbk$-vector space and we get an infinite filtration
\[
X_i \subset \cdots \subset X_i + X'_2 \subset X_i + X'_1 \subset X.
\]
As vector spaces we have $X'_j \cong X'_{j+1} \oplus H_j$ for some $H_j$. We can find $H_j \not\subset X_i$ for arbitrary big $j$ since $X'_k \cong \bigoplus_{j \geq k} H_k$, and thus otherwise we would have $X'_k \subset X_i$. Then we write $\tilde H_j = H_j / H_j \cap X_i$ and $\bigoplus_j \tilde H_j$ is infinite dimensional. But we have $M \cong X_i \oplus \bigoplus_{j \ge 0}  \tilde H_j$, which contradicts the fact that $M/X_i$ is finite dimensional.

\smallskip

Finally, because any c.b.l.f. module is a locally finite dimensional module, we know that $R\modlf$ is locally AB5*. Since $R\amod$ is AB5, and $R\modlf$ is a full subcategory of $R\amod$, we obtain that $R\modlf$ is locally AB5. This concludes the proof. 
\end{proof}

\begin{rem}
Note that the short exact sequence
\[
0 \rightarrow M_{\succ \be} \rightarrow M \rightarrow M/M_{\succ \be} \rightarrow 0,
\]
in the proof of \cref{prop:RposJH} yields a strongly c.b.l.f. cobaric structure on $R\modlf$ (which coincides with the one given by \cref{thm:eqcobaricJH}).
\end{rem}

\begin{cor}\label{cor:GORmodlf}
Let $R \cong \bigoplus_{i \in I} Re_i$ be as in \cref{prop:RposJH} and take a full collection $\{e_j\}_{j \in J \subset I}$ of non-equivalent primitive idempotents. 
Then we have
\[
\bGO(R\modlf) \cong \bigoplus_{j} \bZ\pp{x_1, \dots, x_n} \cdot [S_j],
\] with $S_j := Re_j/R_{\succ 0}e_j$.
\end{cor}

Since $[P_i]$ can be written as $f_i(\boldsymbol x) [S_i]$ for some $f_i(\boldsymbol x) \in \bZ_C\llbracket x_1, \dots, x_n\rrbracket$ with $f_i(0) = 1$, we also have $[S_i] = f_i^{-1}(\boldsymbol x)[P_i]$. Therefore
\begin{equation}\label{eq:G0genPj}
\bGO(R\modlf) \cong \bigoplus_{j} \bZ\pp{x_1, \dots, x_n} \cdot [P_j].
\end{equation}
However, with the current framework, this isomorphism is purely formal, and has no categorical meaning (e.g. it involves minus signs). In the dg-setting (i.e. derived category), we will reinterpret it using the projective resolution of $S_i$.

\begin{prop}
Let $R$ and $R'$ be two positive c.b.l.f. dimensional $\bZ^n$-graded $\Bbbk$-algebras with $R_0$ and $R'_0$ semi-simple.   
Let $B$ be a c.b.l.f. dimensional $R'$-$R$-bimodule. The functor
\[
F : R\modlf \rightarrow R'\modlf, \quad F(X) := B \otimes_{R} X,
\]
is locally finite, and thus induces a map
\[
[F] : \bGO(R\modlf) \rightarrow \bGO(R' \modlf).
\]
If $B$ induces a Morita equivalence between $R$ and $R'$, then $[F]$ is an isomorphism.
\end{prop}

\begin{proof}
Immediate by \cref{prop:exactfunctoronG0}. 
\end{proof}

\begin{exe}
Recall the $\bZ \times \bZ^2$-graded algebra $R$ from \cref{sec:mainexintro} and its simple module $L$. We observe that $R$ respects the hypothesis of \cref{cor:GORmodlf}, and thus
\[
\bGO(R\modlf)  \cong \bZ\pp{h,q,\lambda} \cdot [L].
\]
Moreover, this time, we also have $R \in R\modlf$, and we obtain 
\[
[R] = \frac{1+ h \lambda^2}{1-q^2} [L],
\]
because of \cref{eq:filtintroex}. 
This could be already a satisfying answer to the problem of \cref{sec:mainexintro}, but as we will see below, we can do better. 
\end{exe}

\subsubsection{A surprising isomorphism}

Let 
\[
0 \subset \cdots \subset M_2 \subset M_1 \subset M_0 = M,
\]
be a c.b.l.f. composition series of some c.b.l.f. $R$-module $M$. We observe that $\bigoplus_{r \geq 1} M_r$ is a c.b.l.f. module since for any $x \in M$ then $x \notin M_r$ for $r\gg 0$ (otherwise $M$ would not be c.b.l.f.). Thus we have a short exact sequence 
\[
0 \rightarrow \bigoplus_{r \geq 1} M_r \rightarrow M \oplus \bigoplus_{r \geq 1} M_r \rightarrow \bigoplus_{r \geq 0} M_r/M_{r+1} \rightarrow 0,
\]
in $R\modlf$, where the injection is the sum of the injection maps $M_{r+1} \hookrightarrow M_r$. From this, we obtain
\[
[M] + [ \bigoplus_{r \geq 1} M_r] = [\bigoplus_{r \geq 1} M_r] + [\bigoplus_{r \geq 0} M_r/M_{r+1}]
\quad \Rightarrow \quad
[M] = \sum_{r \geq 0} [M_r/M_{r+1}],
\]
in $G_0(R\modlf)$. Therefore, we have the following proposition:

\begin{prop}
Let $R$ be as in \cref{prop:RposJH}. The canonical projection 
\[
G_0(R\modlf) \xrightarrow{\simeq} \bGO(R\modlf),
\]
is an isomorphism.
\end{prop}


%% file: sections/triangTopK0.tex

\section{Triangulated asymptotic Grothendieck group}\label{sec:derivedtopK0}

Let $\cT$ be an idempotent complete, triangulated category. Also suppose that $\cT$ admits arbitrary products and coproducts, and these preserve distinguished triangles. 
In this section, we will assume we work with a saturated triangulated full subcategory $\cC \subset \cT$. 
 Saturated means that whenever $Y \cong X \oplus Z \in \cT$ for some $Y \in \cC$, there exists $X', Z' \in \cC$ such that $X' \cong X$ and $Z' \cong Z$.

\smallskip

Following the terminology of~\cite{kellernicolas}, the \emph{Milnor colimit $\mcolim_{r  \geq 0} (f_r) $} of a collection of arrows $\{X_r \xrightarrow{f_r} X_{r+1}\}_{r \in \bN}$ in $\cT$ is the mapping cone fitting inside the following distinguished triangle
\[
\coprod_{r \in \bN} X_r \xrightarrow{1-f_\bullet} \coprod_{r \in \bN} X_r \rightarrow \mcolim_{r  \geq 0} (f_r) \rightarrow 
\]
where the arrow on the left is given by the infinite matrix
\[
1-f_\bullet := 
\begin{pmatrix}
1      & 0       &  0 & 0 & \cdots \\
-f_0 & 1       & 0 & 0 & \cdots  \\
0      & -f_1  & 1 & 0  & \cdots \\
\vdots & \ddots & \ddots & \ddots & \ddots
\end{pmatrix}
\] 

\begin{rem}
Sometimes the Milnor colimit is also called ``homotopy colimit'' in the literature (e.g. in~\cite{qithesis}). 
\end{rem}

There is a dual notion of Milnor limit. Consider a collection of arrows $\{X_{r+1} \xrightarrow{f_r} X_r\}_{r \geq 0}$ in $\cT$. The \emph{Milnor limit} is the object fitting inside the distinguished triangle
\[
\mlim_{r \geq 0} (f_r) \rightarrow \prod_{r \geq 0} X_r \xrightarrow{1 - f_\bullet} \prod_{r \geq 0}  X_r \rightarrow
\]

\begin{rem}
In general, Milnor limits and Milnor colimits do not preserve distinguished triangles. However, since in $\cT$ products and coproducts preserve distinguished triangles, it is easy to see using the 3x3 lemma \cite[Lemma~2.6]{may} that Milnor limits and Milnor colimits also preserve distinguished triangles in our situation. 
\end{rem}

Similarly as before, we say that a Milnor limit or a Milnor colimit is \emph{acceptable} whenever $\bigoplus_{r \geq 0} \cone(f_r) \in \cC$. 

\smallskip

\begin{defn}\label{def:toptriangulatedK0}
The \emph{asymptotic triangulated Grothendieck group} of $\cC \subset \cT$ is given by
\[
\bKO^\Delta(\cC) := K_0^\Delta(\cC) / T(\cC),
\]
where $T(\cC)$ is generated by 
\[
[Y] -  [X] = \sum_{r\geq 0} [E_r],
\]
whenever $\bigoplus_{r \geq 0} E_r \in \cC$, and
\begin{align*}
Y \cong \mcolim\bigl(X = F_0 \xrightarrow{f_0} F_1 \xrightarrow{f_1} \cdots \bigr),
\intertext{is an acceptable Milnor colimit, or}
X \cong \mlim \bigl( \cdots \xrightarrow{f_1} F_1 \xrightarrow{f_0} F_0 = Y \bigr),
\end{align*}
is an acceptable Milnor limit, and 
\[ 
	[E_r] = [\cone(f_r)] \in K_0^\Delta(\cC),
\]
for all $r \geq 0$.
\end{defn}

\subsection{T-structures}

Recall that a \emph{t-structure}~\cite{perversesheaves}  is the datum of two strictly full subcategories $\cC^{\leq 0}$ and $\cC^{\geq 0}$ such that 
\begin{itemize}
\item $\cC^{\leq 0}$ is closed under $[-1]$ and $\cC^{\geq 0}$ is closed under $[1]$ ;
\item $\Hom_\cC(M[1], N) = 0$ for  all $M \in \cC^{\geq 0}$ and $N \in \cC^{\leq 0}$ ;
\item for each $M \in \cC$ there is a distinguished triangle
\begin{equation}\label{eq:tstructtriangle}
M^{\geq 1} \rightarrow M \rightarrow M ^{\leq 0} \rightarrow
\end{equation}
where $M^{\leq 0} \in \cC^{\leq 0}$ and $M ^{\geq 1}[-1] \in \cC^{\geq 0}$.
\end{itemize}

\begin{rem}
Because we chose a chain complex convention for dg-algebras before, we use a reverse notation as the one in the reference~\cite{perversesheaves}. 
\end{rem}

We refer to~\cite{perversesheaves} for the proof of all the following facts. 
The inclusion $\cC^{\leq n} \hookrightarrow \cC$ admits a right adjoint $\tau_{\leq n} : \cC \rightarrow \cC^{\leq n}$  which we call \emph{truncation functor}. Similarly, the inclusion $\cC^{\geq n} \hookrightarrow \cC$ admits a left adjoint $\tau^{\geq n} : \cC \rightarrow \cC^{\geq n}$.  Moreover, any distinguished triangle similar to the one in \cref{eq:tstructtriangle} is isomorphic to
\[
\tau^{\geq 1} M \rightarrow M \rightarrow \tau^{\leq 0} M \rightarrow
\]
and we have
\begin{align*}
\tau^{\leq n} \circ \tau^{\leq m} &\cong \tau^{\leq n}, &
\tau^{\leq n} \circ \tau^{\geq m} &\cong \tau^{\geq m} \circ \tau^{\leq n} \cong 0,
\\
\tau^{\geq n} \circ \tau^{\geq m} &\cong \tau^{\geq m}, &
\tau^{\geq n} \circ \tau^{\leq m} &\cong \tau^{\leq m} \circ \tau^{\geq n}, 
\end{align*}
for $n < m$.  In particular, we obtain a distinguished triangle
\[
\tau^{\geq n} \circ \tau^{\leq n}(M) \rightarrow \tau^{\leq n} M \rightarrow \tau^{\leq n-1} M \rightarrow
\]
for all $n \in \bZ$ and $M \in \cC$. 

\smallskip

The \emph{heart} of the t-structure is  $\cC^\heartsuit := \cC^{\leq 0} \cap \cC^{\geq 0}$, and the \emph{homological functor} is $H^0 := \tau^{\geq 0} \circ \tau^{\leq 0} : \cC \rightarrow \cC^\heartsuit$. One also puts $H^i := H^0 \circ [-i]$. 
It is a well-known fact that $\cC^\heartsuit$ is abelian. In particular given $X \xrightarrow{f} Y \in \cC^\heartsuit$, we have
\begin{align*}
\cok(f) &:= H^0(\cone(f)) \cong \tau^{\leq 0}(\cone(f)), \\
 \ker(f) &:= H^1(\cone(f)) \cong \tau^{\geq 0}(\cone(f)[-1]).
 \end{align*}
 Moreover, a distinguished triangle in $\cC$ gives rise to a long exact sequence in homology. 
One says that a t-structure is \emph{bounded from below} if $\cC = \bigcup_n \cC^{\geq n}$. It is \emph{non-degenerate} if $\bigcap_{n \in \bZ} \cC^{\leq n} = 0 = \bigcap_{n \in \bZ} \cC^{\geq n}$. When $\tau$ is non-degenerate, then $X \cong 0 \in  \cC$ if and only if $H^i(X) \cong 0 \in \cC^\heartsuit$ for all $i \in \bZ$. This also implies that whenever $f : X \rightarrow Y$ induces isomorphisms $H^i(f) : H^i(X) \xrightarrow{\simeq} H^i(Y)$ for all $i \in \bZ$, then $f : X \xrightarrow{\simeq} Y$ is an isomorphism in $\cC$. 

\smallskip

In our context, we will assume that the t-structure extends to $\cT$ such that there is a t-structure on $\cT$ with $\cC^{\leq 0} = \cT^{\leq 0} \cap \cC$ and $\cC^{\geq 0} = \cT^{\geq 0} \cap \cC$. Because we assumed that $\cT$ has arbitrary products and coproducts, its heart $\cT^\heartsuit$ is both AB3 and AB3* by~\cite[Proposition 3.1.2.]{parrathesis}. 

\smallskip

Since $\tau^{\leq 0}$ is a right adjoint, it preserves limits. Following~\cite{parrathesis}, we say that $\tau$ is a \emph{smashing} t-structure whenever $\tau^{\leq 0}$ also preserves coproducts (equivalently when $\cC^{\geq 0}$ is closed under taking coproducts). For example, it is the case whenever the t-structure is compactly generated (see~\cite[Definition 1.6.4.]{parrathesis}). 
Since $\tau^{\geq 0}$ always preserves colimits, we obtain that $H^0$ preserves coproducts whenever $\tau$ is smashing.   
 More generally, if $\cT^\heartsuit$ is closed under coproducts, then it is AB4 \cite[Proposition 3.1.5.]{parrathesis}. This is in particular the case when $\tau$ is smashing. 
 The dual version of these facts also holds. 
 
\begin{lem}
Suppose $\cT^\heartsuit$ is AB4 and AB4*. 
The inclusion $\cC^\heartsuit \subset \cC$ induces a map
\[
\Psi : \bGO(\cC^\heartsuit) \rightarrow \bKO^\Delta(\cC).
\]
\end{lem}

\begin{proof}
There is a map $G_0(\cC^\heartsuit) \rightarrow K_0^\Delta(\cC)$  since short exact sequences in $\cC^\heartsuit$ give rise to distinguished triangles in $\cC$. Thus, we want to show that $J(\cC^\heartsuit) \subset T(\cC)$ and $J^*(\cC^\heartsuit) \subset T(\cC)$. 

We show that $J(\cC^\heartsuit) \subset T(\cC)$, the dual case being similar. Consider a map $X \xrightarrow{f} Y$ in the heart $\cC^\heartsuit$. Then we can view $\ker (f)$ and $\cok(f)$ as objects in $\cC$, and we obtain
\[
 [\cok f] - [\ker f]  = [Y] - [X] = [\cone(f)] \in K_0^\Delta(\cC).
\]
Now consider an acceptable filtered colimit and its corresponding acceptable derived colimit
\begin{align*}
Y &\cong \colim \bigl( X_0 \xrightarrow{f_0} X_1 \xrightarrow{f_1} \cdots \bigr), \\
\tilde Y &\cong {\colim}^1 \bigl( X_0 \xrightarrow{f_0} X_1 \xrightarrow{f_1} \cdots \bigr),
\end{align*}
in $\cC^\heartsuit$.  There is a distinguished triangle
\[
\bigoplus_{r \geq 0} \ker(f_r)[1] \rightarrow  \bigoplus_{r \geq 0} \cone(f_r) \rightarrow \bigoplus_{r \geq 0} \cok(f_r) \rightarrow 
\]
in $\cT$, and thus $\bigoplus_{\geq 0} \cone(f_r) \in \cC$ since $\bigoplus_{r \geq 0} \cok(f_r) \in \cC$ and $\bigoplus_{r \geq 0} \ker(f_r) \in \cC$.
Moreover, there is an exact sequence
\[
0 \rightarrow \tilde Y \rightarrow \coprod_{r \geq 0} X_r \xrightarrow{1-f_\bullet} \coprod_{r \geq 0} X_r \rightarrow Y \rightarrow 0,
\]
in $\cT^\heartsuit$. Thus, we obtain
\begin{align*}
Y &\cong H^0(\cone(1-f_\bullet)), & 
\tilde Y &\cong H^1(\cone(1-f_\bullet)),
\end{align*}
and $ \mcolim \bigl( X_0 \xrightarrow{f_0} X_1 \xrightarrow{f_1} \cdots \bigr) = \cone(1-f_\bullet) \in \cC$. 
In conclusion,
\begin{align*}
[Y] - [\tilde Y] - [X_0] &=  [ \mcolim \bigl( X_0 \xrightarrow{f_0} X_1 \xrightarrow{f_1} \cdots \bigr)] - [X_0]  \\
&= \sum_{r \geq 0} [\cone(f_r)] = \sum_{r \geq 0} ([\cok f_r] - [\ker f_r]),
\end{align*}
in $\bKO^\Delta(\cC)$. 
\end{proof}

\begin{lem}\label{lem:longexactinGO}
Let
\[
\cdots \xrightarrow{f_0} X_1 \xrightarrow{f_1} X_2 \xrightarrow{f_2} X_3 \xrightarrow{f_3} \cdots
\]
be a long exact sequence in an abelian category $\cA$ which is a full subcategory of some AB4 category $\cG$. 
If $\bigoplus_i X_i \in \cA$, then we obtain
\[
 \sum_i [X_{2i}] =  \sum_{i} [X_{2i+1}],
\]
in $G_0(\cA)$.
\end{lem}

\begin{proof}
For each $i \in \bZ$ there is a short  exact  sequence
\[
0 \rightarrow \ker f_i \cong \Image f_{i-1} \rightarrow X_i \rightarrow \Image f_i \rightarrow 0 \in \cA \subset \cG.
\]
Since $\cG$ is AB4 , we obtain two short exact sequences
\begin{align*}
&0 \rightarrow \bigoplus_i \Image f_{2i} \rightarrow \bigoplus_i X_{2i+1} \rightarrow \bigoplus_i \Image f_{2i+1} \rightarrow 0, \\
&0 \rightarrow \bigoplus_i \Image f_{2i-1} \rightarrow \bigoplus_i X_{2i} \rightarrow \bigoplus_i \Image f_{2i} \rightarrow 0.
\end{align*}
Since $\bigoplus_i X_i \in \cA$, we also have $\Image \bigl( \bigoplus_i  f_i  \bigr) \cong \bigoplus_i \Image f_i \in \cA$.
Thus, we obtain
\[
\sum_i [X_{2i}] =  \sum_i \bigl( [\Image f_{2i-1}] + [\Image f_{2i}]\bigr) = \sum_i [X_{2i+1}],
\]
in $G_0(\cA)$.
\end{proof}

\begin{defn}
We say that $\tau$ is \emph{full} if $\bigoplus_{i \in \bZ} H^i(X) \in \cC$ for all $X \in \cC$  and  $\cT^\heartsuit$ is AB4.
\end{defn}

\begin{lem}
If $\tau$ is full and $H^0$ commutes with direct sums, 
then there is a map
\[
\Phi_0 : K_0^\Delta(\cC) \rightarrow G_0(\cC^\heartsuit),
\]
given by
\[
[X] \mapsto [H^*(X)] := \sum_{i \in \bZ} \bigl( [H^{2i}(X)] - [H^{2i+1}(X)] \bigr). 
\]
\end{lem}

\begin{proof}
Since $\cC$ is saturated, we have both $\bigoplus_{i\in \bZ} H^{2i}(X) \in \cC^\heartsuit$ and $\bigoplus_{i\in \bZ} H^{2i+1}(X) \in \cC^\heartsuit$. 
Suppose there is a distinguished triangle
\[
X \rightarrow Y \rightarrow Z \rightarrow 
\]
It induces a long exact sequence  
\[
\cdots \rightarrow H^1(Z) \rightarrow H^0(X) \rightarrow H^0(Y) \rightarrow H^0(Z) \rightarrow \cdots
\]
in $\cC^\heartsuit$. 
Using \cref{lem:longexactinGO}, it implies that
\[
[H^*(X)] + [H^*(Z)] = [H^*(Y)],
\]
in $G_0(\cC^\heartsuit)$. 
\end{proof}

\begin{lem}\label{lem:coneHinC}
Consider $X \xrightarrow{f} Y \in \cT$ and suppose $\cC^\heartsuit$ is thick in $\cT^\heartsuit$. If $\cone(f) \in \cC$, then we have $\cone(H^0(f)) \in \cC$.
If $\tau$ is full, then we also have $\bigoplus_{i \in \bZ} \cone(H^i(f)) \in \cC$. 
\end{lem}

\begin{proof}
There is an exact sequence in $\cT^\heartsuit$
\[
H^1(Y) \rightarrow H^1(\cone(f)) \rightarrow H^0(X) \rightarrow H^0(Y) \rightarrow H^0(\cone(f)) \rightarrow H^{-1}(X) \rightarrow H^{-1}(Y).
\]
Using the epi-mono decomposition we obtain a commutative diagram
\[
\begin{tikzcd}
H^1(Y)  \ar{rr} \ar[twoheadrightarrow]{dr} && H^1(\cone(f)) \ar{rr}  \ar[twoheadrightarrow]{dr}  && H^0(X). \\
&I \ar[hookrightarrow]{ur} && K \ar[hookrightarrow]{ur}&
\end{tikzcd}
\]
Since $H^1(\cone(f)) \in \cC^\heartsuit$, we have $I,K \in \cC^\heartsuit$. Doing the same reasoning on the part on the right, we obtain an exact sequence
\[
0 \rightarrow K \rightarrow H^0(X) \rightarrow H^0(Y) \rightarrow C \rightarrow 0,
\]
with $K,C \in \cC^\heartsuit$. Moreover, there is a distinguished triangle
\[
C \rightarrow \cone(H^0(f)) \rightarrow K \rightarrow 
\]
in $\cT$, and thus $\cone(H^0(f)) \in \cC$.

If $\tau$ is full, we can apply a similar reasoning using the fact that $\bigoplus_{i \in I} H^i(\cone(f)) \in \cC$ and that $\cT^\heartsuit$ is AB4.
\end{proof}

It is well-known that Milnor colimits in the derived category of an AB5 category induce colimits in homology (see for example~\cite{bokstedtneeman}). We obtain a similar result:

\begin{lem}\label{lem:mcolimtocolim}
Suppose $\cC^\heartsuit$ thick in $\cT^\heartsuit$ and is locally AB5.  
If $H^0$ commutes with coproducts and
\[
X \cong \mcolim\bigl(X_0 \xrightarrow{f_0} X_1 \xrightarrow{f_1} \cdots \bigr) \in \cC,
\]
is an acceptable Milnor colimit in $\cC$, then
\[
H^i(X) \cong \colim \bigl(H^i(X_0) \xrightarrow{H^i(f_0)} H^i(X_1) \xrightarrow{H^i(f_1)} \cdots \bigr) \in \cC^\heartsuit,
\]
is an acceptable filtered colimit for all $i \in \bZ$.  If $\tau$ is also full, then
\[
\bigoplus_{i \in \bZ} H^i(X) \cong \colim \bigl(\bigoplus_{i \in \bZ}  H^i(X_0) \xrightarrow{\sum H^i(f_0)} \bigoplus_{i \in \bZ}  H^i(X_1) \xrightarrow{\sum H^i(f_1)} \cdots \bigr),
\]
is an acceptable filtered colimit in $\cC^\heartsuit$. 
The dual also holds whenever $H^0$ commutes with products and $\cC^\heartsuit$ is locally AB5*.
\end{lem}

\begin{proof}
Let
\begin{equation}\label{eq:Hicolim}
Y^i := \colim  \bigl(H^i(X_0) \xrightarrow{H^i(f_0)} H^i(X_1) \xrightarrow{H^i(f_1)} \cdots \bigr).
\end{equation}
We have $Y^i \in \cC^\heartsuit$ by \cref{lem:coneHinC} since $Y^i \cong H^0(\cone(H^i(1-f_\bullet)))$ and $\cone(1-f_\bullet) \cong X \in \cC$. 
There is a commutative diagram
\[
\begin{tikzcd}
H^{i+1}(X) \ar{r}{\delta_{i+1}} & H^i \bigl( \coprod_{r \geq 0}  X_r \bigl) \ar{d}{\vsimeq} \ar{r}{H^i(1 - f_\bullet)} & H^i \bigl(\coprod_{r\geq 0} X_r \bigr)  \ar{d}{\vsimeq} \ar{r} & H^i (X) \ar{r}{\delta_i} & \coprod_{r \geq 0}  H^{i-1}(X_r)
\\
Y^{i+1} \ar{r}{0} & \coprod_{r\geq 0} H^i(X_r) \ar{r}{1 - H^i(f_\bullet)} & \coprod_{r\geq 0} H^i(X_r)  \ar{r}{\varphi_i} & Y^i \ar{r}{0} &  \coprod_{r\geq 0} H^{i-1}(X_r)
\end{tikzcd}
\]
in $\cC^\heartsuit$, where the rows are exact. 
 Thus, $\delta_{i+1} =  0$. Since this does not depend on $i$, we also have $\delta_i = 0$, and by consequence $H^i(X) \cong Y^i$.  
 
 The colimit in \cref{eq:Hicolim} is acceptable since
 \[
 \bigoplus_{r\geq 0} \cone(f_r) \cong \cone( \bigoplus_{r\geq 0} X_r \xrightarrow{f_\bullet} \bigoplus_{r \geq 1} X_r ) \in \cC,
 \]
 which implies  by \cref{lem:coneHinC} that
 \[
 \cone( \bigoplus_{r\geq 0} H^i(X_r) \xrightarrow{H^i(f_\bullet)} \bigoplus_{r \geq 1} H^i(X_r) ) \cong \bigoplus_{r\geq 0} \cone(H^i(f_r)) \in \cC^\heartsuit,
 \]
 so that
 \[
 \bigoplus_{r\geq 0} \cok H^i(f_r) \cong H^0\bigl( \bigoplus_{r \geq 0}  \cone(H^i(f_r)) \in \cC^\heartsuit,
 \]
 and similarly for $\bigoplus_{r \geq 0} \ker H^i(f_r) \in \cC^\heartsuit$.
 
 The case when $\tau$ is full follows from the same arguments, using the fact that direct sums (which we recall are biproducts) commute with filtered colimits and filtered limits. 
 The dual case is similar. 
\end{proof}

Because of \cref{lem:mcolimtocolim}, we introduce the following:

\begin{defn}
We say that a t-structure $\tau$ on $\cC$ is \emph{stable} if
\begin{itemize}
\item $\cC^\heartsuit$ is thick in $\cT^\heartsuit$;
\item $\cC$ is both locally AB5 and locally AB5*;
\item $H^0$ commutes with both products and coproducts in $\cT$.
\end{itemize}
\end{defn}

\begin{lem}\label{lem:rightinvGO}
If $\tau$ is full and stable,
then $\Phi_0 : K^\Delta_0(\cC) \rightarrow G_0(\cC^\heartsuit)$ induces a map
\[
\Phi : \bKO^\Delta(\cC) \rightarrow \bGO(\cC^\heartsuit),
\]
which is the right inverse of $\Psi : \bGO(\cC^\heartsuit) \hookrightarrow  \bKO^\Delta(\cC)$. 
\end{lem}

\begin{proof}
We need to show that $\Phi_0(T(\cC)) \subset \bigl(J(\cC^\heartsuit) + J^*(\cC^\heartsuit)\bigr)$. 
Suppose 
\[
Y \cong \mcolim\bigl(X = F_0 \xrightarrow{f_0} F_1 \xrightarrow{f_1} \cdots \bigr),
\]
is an acceptable Milnor colimit,  
and $[\cone(f_r)] = [E_r] \in K_0^\Delta(\cC)$ for some $\{E_r\}_{r \geq 0}$ with $\bigoplus_{r \geq 0} E_r \in \cC$. 
We write $[\cok H^*(f)] := \sum_{i \in \bZ} \bigl( [\cok H^{2i}(f)] - [\cok H^{2i+1}(f)] \bigr)$ and similarly for $[\ker H^*(f)]$. 
Then we obtain 
\begin{align*}
[H^*(\cone(f_r))] 
 &=  [H^*(F_{r+1})] - [H^*(F_r)]  \\
 &=  [\cok H^{*}(f_r)] - [\ker H^{*}(f_r)], 
\end{align*}
in $G_0(\cC^\heartsuit)$. 
Moreover, we have that
\[
\bigoplus_{i \in \bZ} H^i(Y) \cong \colim \bigl(\bigoplus_{i \in \bZ}  H^i(X) =  \bigoplus_{i \in \bZ}  H^i( F_0) \xrightarrow{\sum H^i(f_0)}  \bigoplus_{i \in \bZ}  H^i(F_1) \xrightarrow{\sum H^i(f_1)} \cdots \bigr),
\] 
is an acceptable colimit in $\cC^\heartsuit$, by \cref{lem:mcolimtocolim}.
Thus, we obtain
\[
[H^*(Y)] - [H^*(X)] = \sum_{r\geq 0} \bigl( [\cok H^*(f_r)] - [\ker H^*(f_r)] \bigr),
\]
in $\bGO(\cC^\heartsuit)$. 
Finally, because $[\cone(f_r)] = [E_r]$ in $K_0^\Delta(\cC)$ we have $[H^*(\cone(f_r))] = [H^*(E_r)]$ in $G_0(\cC^\heartsuit)$. Therefore
\[
[H^*(Y)] - [H^*(X)] = \sum_r [H^*(E_r)],
\]
in $\bGO(\cC^\heartsuit)$. 
The dual case is similar.
\end{proof}

\begin{prop}\label{prop:colimtau}
Suppose $\tau$ is a non-degenerate, full and stable t-structure.
Let $X \in \cC$. 
For all $e \in \bZ$ we have
\[
X \cong \mlim \bigl(\cdots \rightarrow \tau^{\leq e + 1} X \rightarrow  \tau^{\leq e } X \bigr).
\]
\end{prop}

\begin{proof}
The proof is similar to~\cite[Theorem 10.2]{kellernicolas}. 
Let 
\[
Y :=  \mlim \bigl(\cdots \rightarrow \tau^{\leq e + 1} X \rightarrow  \tau^{\leq e } X \bigr).
\]
The maps $X \rightarrow \tau^{\leq e + r} X$ induce a map $h : X \rightarrow Y$. 
For all $i \in \bZ$ we have
\[
H^i(Y) \cong \lim \bigl( \cdots\rightarrow  H^i(\tau^{\leq e + 1} X) \rightarrow H^i(\tau^{\leq e}  X)   \bigr),
\]
thanks to \cref{lem:mcolimtocolim}.
For $e +r \geq i$ then $H^i(\tau^{\leq e + r}  X)  = H^i(X)$ so that $H^i(h) : H^i(X) \cong H^i(Y)$. Thus $h$ is an isomorphism by non-degeneracy of $\tau$. 
\end{proof}

\begin{thm}\label{thm:eqasympKOheart}
If $\tau$ is bounded from below, non-degenerate, full and stable, then
\[
\bKO^\Delta(\cC) \cong \bGO(\cC^\heartsuit),
\]
through $\Phi$. 
\end{thm}

\begin{proof}
We only need to show the surjectivity of $\Psi$ since we already know by \cref{lem:rightinvGO} that $\Psi$ is injective. This follows from \cref{prop:colimtau} and the fact that $\bigoplus_i H^i(X) \in \cC$ for all $X \in \cC$, since it implies that
$[X] =  [H^*(X)] \in \bKO^\Delta(\cC).$
\end{proof}

\subsection{Baric structures}

Inspired by~\cite{baric}, we say that a \emph{baric structure} on a $\bZ^n$-graded triangulated category $\cC$ is the datum of a pair of saturated 
subcategories $\cC_{\preceq 0}$ and $\cC_{\succ 0}$ such that
\begin{itemize}
\item $x^\bg \cC_{\preceq 0} \subset \cC_{\preceq 0}$ for all $\bg \prec 0$ ;
\item $x ^\bg \cC_{\succ 0} \subset \cC_{\succ 0}$ for all $\bg \succ 0$ ;
\item $\Hom_\cC(X,Y) = 0$ whenever $X \in \cC_{\preceq 0}$ and $Y \in \cC_{\succ 0}$ ;
\item for each $X \in \cC$ there is a distinguished triangle
\begin{equation}\label{eq:disttrianglebaric}
X_{\preceq 0} \rightarrow X \rightarrow X_{\succ 0} \rightarrow 
\end{equation}
where $X_{\preceq 0} \in \cC_{\preceq 0}$ and $X_{\succ 0} \in \cC_{\succ 0}$.
\end{itemize}

\begin{prop}
Let  $\cC_{\preceq 0}$ and $\cC_{\succ 0}$ be a baric structure on $\cC$. The inclusion $\cC_{\preceq 0} \hookrightarrow \cC$ admits a right-adjoint $\beta_{\preceq 0} : \cC \rightarrow \cC_{\preceq 0}$ and the inclusion $\cC_{\succ 0} \hookrightarrow \cC$ admits a left-adjoint $\beta_{\succ 0} : \cC \rightarrow \cC_{\succ 0}$. Moreover, for each $X \in \cC$ there is a distinguished triangle
\begin{equation}\label{eq:disttrianglebeta}
\beta_{\preceq 0} X \rightarrow X \rightarrow \beta_{\succ 0} X \rightarrow 
\end{equation}
equivalent to~\cref{eq:disttrianglebaric}.
\end{prop}

\begin{proof}
As in~\cite[Proposition 2.2]{baric}, this is actually similar to the case of a t-structure. Thus we can follow the same arguments as in~\cite[\S1.3.3--1.3.4]{perversesheaves}.
\end{proof}

Then we use similar notations as in \cref{sec:cobaric}.

\begin{lem}\label{lem:betacomp}
For $\bg \preceq \bg'$ there are natural isomorphisms
\begin{align*}
\beta_{\preceq \bg} \circ \beta_{\preceq \bg'} &\cong \beta_{\preceq \bg},  &
\beta_{\preceq \bg} \circ \beta_{\succ \bg '} &\cong   \beta_{\succ \bg '}  \circ \beta_{\preceq \bg}  \cong 0,
\\
\beta_{\succ \bg} \circ \beta_{\succ \bg'} &\cong \beta_{\succ \bg'},  &
\beta_{\succ \bg} \circ \beta_{\preceq \bg '} &\cong  \beta_{\preceq \bg '} \circ \beta_{\succ \bg}.
\end{align*}
\end{lem}

\begin{proof}
Similar as in~\cite[\S1.3.5]{perversesheaves}, and omitted.
\end{proof}

\begin{prop}\label{prop:baricexact}
The functors $\beta_{\preceq \bg}$ and $\beta_{\succ \bg}$ are exact. 
\end{prop}

\begin{proof}
The proof is almost identical to~\cite[Proposition 2.3]{baric}, and omitted.
\end{proof}

\begin{lem}\label{lem:disttriangbetag}
For all $X \in \cC$ and $\bg \in \bZ^n$ there are distinguished triangles
\begin{align*}
&\beta_{\prec \bg} X \rightarrow \beta_{\preceq \bg} X \rightarrow \beta_\bg X \rightarrow \\
&\beta_{\bg} X \rightarrow \beta_{\succeq \bg} X \rightarrow \beta_{\succ \bg} X \rightarrow 
\end{align*}
\end{lem}

\begin{proof}
Similar to \cref{lem:sesbetag}, and omitted.
\end{proof}

\begin{defn}
We say that a baric structure is
\begin{itemize}
\item \emph{non-degenerate} if $\bigcap_{\bg \in \bZ^n} \cC_{\preceq \bg} = 0 = \bigcap_{\bg \in \bZ^n} \cC_{\succ \bg}$;
\item \emph{locally finite} if there is a finite collection of objects $P_1,\dots, P_m$ in $\cC_0$ such that any object in $\cC_0$ is isomorphic to a finite iterated extension of $P_1,\dots,P_m$;
\item \emph{cone bounded}  if for any $X \in \cC$ there exists a \emph{bounding cone} $C_X \subset \bR^{n}$ and a \emph{minimal degree} $\be \in \bZ^n$ such that $\beta_\bg(X) = 0$ for all $\bg - \be \notin C_X$;
\item \emph{stable} if $\beta_{\preceq \bg}$ commutes with acceptable Milnor colimits and limits; 
\item \emph{full} if $\bigoplus_{\bg \in \bZ^n} \beta_\bg X \in \cC$ for all $X \in \cC$; 
\item \emph{strongly c.b.l.f.} if it is non-degenerate, locally finite, cone bounded, stable and full. 
\end{itemize}
\end{defn}

\begin{prop}\label{prop:stabilitycond}
Let $\beta$ be a cone bounded, locally finite baric structure on $\cC$. If $\cC_{\succ 0}$ is stable under acceptable Milnor colimits and acceptable Milnor limits, then $\beta$ is stable. 
\end{prop}

\begin{proof}
Let $\bg \in \bZ^n$. 
Consider an acceptable Milnor colimit 
\[
X \cong \mcolim \bigl( X_0  \xrightarrow{f_0} X_1 \xrightarrow{f_1} \cdots \bigr).
\]
Let 
\[
X' := \mcolim \bigl( \beta_{\preceq \bg}  X_0  \xrightarrow{ \beta_{\preceq \bg}f_0}  \beta_{\preceq \bg}X_1 \xrightarrow{ \beta_{\preceq \bg}f_1} \cdots \bigr).
\]
Because $\beta$ is cone bounded, locally finite and $\bigoplus_{r \geq 0} \cone(f_r) \in \cC$, we have  $\cone(\beta_{\preceq \bg}f_r) = 0$ for $r \gg 0$. Thus,  $X' \cong \beta_{\preceq \bg} X_m$, and $X_r/X_m := \cone(X_m \rightarrow X_r) \in \cC_{\succ \bg}$ for $m \gg 0$ and $r \geq m$. Let
\[
X'' := \mcolim(X_m/X_m \rightarrow X_{m+1}/X_m \rightarrow \cdots).
\]
There is a distinguished triangle
\[
X_m \rightarrow X \rightarrow X''  \rightarrow 
\] 
and thus
\[
X' \rightarrow \beta_{\preceq \bg}  X \rightarrow \beta_{\preceq \bg} X'' \rightarrow 
\] 
is also a distinguished triangle. By hypothesis, we obtain $X'' \in \cC_{\succ \bg}$, and thus $\beta_{\preceq \bg} X'' \cong 0$. Therefore, $X' \cong \beta_{\preceq \bg}  X$.
The dual case is similar, and left to the reader.
\end{proof}

\begin{prop}\label{prop:colimbetabaric}
Suppose $\beta$ is non-degenerate, stable and cone bounded. Let $X \in \cC$. 
 If  $C_X = \{0 = \bc_0 \prec \bc_1 \prec \cdots \}$ is a bounding cone of $X$ and $\be \in \bZ^n$ the minimal degree of $X$, then
\[
X \cong \mcolim \bigl( \beta_{\preceq \be+\bc_0} X  \xrightarrow{\imath_0} \beta_{\preceq \be +\bc_{1}} X \xrightarrow{\imath_1} \cdots \bigr),
\] 
where $\imath_r$ is given by \cref{lem:disttriangbetag}.  
If $\beta$ is full, then this Milnor colimit is acceptable. 
\end{prop}

\begin{proof}
Let $\beta_{\leq r} := \beta_{\preceq \be + \bc_r}$. 
Since we have maps  $j_r : \beta_{\leq r} X \rightarrow X$, we obtain a map $j_\bullet : \coprod_{r \geq 0} \beta_{\leq r} X \rightarrow X$. Thus, the square on the left commutes in the following diagram
\[
\begin{tikzcd}
\coprod_{r \geq 0} \beta_{\leq r} X \ar{r}{1-\imath_\bullet} \ar{d} & 
\coprod_{r \geq 0} \beta_{\leq r} X \ar{r}  \ar{d}{j_\bullet} & 
\mcolim_{r\geq 0} (\imath_r)  \ar[dashrightarrow]{d}{\exists h}   \ar{r} & {}
\\
0 \ar{r} & 
X \ar[equals]{r} &
X  \ar{r} & {}
\end{tikzcd}
\]
inducing the vertical map $h$. 
We put $R := \cone(h)$. 
Let $\bg \in \bZ^n$. 
Since for $r \gg 0$ we have $\bc_r - \be \succ \bg$, we obtain that 
$\beta_{\preceq \bg} \mcolim_{r\geq 0} (\imath_r) \cong \beta_{\preceq \bg} X$. Therefore, $\beta_{\preceq \bg} R \cong 0$, and thus $\beta_{\succ \bg} R \cong R$.
Since $\bg$ is arbitrary, it means $R \in \bigcap_{\bg \in \bZ^n} \cC_{\succ \bg}$. Because $\beta$ is non-degenerate, we conclude that $R \cong 0$ and $X \cong \mcolim_{r\geq 0}(\imath_r)$. 
\end{proof}

\subsubsection{C.b.l.f. iterated extensions}

In a $\bZ^n$-graded triangulated category $\cT$, we define the notion of \emph{c.b.l.f. coproduct} as follows: 
\begin{itemize}
\item take a a finite collection of objects $\{K_1,\dots, K_m\}$ in $\cT$; 
\item consider a coproduct of the form
\begin{align*}
&\coprod_{\bg \in \bZ^n} x^{\bg} (K_ {1,\bg} \oplus \cdots \oplus K_{m,\bg}), &
&\text{ with }&
K_{i,\bg} &= \bigoplus_{j = 1}^{k_{i,\bg}} K_i[h_{i,j,\bg}],
\end{align*}
where $k_{i,\bg} \in \bN$ and $h_{i,j,\bg} \in \bZ$;
\item there exists a cone $C$ compatible with $\prec$, and $\be \in \bZ^n$ such that for all $j$ we have $k_{j,\bg} = 0$ whenever $\bg -  \be \notin C$;
\item there exists $h \in \bZ$ such that $h_{i,j,\bg} \geq h$ for all $i,j,\bg$. 
\end{itemize}

We  suppose that $\cC$ is a c.b.l. additive category (for c.b.l.f. coproducts defined as above). 

\begin{defn}\label{def:cblfitext}
Let $\{K_1, \dots, K_m\}$ be a finite collection of objects in $\cC$, and let $\{E_r\}_{r \in \bN}$ be a family of direct sums of $\{K_1, \dots, K_m\}$ such that $\bigoplus_{r \in \bN} E_r$ is a c.b.l.f. direct sum of $\{K_1, \dots, K_m\}$. Let $\{M_r\}_{r \in \bN}$  be a collection of objects in $\cC$ with $M_0 = 0$, such that they fit in distinguished triangles
\[
M_r \xrightarrow{f_r} M_{r+1} \rightarrow E_r \rightarrow 
\]
Then we say that an object $M \in \cC$ such that $M \cong_{\cT} \mcolim_{r\geq 0} (f_r)$ in $\cT$ is a \emph{c.b.l.f. iterated extension of  $\{K_1, \dots, K_m\}$}. The sequence $t_{i,\bg} = \{h_{i,1,\bg}, \dots, h_{i,k_{i,\bg},\bg}\}$ of occurences of $x^\bg K_i[h_{i,j,\bg}]$ as a direct summand in $\bigoplus_{r \in \bN} E_r$ is called the \emph{degree $\bg$ multiplicity} of $K_i$.
\end{defn}

\begin{prop}
If $\cC$ admits c.b.l.f. iterated extensions and $\beta$ is a cone bounded, locally finite baric structure on $\cC$, then $\beta$ is full.
\end{prop}

\begin{proof}
Similar to \cref{lem:cblisfull}, and omitted.
\end{proof}

\subsection{Topological Grothendieck group}

Mimicking Achar--Stroppel construction~\cite{acharstroppel}, we define the following:

\begin{defn}
The \emph{topological Grothendieck group} of $(\cC,\beta)$ is
\[
\bKO^\Delta(\cC,\beta) := K^\Delta_0(\cC)/T(\cC,\beta),
\]
where
\begin{equation}\label{eq:defTA}
T(\cC,\beta) := \{ [X] \in K_0^\Delta(\cC) \ |\  [\beta_{\preceq \bg}X] = 0 \text{ for all $\bg \in \bZ^n$} \}.
\end{equation}
It is endowed with a canonical topology given by using $\{\bKO^{\Delta}(\beta_{\succ \be}\cC, \beta)\}_{\be \in \bZ^n}$ as basis of neighborhoods of zero.
\end{defn}

As before, if $\beta$ is non-degenerate, then the canonical topology on $\bKO^\Delta(\cC,\beta)$ is Hausdorff. 

\begin{defn}
We say that a graded functor $F : (\cC,\beta) \rightarrow(\cC',\beta')$  has \emph{finite amplitude} if there exists $|F| \in \bZ^n$ such that 
$
F\cC_{\succ \bg} \subset \cC'_{\succ \bg + |F|}.
$
\end{defn}

\begin{prop}\label{prop:K0ASinducedmap}
Let $F : (\cC,\beta) \rightarrow(\cC',\beta')$ be a $\bZ^n$-homogeneous exact functor. If $F$ has finite amplitude, then it induces a continuous map
\[
[F] : \boldsymbol K^\Delta_0(\cC,\beta)  \rightarrow \boldsymbol K^\Delta_0(\cC',\beta'),
\]
by $[F][X] := [F(X)]$. 
\end{prop}

\begin{proof}
The proof is similar to \cref{prop:abfinampinduces}.
\end{proof}

\begin{prop}\label{prop:baricisoKO}
Suppose $\beta$ is a  full, non-degenerate, stable, cone bounded baric structure on $\cC$.
There is a surjection
\[
 \bKO^\Delta(\cC,\beta) \twoheadrightarrow \bKO^\Delta(\cC),
\]
induced by the identity on $K_0^\Delta(\cC)$. 
If  $\beta$ is locally finite, then it is an isomorphism
\[
 \bKO^\Delta(\cC,\beta) \cong \bKO^\Delta(\cC). 
\]
\end{prop}

\begin{proof}
The proof is similar to the one of \cref{prop:topGOisasympGO}. We obtain a surjective map
\[
 \bGO(\cC,\beta) \twoheadrightarrow \bGO(\cC),
\]
thanks to \cref{prop:colimbetabaric}.
Then, suppose $\beta$ is locally finite and consider an acceptable Milnor colimit
\[
Y \cong \mcolim\bigl(X = F_0 \xrightarrow{f_0} F_1 \xrightarrow{f_1} \cdots \bigr),
\]
and $\bigoplus_{r \geq 0} E_r \in \cC$ such that $[E_r] = [\cone(f_r)]$ in $K_0^\Delta(\cC)$ for all $r \geq 0$.
Because $\beta$ is locally finite we have $\beta_{\preceq \bg} E_r = \beta_{\preceq \bg} \cone(f_r) = 0$ for $r \gg 0$. Hence, for $r' \gg 0$, we have
\begin{align*}
[\beta_{\preceq \bg} Y] - [\beta_{\preceq \bg} X] = \sum_{r = 0}^{r'} [\beta_{\preceq \bg}\cone(f_r)] = \sum_{r = 0}^{r'} [\beta_{\preceq \bg}E_r],
\end{align*}
in $K_0^\Delta(\cC)$. 
The dual case is similar, concluding the proof. 
\end{proof}


%% file: sections/cblfDgAlg.tex

\section{C.b.l.f. positive dg-algebras}\label{sec:cblfdgalg}

For this section, we restrict to derived categories that arise from  $\bZ^n$-graded dg-$\Bbbk$-algebras $(A,d_A)$, where $A = \bigoplus_{(h,\bg) \in \bZ \times \bZ^n} A_\bg^h$, and where $\Bbbk$ is a field. We will consider c.b.l. additive subcategories of $\cT := \cD(A,d_A)$. 
In this context, an $(A,d_A)$-module 
is said to be c.b.l.f. dimensional  (resp. \emph{c.b.l.f. generated}) if it is c.b.l.f. dimensional (resp. c.b.l.f. generated) for the $\bZ^n$-grading, and bounded from below for the homological grading. 

\smallskip

It is a known fact, and actually not hard to show (see for example~\cite[Lemma 5.2]{kelleryang}), that if $(A,d_A)$ is positive for the homological grading, i.e. $A_\bg^h = 0$ whenever $h < 0$, then $\cD(A,d_A)$ has a canonical non-degenerate t-structure. We take $\cD(A,d_A)^{\leq 0}$ as given by the dg-modules with homology concentrated in non-positive degree, and $\cD(A,d_A)^{\geq 0}$ by those in non-negative degree. Moreover, the truncation functors are given by the usual intelligent truncations: 
\begin{align*}
\tau^{\geq 0} (M,d_M) &:= \bigl(\bigoplus_{h > 0} M^h_* \oplus \ker(M^0_* \xrightarrow{d_M} M^{-1}_*)  ,d_M\bigr), \\
\tau^{\leq 0} (M,d_M) &:= \bigl(\bigoplus_{h < 0} M^h_* \oplus \cok(M^1_* \xrightarrow{d_M} M^0_*)  ,d_M\bigr),
\end{align*}
for $(M,d_M)$ an $(A,d_A)$-module with $M = \bigoplus_{(h,\bg)} M_\bg^h$.  
Then the homological functor $H^0$ coincides with the usual one, and there is an equivalence of abelian categories
\[
\cD(A,d_A)^\heartsuit \cong H^0(A,d_A)\amod.
\]

Let $\cD^{cblf}(A,d_A)$ denotes the triangulated full subcategory of $\cD(A,d_A)$ consisting of $(A,d_A)$-modules having a c.b.l.f. dimensional homology (meaning the homological degree is also bounded from below). The t-structure of $\cD(A,d_A)$ restricts to $\cD^{cblf}(A,d_A)$. 
It becomes bounded from below and we obtain
\[
\cD^{cblf}(A,d_A)^\heartsuit \cong H^0(A,d_A)\modlf.
\]
Note that the restriction of $\tau$ on $\cD^{cblf}(A,d_A)$ gives a full t-structure, since the homology of a c.b.l.f. dimensional dg-module is c.b.l.f. dimensional for the $\bZ^n$-grading. 

\begin{prop}\label{prop:canontauondgisstable}
Let $(A,d_A)$ be a $\bZ^n$-graded dg-algebra. If $A$ is positive for the homological grading and $H^0(A,d_A)$ is locally finite dimensional, then the canonical t-structure on  $\cD^{cblf}(A,d_A)$ is stable. 
\end{prop}

\begin{proof}
Clearly,  $H^0(A,d_A)\modlf$ is thick in $H^0(A,d_A)\amod$, since both submodules and quotients of a c.b.l.f. dimensional module are c.b.l.f. dimensional. 
Since the category of $\bZ^n$-graded $\Bbbk$-vector spaces if both AB4 and AB4*, we obtain that $H^0$ commutes with both products and coproducts in $\cD^{cblf}(A,d_A)$. 
It only remains to show that $H^0(A,d_A)\modlf$ is both locally AB5 and locally AB5*. This follows from the same arguments as in the proof of \cref{prop:RposJH}.
\end{proof}

\begin{cor}\label{cor:dgheart}
Let $(A,d_A)$ be a $\bZ^n$-graded dg-algebra. If $A$ is positive for the homological grading and $H^0(A,d_A)$ is locally finite dimensional, then
\[
\bKO^\Delta(\cD^{cblf}(A,d_A)) \cong \bGO(H^0(A,d_A)\modlf).
\]
\end{cor}

\begin{proof}
This follows from \cref{thm:eqasympKOheart}, using \cref{prop:canontauondgisstable}.
\end{proof}

We investigate the case of such dg-algebras where $\cD^{cblf}(A,d_A)$ carries a natural baric structure, and for which we can explicitly compute the Grothendieck group.

\begin{defn}\label{def:cblfposdgalg}
We say that $(A,d_A)$ is a \emph{c.b.l.f. positive dg-algebra} if
\begin{enumerate}
\item $A$ is positive for the homological degree: $A_\bg^h  = 0$ whenever $h < 0$;
\item $A$ is positive, c.b.l.f. dimensional for the $\bZ^n$-grading: 
$\bigoplus_{h \in \bN} A_\bg^h$ is finite dimensional for each $\bg \in \bZ^n$, and $\sum_{(h,\bg)} \dim(A_\bg^h) x^{\bg} \in \bZ_{C_A} \llbracket{x_1, \dots, x_n}\rrbracket$ for some cone $C_A$ compatible with $\prec$;
\item $A_0^0$ is semi-simple;
\item $A_0^h = 0$ whenever $h > 0$. \label{item:zeroconcdgalg}
\end{enumerate}
\end{defn}

We write $A_0 := A_0^0$ and $A_{\succ 0} := \bigoplus_{\{(h,\bg) | \bg \succ 0\}} A_\bg^h$.  Since $d_A(A_0) = 0$, indecomposable relatively projective $(A,d_A)$-modules are in bijection with indecomposable projective $A$-modules. Let $\{P_i := A e_i\}_{i \in I}$ be the finite set of indecomposable relatively projective modules, up to shift and isomorphism. For a shifted copy $P_i' \cong x^\bg P_i [h]$ of $P_i$, we write $\deg_x(P_i') := \bg$ and $\deg_h(P_i') := h$. Note that $\deg_x(P_i')$ coincides with the minimal $\bZ^n$-degree of $P_i'$.

We see that $A_{\succ 0}e_i$ is a sub-dg-module of $P_i$, since $A_0^1 = 0$. Then, there is a simple dg-module $S_i := P_i/A_{\succ 0}e_i$. We write $\deg_*(S_i') := \deg_*(P_i')$ whenever $S_i' \cong P_i'/A_{\succ 0}P_i'$.

\begin{rem}\label{rem:exttononpositivecblfdgalg} 
For the remaining of the paper, we can drop the hypothesis that $A$ is positive for the $\bZ^n$-grading. Instead, we can assume it decomposes as a finite direct sum $A \cong \bigoplus_{j \in J} \tilde P_j$ (as graded module over itself) where $\tilde P_j$ is a graded shift of some indecomposable $P_i := A e_i$, such that each of these $P_i$ is positive, c.b.l.f. dimensional for the $\bZ^n$-grading. 
This ensures that $\Hom(P'_{i'}, P_i) = 0$ whenever $\deg_x(P'_{i'})  \prec \deg_x(P_i)$ or $\deg_h(P'_{i'}) < \deg_h(P_i)$. 
\end{rem}

\begin{rem}
If we remove point (\ref{item:zeroconcdgalg}.), then $P_i := Ae_i$ could be acyclic. In that case, there is a quasi-isomorphism $(A,d_A) \xrightarrow{\simeq} (A/Ae_iA,d_A)$. 
Thus, we could weaken the hypothesis in (\ref{item:zeroconcdgalg}.) so that it is respected only after removing all acyclic relatively projectives.  
\end{rem}

\subsection{C.b.l.f. cell modules}

We extend the notion of finite cell module (see \cref{def:finitecellmodule}) to c.b.l.f.

\begin{defn}\label{def:cblcompact}
An $(A,d_A)$-module is called \emph{c.b.l.f. cell module}  if it satisfies property~(P) and if it is, as graded $A$-module, at the same time c.b.l.f. generated for the $\bZ^n$-grading and bounded from below for the homological grading.  
\end{defn}

Our goal now will be to prove that any object in $\cD^{cblf}(A,d_A)$ is quasi-isomorphic to a c.b.l.f. cell module. 
Proving this statement ends up being quite similar to constructing a property (P) replacement of a dg-module in the non-graded case (see \cite[\S3.1]{keller}). 
We still give the details for the record. 

\begin{prop}
Any $(A,d_A)$-module $(M,d_M)$ in $\cD^{cblf}(A,d_A)$ is quasi-isomorphic to a c.b.l.f. dimensional $(A,d_A)$-module. 
\end{prop}

\begin{proof}
Let $\{[x_r]\}_{r \in R}$ be a $\Bbbk$-basis of $H(M,d_M)$. For each $r \in R$, let $i_r \in I$ be such that $e_{i_r}x_r = x_r$. Consider $M' := \bigcup_{r \in R} P_{i_r} x_r \subset M$. 
Since $P_{i_r}$ is c.b.l.f. dimensional and positive, and $H(M,d_M)$ is c.b.l.f. dimensional, so is $M'$. 
Moreover, the inclusion $M' \xrightarrow{\simeq} M$ is a quasi-isomorphism, since the induced map $H(M',d_M) \rightarrow H(M,d_M)$ is surjective. 
\end{proof}

\begin{prop}\label{prop:cblfprojarecell}
Let $(M,d_M)$ be an $(A,d_A)$-module which is projective c.b.l.f. generated as $A$-module. Then $(M,d_M)$ is a c.b.l.f. cell module.
\end{prop}

\begin{proof}
As $A$-module, $(M,d_M)$ decomposes as a c.b.l.f. direct sum of shifted indecomposables $P_i$. The generators of $M$ are contained in $\be + C_M$ for some $\be \in \bZ^{n}$ and cone $C_M = \{0 = \bc_0 \prec \bc_1 \prec \cdots \}$ compatible with $\prec$. 
 Moreover, the differential $d_M$, when restricted to $x^\bg P_i[h] \rightarrow x^{\bg'} P_j [h']$, can be non zero only if $h > h'$ and $\bg \succeq \bg'$.  Then we construct $M_{r}$ as the sub-dg-module of $(M,d_M)$ with underlying $A$-module given by the direct sum of $x^\bg P_i [h]$ with $\bg \prec \be + \bc_r$. It yields an exhaustive filtration 
 \[
 0 = M_0 \subset M_1 \subset M_2 \subset \cdots \subset M_r \subset M_{r+1} \subset \cdots \subset M,
 \]
 of $(M,d_M)$ such that $M_{r+1}/M_r$ is, as $A$-module, a finite direct sum of $x^{\be+ \bc_r} P_i[h]$ for different $h \in \bZ$. Since the sum is finite, we can further refine the filtration by filtering on the homological degree.  
\end{proof}

Because any c.b.l.f. cell module admits as underlying graded module a projective c.b.l.f. generated $A$-module, we also obtain the following result from the proof of \cref{prop:cblfprojarecell}.

\begin{prop}\label{prop:orderedfiltP}
Let $M$ be a c.b.l.f. cell module, and $C_M = \{0 = \bc_0 \prec \bc_1 \prec \bc_2 \prec \cdots\}$ be a $\bZ^n$-bounding cone of $M$, and let $\be$ be the minimal $\bZ^n$-degree of $M$. There exists an exhaustive filtration
\[
0 =  F_0 \subset F_1 \subset  F_2 \subset \cdots \subset  F_r \subset  F_{r+1} \subset \cdots \subset M,
\]
such that $ F_{r+1}/  F_{r}$ is isomorphic (as $A$-module) to a finite direct sum of indecomposable relatively projective $\{P_i'\}$ with $\deg_x(P_i') = \bc_r + \be$. This filtration is unique in $(A,d_A)\amod$ whenever $C_M$ is fixed.
\end{prop}

\begin{exe}
In order to make things clearer, we will try to get an insight of the typical look of a c.b.l.f. cell module $M$ when there is only one extra grading (thus $n =1$), and only one indecomposable relatively projective module $P_0$ of degree $0$. We suppose that $P_0$ appears with minimal $x$-degree $g \in \bZ^n$ and homological degree $h \in \bZ$ in the filtration of $M$. Write $P := x^g P_0[h]$. Then we can visualize $M$ as:
\[
\begin{tikzcd}
\vdots  & \vdots  & \vdots  &\reflectbox{$\ddots$} \\
\oplus_{\alpha_{2,0}} P[2] \ar{d} &\oplus_{\alpha_{2,1}}x P[2]  \ar{d}  \ar{ld} & \oplus_{\alpha_{2,2}}x^{2} P[2]  \ar{d}  \ar{ld}  \ar{lld} & \dots \\
\oplus_{\alpha_{1,0}} P[1] \ar{d} &\oplus_{\alpha_{1,1}}x P[1]  \ar{d}  \ar{ld} & \oplus_{\alpha_{1,2}}x^{2} P[1]  \ar{d}  \ar{ld}  \ar{lld} & \dots \\
\oplus_{\alpha_{0,0}} P &\oplus_{\alpha_{0,1}}x P & \oplus_{\alpha_{0,2}}x^{2} P & \dots
\end{tikzcd}
\]
with each $\alpha_{i,j} \in \bN$. For $j$ fixed, we have $\alpha_{k,j} = 0$ for $k \gg 0$. Note that for $i$  fixed, then we can have $\alpha_{i,j} > 0$ for all $j \geq n_0$. A composition of arrows in the diagram represents where the differential can be non zero.
Then, the ordered filtration of \cref{prop:orderedfiltP} takes bigger and bigger vertical slices of this diagram, starting from the left. Since the differential only goes to bottom/left, it is easy to see from here that these vertical slices are sub-dg modules. Moreover, because of  $\alpha_{k,j} = 0$ for $k \gg 0$, the quotient of a slice by another one is a finite direct sum of shifted copies of~$P$ (as graded module).
\end{exe}

The following lemma is inspired by~\cite[\href{https://stacks.math.columbia.edu/tag/09KN}{Lemma 09KN}]{stacks-project}.

\begin{lem}\label{lem:cblfcofcover}
Let $(M,d_M)$ be a c.b.l.f. dimensional $(A,d_A)$-module with bounding cone $C_M$ and minimal $\bZ^n$-degree $\be$, and let $C_A$ be a bounding cone of $(A,d_A)$. There exists a c.b.l.f. cell module $(P,d_P)$ with bounding cone $C_A+C_M$  such that:
\begin{enumerate}
\item there is an epimorphism $\pi : (P,d_P) \rightarrow (M,d_M)$ in $(A,d_A)\amod$;
\item $\ker(d_P) \rightarrow \ker(d_M)$ is surjective;
\item $\deg_{\bx}(\ker (\pi)) \succ \deg_{\bx}((M,d_M))$ or $\deg_{h}(\ker (\pi)) > \deg_{h}((M,d_M))$;
\item $\dim (\ker(\pi)_\be^h) \leq \sum_{(a+b=h)} \dim (M_\be^a) \dim (A_0^b)$.
\end{enumerate}
\end{lem}

\begin{proof}
Take $t \in \bZ$ 
such that $M = \bigoplus_{\{(\bg, h) | \bg \succeq \be \lor h \geq t\}} M_\bg^h$ and $M_\be^t \neq 0$. We also have $\ker(d_M) = \bigoplus_{\{(\bg, h) | \bg \succeq \be \lor h \geq t\}} \ker(d_M)_\bg^h$.

For each $(\bg, h)$ we can choose a $\Bbbk$-basis $\{m_{\bg,h}^j\}_{j \in J_{\bg,h}}$ of $M_{\bg}^h$ such that there is $k_j \in I$ such that $e_{k_j}  m_{\bg,h}^j = m_{\bg,h}^j$. This is possible since given any basis $\{m_j\}$ of $M_{\bg,h}$ then $\brak{e_k m_j}_{k \in I, j \in J_{\bg,h}}$ is a generating family of $M_{\bg,h}$. Note that $e_i m_{\bg,h}^j = 0$ whenever $i \neq k_j$. 
We choose a similar basis $\{x_{\bg,h}^j\}_{j \in K_\bg}$ for $\ker(d_M)_{\bg,h}$.

Moreover, the kernel of the surjection $A_0 e_{k_j} \twoheadrightarrow A_0 e_{k_j} m_{\be,t}^j \subset M_{\be,t} m_{\be,t}^j $ must be zero by semi-simplicity of $A_0$. Therefore, the surjection is an isomorphism. Then, since $M_\be^t$ is finitely generated, it is semi-simple over $A_0$. Thus, 
 we can extract a subset $J'_{\be,t} \subset J_{\be,t}$ such that $\bigoplus_{j \in J'_{\be,t}} A_0 e_{k_j} m_{\be,t}^j \cong M_{\be,t}$. 
We replace $J_{\be,t}$ by $J'_{\be,t}$ for the remaining of the proof.  This will ensure that condition (3) is respected.

We construct
\[
P' := \bigoplus\limits_{\substack{(h,\bg) \in \bZ \times \bZ^n \\ j \in J_{\bg,h}}} P_{k_j} \cdot m_{\bg,h}^j,
\]
where $P_{k_j} \cdot m_{\bg,h}^j$ is the relatively projective module given by $P_{k_j}$ shifted in $x$-degree so that $\deg_x(P_{k_j} \cdot m_{\bg,h}^j) = \deg_x(m_{\bg,h}^j)$.
We set 
\[
d_{P'}(a \cdot m_{\bg,h}^j) := d_A(a) \cdot  m_{\bg,h}^j + (-1)^{\deg_h(a)} a \cdot d_M(m_{\bg,h}^j),
\]
for all $a \in P_{k_j}$. Then, since $M$ is c.b.l.f. dimensional,  $(P', d_{P'})$ is a c.b.l.f. cell module by \cref{prop:cblfprojarecell}.
The obvious surjection $P' \twoheadrightarrow M$ makes $(P', d_{P'})$ respect point (1). Point (4) is clearly respected by construction, and point (3) comes from the fact the surjection restricts to an isomorphism $(P')_\be^t \cong M_\be^t$. 
In order to also respect point (2) we add
\[
P := P' \oplus \bigoplus_{\bg \succ \be \lor h > t} P_{k_j} \cdot x_{\bg,h}^j,
\]
with $d_P$ induced by $d_{P'}$ and $d_P(a \cdot x_{\bg,h}^j) = d_A(a) \cdot x_{\bg,h}^j$. Then, $(P,d_P)$ respects all the requirements, concluding the proof.
\end{proof}

We will also need to following result:

\begin{prop}
 \cite[\href{https://stacks.math.columbia.edu/tag/09IZ}{Lemma 09IZ}]{stacks-project}
\label{prop:exactseqdg}
Let $(M,d_M)$ be a $(\bZ,0)$-module. Let
\[
\cdots \xrightarrow{f_3} (A_2, d_2) \xrightarrow{f_2} (A_1, d_1) \xrightarrow{f_1} (A_0, d_0) \rightarrow (M, d_M) \rightarrow 0,
\]
be an exact sequence, such that  
\[
\cdots \rightarrow \ker(d_2) \rightarrow \ker(d_1) \rightarrow \ker(d_0) \rightarrow \ker(d_M) \rightarrow 0,
\]
is exact as well. Then define the $(\bZ,0)$-module $(T,d_T)$ by $T^h := \bigoplus_{(a+b=h)} A_a^b$ and $d_T(x) := f_a(x) + (-1)^a d_a(x)$ for $x \in A_a^b$ and $f_0 = 0$. There is a surjective quasi-isomorphism
\[
(T, d_T) \xrightarrow{\simeq} (M,d_M),
\]
induced by $(A_0, d_0) \twoheadrightarrow (M,d_m)$.
\end{prop}

\begin{thm}\label{thm:cblfdimiscell}
Let $(A,d_A)$ be a c.b.l.f. positive dg-algebra. 
Any c.b.l.f. dimensional $(A,d_A)$-module $(M,d_M)$ admits a quasi-isomorphic cover by a c.b.l.f. cell module $(P,d_P)$.
\end{thm}

\begin{proof}
The proof is similar to the one of~\cite[\href{https://stacks.math.columbia.edu/tag/09KP}{Lemma 09KP}]{stacks-project}.

We set $(M_0, d_{M_0}) := (M,d_M)$. Then we construct inductively for each $i > 0$ short exact sequences
\[
0 \rightarrow (M_i, d_{M_i}) \rightarrow( P_i, d_{P_i}) \xrightarrow{\pi_i} (M_{i-1}, d_{M_{i-1}}) \rightarrow 0,
\]
with $(M_i,d_{M_i}) :=  \ker(\pi_i)$, using \cref{lem:cblfcofcover}. Thus, we get a resolution
\begin{equation}\label{eq:cblfresolution}
\begin{tikzcd}[column sep = 1ex]
\dots  \ar[swap,twoheadrightarrow]{dr}{\pi_3}  \ar{rr}{f_3} && P_2  \ar{rr}{f_2} \ar[swap,twoheadrightarrow]{dr}{\pi_2} && P_1  \ar{rr}{f_1}  \ar[swap,twoheadrightarrow]{dr}{\pi_1}  && P_0  \ar[twoheadrightarrow]{rrrr}{\pi_0} &&&& M  \ar{rrr} &&& 0. \\
&\ker(\pi_2)  \ar[hookrightarrow]{ur}&&\ker(\pi_1) \ar[hookrightarrow]{ur}&&\ker(\pi_0)  \ar[hookrightarrow]{ur}&&
\end{tikzcd}
\end{equation}
Then we put
\[
P := \bigoplus_{i \geq 0} P_i,
\]
with $\bZ^n$-grading induced by $P_i$, and homological grading given by $P^h = \oplus_{(a+b=h)} P_{a}^b$.
It is a c.b.l.f. generated projective $A$ module thanks to points (3) and (4) from \cref{lem:cblfcofcover}, and the fact that $A_0^h = 0$ whenever $h > 0$. 
We equip $P$ with a differential $d_P$ by setting
\[
d_P(x) = f_{a}(x) + (-1)^a d_{P_{a}}(x),
\]
for each $x \in P^b_{a}$. 
Therefore, because of \cref{prop:cblfprojarecell}, $(P,d_P)$ is a c.b.l.f. cell module, coming with an epimorphism $\bar \pi : (P,d_P) \rightarrow (M,d_M)$ induced by $(M_0, d_{M_0}) \twoheadrightarrow (M,d_M)$. 

Moreover, point (2) of \cref{lem:cblfcofcover} tells us that $\pi_{i+1}(\ker d_{P_{i+1}}) =  \ker(d_{P_i}) \cap \ker (\pi_i)$. Thus, we have an exact sequence
\[
\cdots \rightarrow \ker(d_{P_2}) \rightarrow \ker(d_{P_1}) \rightarrow \ker(d_{P_0})  \rightarrow \ker(d_M) \rightarrow 0.
\]
Then, using the exactness of~\eqref{eq:cblfresolution}, we conclude by \cref{prop:exactseqdg} that $\bar \pi$ is a quasi-isomorphism. 
\end{proof}

Let $\cD^{cblc}(A,d_A)$ be the triangulated full subcategory of $\cD(A,d_A)$ given by objects quasi-isomorphic to c.b.l.f. iterated extensions of compact objects in $\cD(A,d_A)$.

\begin{lem}\label{lem:cblfextarecell}
Let $(M,d_M)$ be a c.b.l.f. iterated extension of some c.b.l.f. cell modules $\{(K_1,d_1), \dots, (K_m,d_m)\}$ where $K_i$ appears with degree $\bg$ multiplicity $t_{i,\bg} = \{h_{i,1,\bg, \dots, h_{i,k_{i,\bg},\bg}} \}$. Then $(M,d_M)$ is quasi-isomorphic to a c.b.l.f. cell module $(M',d_{M'})$, having the c.b.l.f. direct sum $\bigoplus_{i,\bg} x^{\bg} \bigoplus_{s=0}^{k_{i,\bg}} K_i[h_{i,j,\bg}]$ as underlying graded $A$-module.
\end{lem}

\begin{proof}
Let $\{K_1, \dots, K_m\}$ be a finite set of c.b.l.f. cell modules, and $\bigoplus_{r\in\bN} E_r$ be a c.b.l.f. direct sum of elements in $\{K_1, \dots, K_m\}$. Consider a collection of iterated extensions $\{M_r\}_{r \in \bN}$, $M_0 = 0$, given by distinguished triangles
\[
M_r \xrightarrow{f_r} M_{r+1} \rightarrow E_r \rightarrow M_r[1].
\]
Take $M := \mcolim_{r \geq 0} (f_r)$. We will prove that $M$ is quasi-isomorphic to a c.b.l.f. cell module. 
We can replace inductively each $M_r$ by a quasi-isomorphic $M_{r+1}' := \cone(E_r[-1] \rightarrow M_r')$. The module $M_{r+1}'$ decomposes as
\[
M'_{r+1} \cong
\begin{tikzcd}[row sep=0, column sep = 0]
E_r  \arrow[loop,out=240,in=300,swap,looseness=5]  \ar[bend left=30]{rr} & \oplus & E_{r-1}  \arrow[loop,out=240,in=300,swap,looseness=5]  \ar[bend left=30]{rr}   & \oplus & \llap{$\phantom{E_i}$} \dots  \ar[bend left=30]{rr}   & \oplus & E_{0}, \arrow[loop,out=240,in=300,swap,looseness=5] 
\end{tikzcd}
\]
where the arrows represent the possible action of the differential on $M_{r+1}'$. Thus, we have an inclusion $M_r' \subset M'_{r+1}$ as $(A,d_A)$-modules. Taking the direct limit of this filtration in $(A,d_A)\amod$ yields an $(A,d_A)$-module $(M',d_{M'})$, which decomposes as $M' \cong \bigoplus_{r\in  \bN} E_r$ as graded $A$-module, and which is quasi-isomorphic to $(M,d_M)$.
Then $M'$ decomposes as $A$-module as a c.b.l.f. direct sum of projective c.b.l.f. generated $A$-modules, and thus is projective c.b.l.f. generated itself. By \cref{prop:cblfprojarecell}, $(M',d_{M'})$ is a c.b.l.f. cell module.
\end{proof}

In particular, it means that in our situation there is an equivalence of triangulated categories $\cD^{cblc}(A,d_A)\cong \cD^{cblf}(A,d_A)$.

\begin{thm}\label{thm:triangtopK0genbyPi}
Let $\{P_i\}_{i \in I}$ be a full collection of  distinct, indecomposable relatively projective $(A,d_A)$-modules.
The asymptotic Grothendieck group is a free $\bZ\pp{x_1, \dots, x_\ell}$-module
\begin{align*}
\bKO^\Delta(\cD^{cblf}(A,d_A))  & \cong  \bigoplus_{i \in I} \bZ\pp{x_1, \dots, x_\ell} \cdot [P_i],
\end{align*}
with basis given by $\{[P_i]\}_{i \in I}$. Moreover, $\bKO^\Delta(\cD^{cblf}(A,d_A))$ is also freely generated by the classes of distinct simple modules $\{[S_i]\}_{i \in I}$.
\end{thm}

\begin{proof}
We have by \cref{cor:GORmodlf} that
\[
\bGO(H^0(A,d_A)\modlf) \cong \bZ\pp{x_1, \dots, x_\ell} \cdot [S_i],
\]
and thus
\[
\bKO^\Delta(\cD^{cblc}(A,d_A))  \cong \bZ\pp{x_1, \dots, x_\ell} \cdot [S_i],
\]
by \cref{thm:eqasympKOheart}. 
Moreover, there is clearly a surjection
\[
  \bigoplus_{i \in I} \bZ\pp{x_1, \dots, x_\ell} \cdot [P_i]  \rightarrow \bKO^{\Delta}(\cD^{cblc}(A,d_A)),
\]
since objects in $\cD^{cblc}(A,d_A)$ are given by c.b.l.f. cell modules. Since we have
\[
[P_i] = f(\bx) [S_i] \in \bKO^\Delta(\cD^{cblc}(A,d_A)),
\]
where $f(0) = 1$, we conclude that 
\[
\bKO^\Delta(\cD^{cblc}(A,d_A))  \cong  \bigoplus_{i \in I} \bZ\pp{x_1, \dots, x_\ell} \cdot [P_i].
\]
This ends the proof since $\cD^{cblc}(A,d_A) \cong \cD^{cblf}(A,d_A)$. 
\end{proof}

\subsection{Baric structure on $\cD^{cblc}(A,d_A)$}

Let $\cD^{cblc}(A,d_A)_{\preceq 0} \subset \cD^{cblc}(A,d_A)$ be the full subcategory given by objects quasi-isomorphic to a c.b.l.f. iterated extension of $\{P_i\}_i$ with all $\bZ^n$-degree $\bg$ multiplicities being zero for $\bg \succ 0$. Similarly,  $\cD^{cblc}(A,d_A)_{\succ 0}$ is given by the c.b.l.f. iterared extensions  in $\bZ^n$-degrees bigger than zero.  By \cref{prop:orderedfiltP}, this defines a baric structure on $\cD^{cblc}(A,d_A)$, where the truncation functor $\beta_{\preceq \bg}$ simply truncates the filtration up to $F_{r+1}$ such that $\bc_r + \be = \bg$. Note that we have $\Hom(X,Y) = 0$ whenever $X \in \cD^{cblc}(A,d_A)_{\preceq 0}$ and $Y \in \cD^{cblc}(A,d_A)_{\succ 0}$ since $\Hom(P_i, P'_j) = 0$ whenever $\deg_\bg(P_i) < \deg_{\bg}(P'_j)$. For the remaining of the section, we will write this baric structure as $\beta$. 
It is non-degenerate, locally finite, cone bounded and full. 

\begin{lem}
The baric structure $\beta$ on $\cD^{cblc}(A,d_A)$ is stable.
\end{lem}

\begin{proof}
Consider a Milnor colimit $X := \mcolim\bigl(X_1 \rightarrow X_2 \rightarrow \cdots)$ with all elements $X_r \in \cD^{cblc}(A,d_A)_{\succ 0}$. We can assume each $X_r$ is a c.b.l.f. cell module. Thus, $\coprod_r X_r$ is isomorphic to a dg-module concentrated in degrees $\succ \bg$. Then so is $X := \cone(\coprod_r X_r \rightarrow \coprod X_r)$, and $X \in  \cD^{cblc}(A,d_A)_{\succ 0}$. It is similar in the case of a Milnor limit. In particular, $\cD^{cblc}(A,d_A)_{\succ 0}$  is stable under acceptable Milnor colimits and acceptable Milnor limits. We conclude by applying \cref{prop:stabilitycond}. 
\end{proof}

\begin{cor}\label{cor:isobaricdgclbf}
There is an isomorphism
\[
\bKO^\Delta(\cD^{cblc}(A,d_A),\beta) \cong \bKO^\Delta(\cD^{cblc}(A,d_A)).
\]
\end{cor}

\begin{proof}
We apply \cref{prop:baricisoKO}.
\end{proof}

\begin{prop}\label{prop:derivedtensorinduces}
Let $(A,d_A)$ and $(A',d_{A'})$ be two c.b.l.f. positive dg-algebras. Let $B$ be a c.b.l.f. dimensional $(A',d_{A'})$-$(A,d_A)$-bimodule. The derived tensor product functor
\[
F : \cD^{cblf}(A,d_A) \rightarrow \cD^{cblf}(A',d_{A'}), \quad F(X) := B \Lotimes_{(A,d_A)} X,
\]
induces a continuous map
\[
[F] : \bKO^\Delta(\cD^{cblf}(A,d_A))  \rightarrow \bKO^\Delta(\cD^{cblf}(A',d_{A'})).
\]
\end{prop}

\begin{proof}
Recall that
\[
B \Lotimes_{(A,d_A)} X := B \otimes_{(A,d_A)} \br{(X)}, 
\]
where $\br{(X)}$ is a property (P) replacement of $X$. 
By \cref{thm:cblfdimiscell}, we can assume $\br{(X)}$ is a c.b.l.f. cell module contained in $\cD^{cblc}(A,d_A)_{\succeq \be}$ for some minimal degree $\be \in \bZ^n$. 
Let 
$|F|$ be the minimal degree of $B$ as $\bZ^{n}$-graded $\Bbbk$-vector space.  We claim that $F$ has finite amplitude $|F|$. Indeed, we obtain that $ B \otimes_{(A,d_A)} \br{(X)}$ is a c.b.l.f. dimensional $(A',d_{A'})$-module with minimal degree $\be + |F|$ as $\Bbbk$-vector space. In particular, $\br{(B \otimes_{(A,d_A)} \br{(X)})} \in \cD^{cblc}(A',d_{A'})_{\succeq \be + |F|}$. 
Finally, we apply \cref{prop:K0ASinducedmap}. 
\end{proof}

\begin{exe}
Recall the $\bZ^2$-graded dg-algebra $(R,0)$ from \cref{sec:mainexintro}. It is a c.b.l.f. positive dg-algebra, and thus we obtain
\[
\bKO^\Delta(\cD^{cblf}(R,0)) \cong \bZ\pp{q,\lambda} \cdot [(R,0)].
\] 
Furthermore, $\bKO^\Delta(\cD^{cblf}(R,0))$ is also freely generated by $[(L,0)]$, and the relations in \cref{eq:changeofbasisintro} hold inside of $\bKO^\Delta(\cD^{cblf}(R,0))$, by interpreting \cref{eq:filtintroex} and \cref{eq:resintroex} as Milnor limits. This achieves our goal from \cref{sec:mainexintro}. 
\end{exe}


%% file: bibliography/bibliography.tex

\bibliographystyle{bibliography/habbrv}

\input{bibliography/topK0.bbl}

%% file: topK0.bbl
\begin{thebibliography}{10}

\bibitem{acharstroppel}
P.~Achar and C.~Stroppel.
\newblock Completions of {G}rothendieck groups.
\newblock {\em Bull. Lond. Math. Soc.}, 45(1):200--212, 2013.

\bibitem{baric}
P.~Achar and D.~Treumann.
\newblock Baric structures on triangulated categories and coherent sheaves.
\newblock {\em Int. Math. Res. Not.}, pages 3688--3743, 2011.

\bibitem{laurent}
A.~Aparicio-Monforte and M.~Kauers.
\newblock Formal {L}aurent series in several variables.
\newblock {\em Expo. Math.}, 31(4):350--367, 2013.

\bibitem{perversesheaves}
A.~Beilinson, J.~Bernstein, and P.~Deligne.
\newblock Faisceaux pervers.
\newblock In {\em Analyse et topologie sur les espaces singuliers, I (Luminy,
  1981)}, volume 100, pages 5--171. Soc. Math. France, Paris, 1982.

\bibitem{mixedcat}
A.~Beilinson, V.~Ginzburg, and W.~Soergel.
\newblock Koszul duality patterns in representation theory.
\newblock {\em J. Amer. Math. Soc.}, 9(2):473--527, 1996.

\bibitem{bernsteinlunts}
J.~Bernstein and V.~Lunts.
\newblock {\em Equivariant sheaves and functors}, volume 1578 of {\em Lecture
  Notes in Mathematics}.
\newblock Springer-Verlag, Berlin, 1994.

\bibitem{bokstedtneeman}
M.~B\"okstedt and A.~Neeman.
\newblock Homotopy limits in triangulated categories.
\newblock {\em Compositio Math.}, 86(2):209--234, 1993.

\bibitem{chuangrouquier}
J.~Chuang and R.~Rouquier.
\newblock Derived equivalences for symmetric groups and
  $\mathfrak{sl}_2$-categorification.
\newblock {\em Ann. of Math.}, 167(2):245--298, 2008.

\bibitem{tensorvermas}
B.~Dupont and G.~Naisse.
\newblock {Categorification of infinite-dimensional $\mathfrak{sl}_2$-modules
  and braid group 2-actions}.
\newblock (in preparation).

\bibitem{FKS}
I.~Frenkel, M.~Khovanov, and C.~Stroppel.
\newblock A categorification of finite-dimensional irreducible representations
  of quantum $\mathfrak{sl}_2$ and their tensor products.
\newblock {\em Selecta Math.}, 12(3-4):379--431, 2006.

\bibitem{frenkelstroppelsussan}
I.~Frenkel, C.~Stroppel, and J.~Sussan.
\newblock Categorifying fractional {E}uler characteristics, {J}ones--{W}enzl
  projectors and $3j$-symbols.
\newblock {\em Quantum Topol.}, 3(2):181--253, 2012.

\bibitem{grothendieck}
A.~Grothendieck.
\newblock Sur quelques points d'alg\`ebre homologique.
\newblock {\em T\^ohoku Math. J.}, 9(2):119--221, 1957.

\bibitem{keller}
B.~Keller.
\newblock Deriving {DG} categories.
\newblock {\em Ann. Sci. \'Ecole Norm. Sup.}, 27(1):63--102, 1994.

\bibitem{keller2}
B.~Keller.
\newblock On differential graded categories.
\newblock 2006, math.KT/0601185.

\bibitem{kellernicolas}
B.~Keller and P.~Nicol\`as.
\newblock Weight structures and simple dg modules for positive dg algebras.
\newblock {\em Int. Math. Res. Not.}, (5):1028--1078, 2013.

\bibitem{kelleryang}
B.~Keller and D.~Yang.
\newblock Derived equivalences from mutations of quivers with potential.
\newblock {\em Adv. Math.}, 226(3):2118--2168, 2011.

\bibitem{khovanov}
M.~Khovanov.
\newblock A categorification of the jones polynomial.
\newblock {\em Duke Math. J.}, 101(3):359--426, 2000.

\bibitem{KRhomology}
M.~Khovanov and L.~Rozansky.
\newblock {Matrix factorizations and link homology II}.
\newblock {\em Geom. Topol.}, 12:1387--1425, 2008.

\bibitem{kochman}
S.~O. Kochman.
\newblock {\em Bordism, stable homotopy and {A}dams spectral sequences},
  volume~7 of {\em Fields Institute Monographs}.
\newblock American Mathematical Society, Providence, RI, 1996.

\bibitem{twoblob}
A.~Lacabanne, G.~Naisse, and P.~Vaz.
\newblock {Tensor product categorifications, Verma modules and the blob
  2-category}.
\newblock 2020, 2005.06257.

\bibitem{lauda}
A.~Lauda.
\newblock A categorification of quantum sl(2).
\newblock {\em Adv. Math.}, 225(6):3327--3424, 2010.

\bibitem{may}
J.~P. May.
\newblock The additivity of traces in triangulated categories.
\newblock {\em Adv. Math.}, 163(1):34--73, 2001.

\bibitem{naissevaz3}
G.~Naisse and P.~Vaz.
\newblock 2-verma modules.
\newblock 2017, 1710.06293.

\bibitem{naissevaztriply}
G.~Naisse and P.~Vaz.
\newblock {2-Verma modules and the Khovanov-Rozansky link homologies}.
\newblock 2017, 1704.08485.

\bibitem{naissevaz1}
G.~Naisse and P.~Vaz.
\newblock An approach to categorification of {V}erma modules.
\newblock {\em Proc. Lond. Math. Soc. (3)}, 117(6):1181--1241, 2018.

\bibitem{naissevaz2}
G.~Naisse and P.~Vaz.
\newblock On 2-{V}erma modules for quantum $\mathfrak{sl}_2$.
\newblock {\em Selecta Math. (N.S.)}, 24(4):3763--3821, 2018.

\bibitem{neeman92}
A.~Neeman.
\newblock The connection between the {K}-theory localization theorem of
  {T}homason, {T}robaugh and {Y}ao and the smashing subcategories of
  {B}ousfield and {R}avenel.
\newblock {\em Ann. Sci. \'Ecole Norm. Sup.}, 25(5):547–566, 1992.

\bibitem{neeman96}
A.~Neeman.
\newblock The {G}rothendieck duality theorem via {B}ousfield's techniques and
  {B}rown representability.
\newblock {\em J. Amer. Math. Soc.}, 9(1):205--236, 1996.

\bibitem{neemancounterex}
A.~Neeman.
\newblock A counterexample to a 1961 ``theorem'' in homological algebra.
\newblock {\em Invent. Math.}, 148(2):397--420, 2002.

\bibitem{parrathesis}
C.~E. Parra.
\newblock {\em Hearts of t-structures which are {G}rothendieck or module
  categories}.
\newblock PhD thesis, Universidad de Murcia, 2014.

\bibitem{qithesis}
Y.~Qi.
\newblock {\em Hopfological {A}lgebra}.
\newblock PhD thesis, Columbia University, 2013.

\bibitem{quillen}
D.~Quillen.
\newblock {\em Higher algebraic K-theory I}, volume 341 of {\em Lecture Notes
  in Math.}
\newblock Springer, Berlin, 1973.

\bibitem{ravenel}
D.~C. Ravenel.
\newblock Localization with respect to certain periodic homology theories.
\newblock {\em Amer. J. Math.}, 106(2):351--414, 1984.

\bibitem{soergel}
W.~Soergel.
\newblock Kazhdan--{L}usztig polynomials and indecomposable bimodules over
  polynomial rings.
\newblock {\em J. Inst. Math. Jussieu}, 6:501-- 525, 2007.

\bibitem{stacks-project}
{The Stacks project authors}.
\newblock The stacks project.
\newblock \url{https://stacks.math.columbia.edu}, 2019.

\bibitem{verdier}
J.-L. Verdier.
\newblock {\em Des cat\'egories d\'eriv\'ees des cat\'egories ab\'eliennes}.
\newblock Number 239. Ast\'erisque, 1996.
\newblock With a preface by Luc Illusie. Edited and with a note by Georges
  Maltsiniotis.

\bibitem{krullschmidt}
C.~Walker and R.~B. Warfield.
\newblock Unique decomposition and isomorphic refinement theorems in additive
  categories.
\newblock {\em J. Pure Appl. Algebra}, 7(3):347--359, 1976.

\bibitem{catgeom}
B.~Webster.
\newblock Geometry and categorification.
\newblock In A.~Beliakova and A.~D. Lauda, editors, {\em Categorification in
  {G}eometry, {T}opology, and {P}hysics}, volume 684 of {\em Contemporary
  Mathematics}, pages 1--22. Amer. Math. Soc., Providence, 2017.

\bibitem{webster}
B.~Webster.
\newblock Knot invariants and higher representation theory.
\newblock {\em Mem. Amer. Math. Soc.}, 250(1191):v+141 pp, 2017.

\bibitem{weibel}
C.~A. Weibel.
\newblock {\em An introduction to homological algebra}, volume~38 of {\em
  Cambridge Studies in Advanced Mathematics}.
\newblock Cambridge University Press, Cambridge, 1994.

\end{thebibliography}
